\documentclass[11pt, a4paper]{article}  
\pdfoutput=1
\usepackage{hyperref}
\usepackage[numbers]{natbib}
\usepackage{algpseudocode}
\usepackage{amsthm, bm}
\usepackage{algorithm}
\usepackage{pifont}
\usepackage{dsfont}
\usepackage{amsmath}
\usepackage[nameinlink]{cleveref}
\usepackage{amssymb}

\usepackage{color}
\usepackage{csvsimple}
\usepackage{enumerate, enumitem}
\usepackage{mathrsfs} 
\usepackage{titlesec}
\usepackage{multicol, multirow, tabularx}
\usepackage{mathtools, calc}
\usepackage{verbatim}
\usepackage{listings}
\usepackage{graphicx}
\usepackage{lscape, textcomp}
\usepackage{tikz, subcaption}
\usepackage{xparse}
\usepackage{url}
\usepackage[tableposition=top]{caption}
\usepackage[toc,page]{appendix}
\usepackage[british]{babel}
\usepackage{tkz-euclide}
\usetikzlibrary{arrows}
\usepackage{fancyhdr} 
\usepackage{array, stackengine}
\usepackage{pbox}
\usepackage{sectsty}
\usepackage{tablefootnote, longtable}
\usepackage{epstopdf}
\sectionfont{\fontsize{12}{15}\selectfont}
\subsectionfont{\fontsize{11}{15}\selectfont}
\subsubsectionfont{\fontsize{11}{15}\selectfont}

\newcolumntype{L}[1]{>{\raggedright\let\newline\\\arraybackslash\hspace{0pt}}m{#1}}
\newcolumntype{C}[1]{>{\centering\let\newline\\\arraybackslash\hspace{0pt}}m{#1}}
\newcolumntype{R}[1]{>{\raggedleft\let\newline\\\arraybackslash\hspace{0pt}}m{#1}}

\definecolor{blue}{rgb}{0.0, 0.0,1.0}
\definecolor{niceblue}{rgb}{0, 0.5, 0.95}

\newcommand{\mtx}[1]{\boldsymbol{#1}}
\newcommand{\mvec}[1]{\boldsymbol{#1}}
\newcommand{\wt}[1]{\widetilde{#1}}

\def \est {\mathrm{est}}
\def \eff {\mathrm{eff}}

\theoremstyle{plain}
\newtheorem{theorem}{Theorem}[section]
\newtheorem{lemma}[theorem]{Lemma}
\newtheorem{proposition}[theorem]{Proposition}
\newtheorem{corollary}[theorem]{Corollary}

\newtheorem{assump}[theorem]{Assumption}
	
\theoremstyle{definition}
\newtheorem{example}[theorem]{Example}
\newtheorem{definition}[theorem]{Definition}
\newtheorem*{rem*}{Remark}
\newtheorem*{warning*}{Warning}
\newtheorem{rem}[theorem]{Remark}

\newtheorem{innerassumpA}{Assumption}
\newenvironment{assumpA}[1]
  {\renewcommand\theinnerassumpA{#1}\innerassumpA}
  {\endinnerassumpA}

\DeclareMathOperator{\range}{range}
\DeclareMathOperator{\prob}{\mathbb{P}}
\DeclareMathOperator{\vol}{Vol}

\DeclareMathOperator{\ccone}{Circ}

\DeclareMathOperator{\rint}{relint}
\DeclareMathOperator{\aff}{aff}

\DeclarePairedDelimiter{\ceil}{\lceil}{\rceil}

\def \R {\mathbb{R}}

\newsavebox\CBox

\makeatletter
\newcommand{\subalign}[2][c]{%
	\if#1c\vcenter\else\vtop\fi{%
		\Let@ \restore@math@cr \default@tag
		\baselineskip\fontdimen10 \scriptfont\tw@
		\advance\baselineskip\fontdimen12 \scriptfont\tw@
		\lineskip\thr@@\fontdimen8 \scriptfont\thr@@
		\lineskiplimit\lineskip
		\ialign{\hfil$\m@th\scriptstyle##$&$\m@th\scriptstyle{}##$\hfil\crcr
			#2\crcr
		}%
	}%
}
\makeatother

\makeatletter
\renewcommand*{\@fnsymbol}[1]{\ensuremath{\ifcase#1\or *\or \ddagger\or \mathsection\or \vee\or \wedge\or \dagger\or
		\mathsection\or \mathparagraph\or \|\or **\or \dagger\dagger
		\or \ddagger\ddagger \else\@ctrerr\fi}}

\numberwithin{equation}{section}

\begin{document}
\title{Global optimization using random embeddings}

\author{
	Coralia Cartis \thanks{The Alan Turing Institute, The British Library, London, NW1 2DB, UK. This work was supported by The Alan Turing Institute under The Engineering and Physical Sciences Research Council (EPSRC) grant EP/N510129/1 and under the Turing project scheme.} \textsuperscript{\normalfont,}\thanks{Mathematical Institute, University of Oxford, Radcliffe Observatory Quarter, Woodstock Road,
		Oxford, OX2 6GG, UK; \texttt{cartis,massart,otemissov@maths.ox.ac.uk}}
	\and Estelle Massart \thanks{National Physical Laboratory, Hampton Road, Teddington, Middlesex, TW11 0LW, UK. This author's work was supported by the National Physical Laboratory.} \textsuperscript{\normalfont,}\footnotemark[2] 
	\and
	Adilet Otemissov \footnotemark[1] \textsuperscript{\normalfont,}\footnotemark[2] 
}

\date{\today}
\maketitle
\footnotesep=0.4cm

{\small
	\begin{abstract}
	We propose a random-subspace algorithmic framework for global optimization of Lipschitz-continuous objectives, and analyse its convergence using novel tools from conic integral geometry.  X-REGO randomly projects, in a  sequential or simultaneous manner, the high-dimensional original problem  into low-dimensional  subproblems that can then be solved with any global, or even local, optimization solver. We estimate the probability that the randomly-embedded subproblem shares (approximately) the same global optimum as the original problem. This success probability is then used to show convergence of X-REGO to an approximate global solution of the original problem, under weak assumptions on the problem (having a strictly feasible global solution) and on the  solver (guaranteed to find an approximate global solution of the reduced problem with sufficiently high probability). 
    In the particular case of unconstrained objectives with low effective dimension, that only vary over a low-dimensional subspace, we propose an X-REGO variant that explores random subspaces of increasing dimension until finding the effective dimension of the problem, leading to  X-REGO globally converging after a finite number of embeddings, proportional to the effective dimension. We show numerically that this variant efficiently finds both the effective dimension and an approximate global minimizer of the original problem. 
	\end{abstract}
	
	\bigskip
	
	\begin{center}
	\end{center}
}

\maketitle

\section{Introduction} \label{sec:intro}

We address the global optimization problem
\begin{equation}
\tag{P}
\begin{aligned} \label{eq: GO}
f^* := \min_{\mvec{x}\in \mathcal{X}} & \;\; f(\mvec{x}), \\
\end{aligned}
\end{equation}
where $f: \mathcal{X} \rightarrow \mathbb{R}$ is Lipschitz continuous and possibly non-convex, and where $\mathcal{X}$ is a set with non-empty interior, and possibly unbounded, which thus includes the unconstrained case $\mathcal{X} = \mathbb{R}^D$. We propose a generic algorithmic framework, named X-REGO ($\mathcal{X}$-Random Embeddings for Global Optimization) that  (approximately) solves a sequence of realizations of the following randomized reduced problem,
\begin{equation} \label{eq: AREGO}
\tag{RP$\mathcal{X}$}
\begin{aligned}
\min_{\mvec{y}}  &\;\; f(\mtx{A}\mvec{y}+\mvec{p})  \\
\text{subject to} &\;\; \mtx{A}\mvec{y} + \mvec{p} \in \mathcal{X},
\end{aligned}
\end{equation}
where $\mtx{A}$ is a $D\times d$ Gaussian random matrix (see \Cref{def: Gaussian_matrix}) with $d \ll D$,
 and where $\mvec{p} \in \mathcal{X}$ may vary between realizations, may be arbitrary/user-defined, and provides additional flexibility that can be exploited algorithmically. The reduced problem \eqref{eq: AREGO} can be solved by any global, or even local or stochastic, optimization solver.

 When a (possibly stochastic) global solver is used in the subproblems, we prove that X-REGO converges, with probability one, to a global $\epsilon$-minimizer of \eqref{eq: GO} (namely, a feasible point $\mvec{x}$ satisfying $f(\mvec{x}) \leq f^* + \epsilon$ for some accuracy $\epsilon>0$);   we also provide estimates of the corresponding convergence rate. For this, we need to evaluate  the $\epsilon$-\emph{success} of the reduced problem \eqref{eq: AREGO}.
\begin{definition} \label{def: success_red_prob}
\eqref{eq: AREGO} is $\epsilon$-\emph{successful} if there exists $\mvec{y} \in \mathbb{R}^d$ such that $\mtx{A} \mvec{y} + \mvec{p} \in \mathcal{X}$ and  $f(\mtx{A} \mvec{y} + \mvec{p}) \leq f^* + \epsilon$, where $\epsilon>0$ is the desired/user-chosen accuracy tolerance.
\end{definition}
Equivalently, this success probability can be rephrased as follows.
\begin{center}
    \emph{What are the chances that a random low-dimensional subspace spanned by the columns of a (rectangular) Gaussian matrix contains a global $\epsilon$-minimizer of \eqref{eq: GO}?}
\end{center}
We use crucial tools from conic integral geometry to estimate the probability above. Applications of these bounds to functions with low effective dimensionality are also provided.

\subsection{Related work.}

Dimensionality reduction is essential to the efficient solution of high-dimensional optimization problems.  Sketching techniques reduce the ambient dimension of a given subspace by projecting it randomly into a lower dimensional one while preserving lengths \cite{woodruff2014}; such techniques have been used successfully for improving the efficiency of linear and nonlinear least squares (local) solvers and of those for more general sums of functions; see for example, \cite{pilanci2017,roosta2019,berahas2020, Cartis2021hash} and the references therein. Here, we sketch the problem variables/search space in order to reduce its dimension for the specific aim of global optimization;  furthermore, our results are not derived using sketching techniques but conic integral geometry ones.

 In a huge-scale setting, where full-dimensional vector operations are computationally expensive, \citeauthor{Nesterov12} \cite{Nesterov12} advocates the use of coordinate descent, a local optimization method that updates successively one of the coordinates of a candidate solution using a coordinate-wise variant of a first-order method, while keeping other coordinates fixed. Coordinate descent methods and their block counterparts have become a method of choice for many large-scale applications, see, e.g., \cite{Bach2011,Richtarik2015,Wright2015} and have been extended to random subspace descent \cite{Lacotte2019,kozak2019} that operates over a succession of random low-dimensional subspaces, not necessarily aligned with coordinate axes. See also \cite{Grishchenko2021} for a random proximal subspace descent algorithm, and \cite{gower2019, hanzely2020} for higher-order random subspace methods for local nonlinear optimization.

In local derivative-free optimization, several algorithms explore successively one-dimensional \cite{Stich2013, Nesterov2017, Bergou2020} and low-dimensional \cite{Cartis2021} random subspaces. \citeauthor{Gratton2015} \cite{Gratton2015,Gratton2019b} propose and explore a randomized version of direct search where at each iteration the function is explored along a collection of directions, i.e., one-dimensional half-spaces. \citeauthor{Golovin2020} \cite{Golovin2020} develop convergence rates to a ball of $\epsilon$-minimizers for a variant of randomized direct search for a special class of quasi-convex objectives. Their convergence analysis heavily relies on high-dimensional geometric arguments: they show that sublevel sets contain a sufficiently large ball tangent to the level set, so that at each iteration, with a given probability, sampling the next iterate from a suitable distribution centred at the current iterate decreases the cost.

Unlike the above-mentioned works, our focus here is on the \emph{global} optimization of \emph{generic Lipschitz-continuous objectives}. Stochastic global optimization methods abound, such as simulated annealing \cite{Glover2003}, random search \cite{Solis1981}, multistart methods \cite{Glover2003}, and genetic algorithms \cite{Holland1973}.  Our proposal here is connected to random search methods, namely, it can be viewed as a multi-dimensional random search, where a deterministic or stochastic method is applied to the subspace minimization. Recently, random subspace methods  have been developed/applied for the global optimization of objectives with special structure, assuming typically, low-effective dimensionality of the objective \cite{Wang2016, Binois2014, Binois2017, Kirschner19, Cartis2020, Cartis2020b, QianHuYu2016}. These functions only vary over a low-dimensional subspace, and are also called multi-ridge functions \cite{Fornasier2012,Tyagi2014}, functions with active subspaces \cite{Constantine2015}, or functions with functional sparsity when the subspace of variation is aligned with coordinate axes \cite{Wang2018b}. 
Assuming the random subspace dimension $d$ (in \eqref{eq: AREGO}) to be an overestimate of the objective's effective dimension $d_e$ (the dimension of the subspace of variation), these works have proven that one random embedding is sufficient with probability one to solve the original problem \eqref{eq: GO} in the unconstrained case ($\mathcal{X} = \mathbb{R}^d)$ \cite{Wang2016,Cartis2020} while several random embeddings are required in the constrained case \cite{Cartis2020b}. 
In particular, in \cite{Cartis2020b}, we propose an X-REGO variant that is designed specifically for the bound-constrained optimization of functions with low effective dimensionality. As such it keeps the random subspace dimension $d$ in \eqref{eq: AREGO} fixed and greater than the effective dimension which is assumed to be known. Here, X-REGO is designed and analysed for a generic objective and a possibly unbounded/unconstrained and nonconvex domain $\mathcal{X}$, and the random subspace dimension $d$ is arbitrary and allowed to vary during the optimization.

Recently, random projections have been successfully applied to highly overparametrized settings, such as in deep neural network training \cite{Li2018,Izmailov2019} and adversarial attacks in deep learning \cite{Cai2021,Ughi2021}. Though there is no theoretical guarantee at present that a precise low-dimension subspace exists in these problems, it is a reasonable assumption to make given the high dimensionality of the search space and the supporting numerical evidence. Our approach here investigates the validity of random subspace methods when low effective dimensionality is absent or unknown to the user; we find - both theoretically and numerically - that for large scale problems, such techniques are still beneficial, and furthermore, at least in the unconstrained case, they can naturally adapt and capture such special structures efficiently. We hope that this provides a general theoretical  justification to a broader application  of such techniques.

The second part of the paper applies the generic X-REGO convergence results and the \eqref{eq: AREGO} related probabilistic bounds to the case when the objective is unconstrained and has low effective dimension, but the effective dimension $d_e$ is unknown. Related results have  been proposed that aim to learn the effective subspace before \cite{Fornasier2012,Djolonga2013,Tyagi2014,Eriksson2018} or during the optimization process \cite{Garnett2014,Zhang2019,Chen2020,Demo2020};  additional costs/evaluations are needed in these approaches. Some apply a principal component analysis (PCA) to the gradient evaluated at a collection of random points \cite{Constantine2015, Eriksson2018, Demo2020}. Alternatively, \cite{Fornasier2012,Djolonga2013,Tyagi2014} recast the problem into a low-rank matrix recovery problem, and \cite{Garnett2014} proposes a Bayesian optimization algorithm that sequentially updates a posterior distribution over effective subspaces, and over the objective, using new functions evaluations. Still in the context of Bayesian optimization, \citeauthor{Zhang2019} \cite{Zhang2019} estimate the effective subspace using Sliced Inverse Regression, a supervised dimensionality reduction technique in contrast with the above-mentioned PCA, while \citeauthor{Chen2020} \cite{Chen2020} extend Sliced Inverse Regression to learn the effective subspace in a semi-supervised way. Instead, our proposed algorithm explores a sequence of random subspaces of increasing dimension until it discovers the effective dimension of the problem. Independently,  a similar idea has been recently used in sketching methods for regularized least-squares optimization \cite{Lacotte2020}. 

\paragraph{Our contributions.}

We explore the use of random embeddings for the generic global optimization problem \eqref{eq: GO}. Our proposed algorithmic framework, X-REGO, replaces \eqref{eq: GO} by a sequence of reduced random subproblems \eqref{eq: AREGO}, that are solved (possibly approximately and  probabilistically) using any global optimization solver. As such, X-REGO extends block coordinate descent and local random subspace methods to the global setting.

Our convergence analysis for X-REGO crucially relies on a lower bound on the probability of $\epsilon$-success of \eqref{eq: AREGO}, whose computation, exploiting connections between \eqref{eq: AREGO} and the field of conic integral geometry, is a key contribution of this paper\footnote{Note that this is not the first work applying conic integral geometry to optimization, see \cite{Amelunxen2011} for an application to the study of phase transitions in random convex optimization problems.}. Using asymptotic expansions of integrals, we derive interpretable lower bounds in the setting where the random subspace dimension $d$ is fixed and the original dimension $D$ grows to infinity.  In the box-constrained case $\mathcal{X} = [-1,1]^D$, we also compare  these bounds with the probability of success of the simplest random search strategy, where a point is sampled in the domain uniformly at random at each iteration. We show that when the point $\mvec{p}$ at which the random subspace is drawn is close enough to a global solution $\mvec{x}^*$ of \eqref{eq: GO}, the random subspace is more likely to intersect a ball of $\epsilon$-minimizer than finding an $\epsilon$-minimizer using random search. Provided that the reduced problem can be solved at a reasonable cost, random subspace methods are thus provably better than random search in some cases; and even more so, numerically.

In the second part of the paper, we address global optimization of functions with low effective dimension, and propose an X-REGO variant that progressively increases the random subspace dimension.  Instead of requiring a priori knowledge of the effective dimension of the objective, we show numerically that this variant is able to \emph{learn} the effective dimension of the problem. We also provide convergence results for this variant after a finite number of embeddings, using again our conic integral geometry bounds. Noticeably, these convergence results have no dependency on $D$. We compare numerically several instances of X-REGO when the reduced problem is solved using the (global and local) KNITRO solver \cite{Byrd2006}. We also discuss several strategies to choose the parameter $\mvec{p}$ in \eqref{eq: AREGO}.

\paragraph{Paper outline.}
\Cref{sec: prelim} presents the geometry of the problem, and motivates the use of conic integral geometry to estimate the probability of \eqref{eq: AREGO} being $\epsilon$-successful. \Cref{sec: Conic_integral_geomtry} summarizes key results from conic integral geometry that are used later in the paper. In \Cref{sec:est_prob_eps_success_RPX_gen_fun}, we derive lower bounds on the probability of \eqref{eq: AREGO} to be $\epsilon$-successful, obtain asymptotic expansions of this probability, and compare the search within random embeddings with random search. \Cref{sec: XREGO} presents the X-REGO algorithmic framework, and \Cref{sec:convergence} the corresponding convergence analysis. Finally, \Cref{sec: loweffdim} proposes a specific instance of X-REGO for global optimization of functions with low effective dimension, with associate convergence results, and \Cref{sec: Numerics} contains numerical illustrations.

\paragraph{Notation.}
We use bold capital letters for matrices ($\mtx{A}$) and bold lowercase letters ($\mvec{a}$) for vectors. In particular, $\mtx{I}_D$ is the $D \times D$ identity matrix and $\mvec{0}_D$, $\mvec{1}_D$ (or simply $\mvec{0}$, $\mvec{1}$) are the $D$-dimensional vectors of zeros and ones, respectively. We write $a_i$ to denote the $i$th entry of $\mvec{a}$ and write $\mvec{a}_{i:j}$, $i<j$, for the vector $(a_i \; a_{i+1} \cdots a_{j})^T$. We let  $\range(\mtx{A})$ denote the linear subspace spanned in $\mathbb{R}^D$ by the columns of $\mtx{A} \in \mathbb{R}^{D \times d}$. We write $\langle \cdot , \cdot \rangle$, $\| \cdot \|$ (or equivalently $\| \cdot \|_2$) for the usual Euclidean inner product and Euclidean norm, respectively.

Given two random variables (vectors) $x$ and $y$ ($\mvec{x}$ and $\mvec{y}$), the expression $x \stackrel{law}{=} y$ ($\mvec{x} \stackrel{law}{=} \mvec{y}$) means that $x$ and $y$ ($\mvec{x}$ and $\mvec{y}$) have the same distribution. We reserve the letter $\mtx{A}$ for a $D\times d$ Gaussian random matrix (see \Cref{def: Gaussian_matrix}). 

Given a point $\mvec{a} \in \mathbb{R}^D$ and a set $S$ of points in $\mathbb{R}^D$, we write $\mvec{a}+S$ to denote the set $\{ \mvec{a}+\mvec{s}: \mvec{s} \in S \}$. Given functions $f(x):\mathbb{R}\rightarrow \mathbb{R}$ and $g(x):\mathbb{R}\rightarrow \mathbb{R}^+$, we write $f(x) = \Theta(g(x))$ as $x \rightarrow \infty$ to denote the fact that there exist positive reals $M_1,M_2$ and a real number $x_0$ such that, for all $x \geq x_0$, $M_1g(x)\leq|f(x)| \leq M_2g(x)$.

\section{Geometric description of the problem} \label{sec: prelim}

Let $\epsilon >0$ denote the accuracy to which problem (P) is to be solved, and so let $G_\epsilon$ be the set of $\epsilon$-minimizers of \eqref{eq: GO},
\begin{equation} \label{eq: G_epsilon}
    G_\epsilon = \{ \mvec{x} \in \mathcal{X} : f(\mvec{x}) \leq f^* + \epsilon\}.
\end{equation}  
Note that, by \Cref{def: success_red_prob}, the reduced problem \eqref{eq: AREGO} is $\epsilon$-successful if and only if the intersection of the (affine) subspace $\mvec{p}+\range(\mtx{A})$ and $G_{\epsilon}$ is non-empty:
\begin{equation}\label{eq:RPX_eps_succ=subspace_cap_G_eps}
	\prob[\eqref{eq: AREGO} \ \text{is $\epsilon$-successful}] = \prob[\mvec{p}+\range(\mtx{A}) \cap G_{\epsilon} \neq \varnothing].
\end{equation}

To further characterize this probability, let us now introduce the following assumptions.
\begin{assumpA}{LipC}[Lipschitz continuity of $f$] \label{ass:f_is_Lipschitz}
The objective function $f : \mathcal{X} \to \mathbb{R}$ is Lipschitz continuous with constant $L$, i.e., there holds $|f(\mvec{x})-f(\mvec{y})| \leq L \rVert \mvec{x}-\mvec{y} \rVert_2$ for all $\mvec{x}, \mvec{y} \in \mathcal{X}$.
\end{assumpA}

\begin{assumpA}{FeasBall}[Existence of a ball of $\epsilon$-minimizers]\label{ass:B_eps_is_a_ball} 
	There exists a global minimizer $\mvec{x^*}$ of \eqref{eq: GO} that satisfies $B_{\epsilon/L}(\mvec{x}^*) \subset \mathcal{X}$, where $B_{\epsilon/L}(\mvec{x}^*)$ is the D-dimensional Euclidean ball of radius $\epsilon/L$ and centered at $\mvec{x}^*$, where $L$ is the Lipschitz constant of $f$ and $\epsilon>0$ is the desired  accuracy tolerance.
\end{assumpA}

We then have the following result.
\begin{proposition} \label{prop:lower_bound_ball}
Let \Cref{ass:f_is_Lipschitz} hold. Let $\mtx{A}$ be a $D \times d$ Gaussian matrix, $\epsilon$ a positive accuracy tolerance and $\mvec{x}^*$ any global minimizer of \eqref{eq: GO} satisfying \Cref{ass:B_eps_is_a_ball}. Let $\mvec{p} \in \mathcal{X}$ be a given vector. Then, 
\begin{equation}\label{eq:RPX_eps_succ>subspace_cap_B_eps}
	\prob[\eqref{eq: AREGO} \ \text{is $\epsilon$-successful}] \geq \prob[\mvec{p}+\range(\mtx{A}) \cap B_{\epsilon/L}(\mvec{x}^*) \neq \varnothing].
\end{equation}
\end{proposition}

\begin{proof}
Let $\mvec{x}^*$ be a global minimizer of $f$ in $\mathcal{X}$ satisfying \Cref{ass:B_eps_is_a_ball}, and let $\mvec{x} \in B_{\epsilon/L}(\mvec{x}^*)$. Then, $\mvec{x} \in G_{\epsilon}$ due to the Lipschitz continuity property of $f$, namely
\begin{equation}\label{eq:|f-f^*|<Leps/L}
	|f(\mvec{x}) - f(\mvec{x}^*)| \leq L\| \mvec{x} - \mvec{x}^* \|_2 \leq L \frac{\epsilon}{L} = \epsilon.
\end{equation}
The result follows then simply from \eqref{eq:RPX_eps_succ=subspace_cap_G_eps}.
\end{proof}

In the case of non-unique solutions, each global minimizer $\mvec{x}^*$ of \eqref{eq: GO} satisfying \Cref{ass:B_eps_is_a_ball} provides  a different lower bound in \Cref{prop:lower_bound_ball}. If all the balls $B_{\epsilon/L}(\mvec{x}^*)$ associated with different global minimizers are disjoint, the probability of $\epsilon$-success of \eqref{eq: AREGO} is lower bounded by the sum, over each $\mvec{x}^*$ satisfying \Cref{ass:B_eps_is_a_ball}, of the probability $\prob[ \mvec{p} + \range(\mtx{A}) \cap B_{\epsilon/L}(\mvec{x}^*) \neq \varnothing]$. In this paper,
we estimate the latter probability for an arbitrary $\mvec{x}^*$; this is a worst-case bound in the sense that it
clearly underestimates the chance of subproblem success (for a(ny) $\mvec{x}^*$) in the presence of multiple global minimizers of \eqref{eq: GO}.

Given $\mvec{x}^*$ satisfying \Cref{ass:B_eps_is_a_ball}, let us assume that $\mvec{p} \notin B_{\epsilon/L}(\mvec{x}^*)$ (otherwise, the reduced problem \eqref{eq: AREGO} is always $\epsilon$-successful, which can be seen by simply taking $\mvec{y} = \mvec{0}$). To estimate the right-hand side of \eqref{eq:RPX_eps_succ>subspace_cap_B_eps}, we first construct a set $C_{\mvec{p}}(\mvec{x}^*)$ containing the rays connecting $\mvec{p}$ with points in $B_{\epsilon/L}(\mvec{x}^*)$,
\begin{equation}\label{def:C(p)}
	\text{$C_{\mvec{p}}(\mvec{x}^*) = \{ \mvec{p}+\theta(\mvec{x}-\mvec{p}): \theta \geq 0, \mvec{x} \in B_{\epsilon/L}(\mvec{x}^*) \}$ for $\mvec{p} \notin B_{\epsilon/L}(\mvec{x}^*)$.}
\end{equation}
Note that $C_{\mvec{p}}(\mvec{x}^*)$ is a convex cone that has been translated by $\mvec{p}$ (see \Cref{fig:RPX_through_cones}). We can easily verify this fact by recalling the definition of a convex cone.

\begin{figure}[!t]
	\centering
	\begin{tikzpicture}[scale = 1]
		\fill[brown!40!white] (-1.2,0.7) -- +(26:4.2) arc (26:5.5:4.2);
		\node at(1.3,1.4) [fill=black!30!white,circle,minimum size=9mm,inner sep=0mm] {};
		
		\draw (-2,-2) rectangle (2,2);
		
		\draw[line width = 0.5pt] (-2.8,0.337) -- (3,1.652);
		\draw[line width = 1.5pt, color = red] (-2,0.519) -- (2,1.425);
		
		\draw[fill, color = blue] (1.3, 1.4) circle(1.2pt);
		\draw[fill] (0,0) circle(1.2pt);
		\draw[fill] (-1.2, 0.7) circle(1.2pt);
		
		\node at (0,-2.5) {$\mathcal{X} \subset \mathbb{R}^D$};
		\node at (-1.1,0.3) {$\mvec{p}$};
		\node at (-4.2,0.3) {$\mvec{p}+ \range(\mtx{A})$};
		\node at (1.2, 0.6) {$B_{\epsilon/L}(\mvec{x}^*)$};
		\node at (3.3, 0.7) {$C_{\mvec{p}}(\mvec{x}^*)$};

		\draw[line width = 0.5pt, ->] (0,0) -- (-1.2,0.7);
		\node at (5.8,4) {};
	\end{tikzpicture}
	\caption{Abstract illustration of the embedding of an affine $d$-dimensional subspace $\mvec{p}+\range(\mtx{A})$ into $\mathbb{R}^D$, in the case $\mathcal{X} = [-1,1]^D$. The red line represents the set of solutions along $\mvec{p}+\range(\mtx{A})$ that are contained in $\mathcal{X}$ and the blue dot represents a global minimizer $\mvec{x}^*$ of \eqref{eq: GO}. \eqref{eq: AREGO} is $\epsilon$-successful when the red line intersects $B_{\epsilon/L}(\mvec{x}^*)$. We construct a cone $C_{\mvec{p}}(\mvec{x}^*)$ in such a way that the following condition holds: $\mvec{p}+\range(\mtx{A})$ intersects $B_{\epsilon/L}(\mvec{x}^*)$ if and only if $\mvec{p}+\range(\mtx{A})$ and $C_{\mvec{p}}(\mvec{x}^*)$ share a ray.}
	\label{fig:RPX_through_cones}
\end{figure}
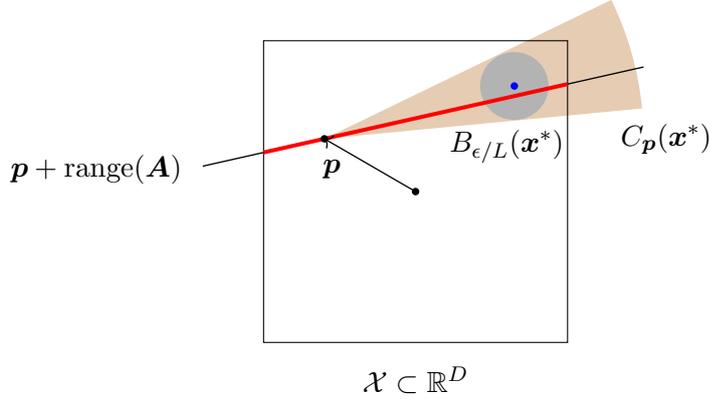

\begin{definition}\label{def:cone}
	A convex set $C$ is called a convex cone if for every $\mvec{c} \in C$ and any non-negative scalar $\rho$, $\rho \mvec{c} \in C$.
\end{definition}
\begin{rem}\label{rem:lin_subspace_is_a_cone}
	Note that, according to \Cref{def:cone}, a $d$-dimensional linear subspace in $\mathbb{R}^D$ is a cone. Hence, $\range(\mtx{A})$ is a cone.
\end{rem}
The next result indicates that, based on \eqref{eq:RPX_eps_succ>subspace_cap_B_eps} and the definition of $C_{\mvec{p}}(\mvec{x}^*)$, we can rewrite the right-hand side of \eqref{eq:RPX_eps_succ>subspace_cap_B_eps} as
\begin{equation}\label{eq:p+range(A)_cap_C_neq_p}
\prob[\mvec{p}+\range(\mtx{A}) \cap B_{\epsilon/L}(\mvec{x}^*) \neq \varnothing]=	\prob[\mvec{p}+\range(\mtx{A}) \cap C_{\mvec{p}}(\mvec{x}^*) \neq \{ \mvec{p} \}]
\end{equation} --- the probability of the event that translated cones $\mvec{p}+\range(\mtx{A})$ and $C_{\mvec{p}}(\mvec{x}^*)$ share a ray. It turns out that this probability has a quantifiable expression based on conic integral geometry, where a broad concern is the quantification/estimation of probabilities of a random cone (e.g., $\mvec{p}+\range(\mtx{A})$) and a fixed cone (e.g., $C_{\mvec{p}}(\mvec{x}^*)$) sharing a ray. We then present in \Cref{sec: Conic_integral_geomtry} key tools from conic integral geometry to help us estimate the probability of $\epsilon$-success of \eqref{eq: AREGO}.

\begin{theorem} \label{thm:RPX_eps_succ>p+range_cap_C(p)}
	Let \Cref{ass:f_is_Lipschitz} hold. Let $\mtx{A}$ be a $D \times d$ Gaussian matrix, $\epsilon$ a positive accuracy tolerance and $\mvec{x}^*$ any global minimizer of \eqref{eq: GO} satisfying \Cref{ass:B_eps_is_a_ball}. Let $\mvec{p} \in \mathcal{X}\setminus G_{\epsilon}$ be a given vector and let  $C_{\mvec{p}}(\mvec{x}^*)$ be defined in \eqref{def:C(p)}. Then,
	\begin{equation} \label{eq:RPX_eps_succ>p+range(A)_cap_C_noteq_p}
		\prob[\eqref{eq: AREGO} \ \text{is $\epsilon$-successful}] \geq \prob[\mvec{p}+\range(\mtx{A}) \cap C_{\mvec{p}}(\mvec{x}^*) \neq \{\mvec{p}\}].
	\end{equation}
\end{theorem}
\begin{proof}
	From \Cref{prop:lower_bound_ball}, we have
	\[\prob[\eqref{eq: AREGO} \ \text{is $\epsilon$-successful}] \geq \prob[\mvec{p}+\range(\mtx{A}) \cap B_{\epsilon/L}(\mvec{x}^*) \neq \varnothing]. \]
	The result follows from the fact that the event $\{\mvec{p}+\range(\mtx{A}) \cap C_{\mvec{p}}(\mvec{x}^*) \neq \{\mvec{p}\}\}$ is a subset of the event $\{ \mvec{p}+\range(\mtx{A}) \cap B_{\epsilon/L}(\mvec{x}^*) \neq \varnothing \}$. We prove this fact below.
	
	Suppose that the event $\{\mvec{p}+\range(\mtx{A}) \cap C_{\mvec{p}}(\mvec{x}^*) \neq \{\mvec{p}\}\}$ occurs. Then, there exists a point $\mvec{x}' \neq \mvec{p}$ in $\mvec{p}+\range(\mtx{A}) \cap C_{\mvec{p}}(\mvec{x}^*) $. Define $R = \{\mvec{p}+ \theta (\mvec{x}'-\mvec{p}) : \theta \geq 0\}$ and note that $ R \subset \mvec{p}+\range(\mtx{A})$. Now, since $\mvec{x}' \in C_{\mvec{p}}(\mvec{x}^*)$, by definition of $C_{\mvec{p}}(\mvec{x}^*)$ there exists $\tilde{\mvec{x}} \in B_{\epsilon/L}(\mvec{x}^*)$ and $\tilde{\theta} > 0$ such that $\mvec{x}' = \mvec{p}+ \tilde{\theta}(\tilde{\mvec{x}}-\mvec{p}) $. We express $\tilde{\mvec{x}}$ in terms of $\mvec{x}'$: $\tilde{\mvec{x}} = \mvec{p}+\theta'(\mvec{x}'-\mvec{p}) $, where $\theta' = 1/\tilde{\theta} > 0$. By definition of $R$, $\tilde{\mvec{x}} \in R$ and, thus, $\tilde{\mvec{x}}$ also lies in $\mvec{p}+\range(\mtx{A})$. This proves that the set $\{ \mvec{p}+\range(\mtx{A}) \cap B_{\epsilon/L}(\mvec{x}^*)\}$ is non-empty.
\end{proof}

\section{A snapshot of conic integral geometry} \label{sec: Conic_integral_geomtry}
A central question posed in conic integral geometry is the following:
\begin{center}
	\emph{What is the probability that a randomly rotated convex cone shares a ray with a fixed convex cone?}
\end{center} 
The answer to this question is given by the conic kinematic formula  \cite{Schneider2008}.
\begin{theorem}[Conic kinematic formula] \label{thm:Conic_kinematic_formula}
	Let $C$ and $F$ be closed convex cones in $\mathbb{R}^D$ such that at most one of them is a linear subspace. Let $\mtx{Q}$ be a $D \times D$ random orthogonal matrix drawn uniformly from the set of all $D \times D$ real orthogonal matrices. Then, 
	\begin{equation} \label{eq: conic_kin_formula}
		\prob[\mtx{Q}F \cap C \neq \{ \mvec{0} \}] = \sum_{k = 0}^{D} (1+(-1)^{k+1}) \sum_{j = k}^{D} v_k(C)v_{D+k-j}(F),
	\end{equation}
where $v_k(C)$ denotes the $k$th intrinsic volume of cone $C$.
\end{theorem}
\begin{proof}
	A proof can be found in \cite[p.~261]{Schneider2008}.
\end{proof}
We plan to use the conic kinematic formula to estimate \eqref{eq:p+range(A)_cap_C_neq_p}. This formula expresses the probability of the intersection of the two cones in terms of quantities known as conic intrinsic volumes. It is thus important to understand the conic intrinsic volumes and ways to compute them. 

\subsection{Conic intrinsic volumes}
Conic intrinsic volumes are commonly defined through the spherical Steiner formula (see \cite{Amelunxen2017}), which we do not define here as it is beyond the scope of this work/not needed here. Instead, we will familiarise ourselves with the conic intrinsic volumes through their properties and specific examples. This is a short introductory review of conic intrinsic volumes; for more details, an interested reader is directed to \cite{Amelunxen2017, Lotz2013, Amelunxen2011, McCoy2014, Schneider2008} and the references therein.

For a closed convex cone $C$ in $\mathbb{R}^D$, there are exactly $D+1$ conic intrinsic volumes: $v_0(C)$, $v_1(C)$,$\dots$, $v_{D}(C)$. Conic intrinsic volumes have useful properties, some of which are summarized below.  Given a closed convex cone $C \subseteq \mathbb{R}^D$, we have (see \cite[Fact 5.5]{Lotz2013}):
\begin{enumerate}[label={(\arabic*)}]
	\item \textbf{Probability distribution}. The intrinsic volumes of the cone $C$ are all nonnegative and sum up to $1$, namely
	\begin{equation} \label{eq: sum_v(C)=1}
		\text{$\sum_{k=0}^{D} v_k(C) = 1$ and $v_k(C) \geq 0$ for $k=0,1, \dots, D.$}
	\end{equation}
	In other words, they form a discrete probability distribution on $ \{0,1, \dots, D \}$.

	\item \textbf{Invariance under rotations}. Given any orthogonal matrix $\mtx{Q} \in \mathbb{R}^{D \times D}$, the intrinsic volumes of the rotated cone $\mtx{Q}C$ and the original cone $C$ are equal:
	\begin{equation}\label{eq: rotational_invariance}
		v_k(\mtx{Q}C) = v_k(C).
	\end{equation}

	\item \textbf{Gauss-Bonnet formula}. If $C$ is not a subspace, we have
	\begin{equation}\label{eq:Gauss-Bonnet}
		 \sum_{\substack{k = 0 \\ k \mbox{ }even}}^{D} v_k(C) = \sum_{\substack{k = 1 \\ k \mbox{ } odd}}^{D} v_k(C)= \frac{1}{2}.
	\end{equation}
	The Gauss-Bonnet formula implies that $v_k(C)\leq 1/2$ for any $k$.
\end{enumerate} 
\begin{rem}
	Conic intrinsic volumes can be viewed as `cousins' of the more familiar Euclidean intrinsic volumes. For a compact convex set $K$ living in $\mathbb{R}^D$, Euclidean intrinsic volumes $v_0^E(K)$, $v_{D-1}^E(K)$ and $v_D^E(K)$ have familiar geometric interpretations: $v_0^E(K)$ --- Euler characteristic, $2v_{D-1}^E(K)$ --- surface area and $v_D^E(K)$ is the usual volume. 
\end{rem}
\begin{rem}
	Conic intrinsic volumes can also be understood using polyhedral cones --- cones that can be generated by intersecting a finite number of halfspaces. If $C$ is a polyhedral cone in $\mathbb{R}^D$, then the $k$th intrinsic volume of $C$ is defined as follows (see \cite[Definition 5.1]{Lotz2013})
	\begin{equation}\label{eq:v_k(C)_of_polyhedral_cone}
		\text{$v_k(C) := \prob[\mvec{\Pi}_{C}(\mvec{a})$ lies in the relative interior\footnotemark of a $k$-dimensional face of $C]$.}
		\footnotetext{The formal definition of relative interior of a set $S$ is as follows: $\rint(S):=\{ \mvec{x} \in S: \exists \delta >0, B_{\delta}(\mvec{x}) \cap \aff(S) \subseteq S \}$, where the affine hull $\aff(S)$ is the smallest affine set containing $S$. For example, the relative interior of a line segment $[A,B]$ living in $\mathbb{R}^2$ is $(A,B)$; the relative interior of a two-dimensional square living in $\mathbb{R}^3$ is the square minus its boundary.}
	\end{equation}
	Here, $\mvec{a}$ denotes the standard Gaussian vector\footnote{A random vector for which each entry is an independent standard normal variable.} in $\mathbb{R}^D$ and $ \mvec{\Pi}_Y(\mvec{x}) := \arg \min_{\mvec{y}} \{ \| \mvec{x} - \mvec{y} \| : \mvec{y} \in Y \}  $ denotes the Euclidean/orthogonal projection of $\mvec{x}$ onto the set $Y$, namely the vector in $Y$ that is the closest to $\mvec{x}$.
\end{rem}
\begin{figure}[!t]
	\centering
	\begin{tikzpicture}
		\fill[blue!40!white] (0,0) -- +(30:3) arc (30:-30:3);
		\fill[red!40!white] (0,0) -- +(120:3) arc (120:240:3);
		
		\draw[color=black] (30:0.5) arc (30:-30:0.5);
		\draw[color=black] (120:0.5) arc (120:240:0.5);
		
		\draw[fill] (0,0) circle(1.2pt);
		\draw[fill] (1,2) circle(1pt);
		\draw[fill] (1.616,0.933) circle(1pt);
		
		\draw[-] (0,0) -- (2.598,1.5);
		\draw[-] (0,0) -- (2.598,-1.5);
		\draw[-] (0,0) -- (-1.5,2.598);
		\draw[-] (0,0) -- (-1.5,-2.598);
		
		\draw[->] (1,2) -- (1.59,0.97);
		
		\draw[rotate around={30:(0,0)}] (0,0) rectangle (0.25,0.25);
		\draw[rotate around={240:(0,0)}] (0,0) rectangle (0.25,0.25);
		
		\node at (3,-1.8) {$C_{\pi/3}$};
		\node at (-3,-1.8) {$C_{\pi/3}^{\circ}$};
		\node at (0.8,0) {$\displaystyle \frac{\pi}{3}$};
		\node at (-0.8,0) {$\displaystyle \frac{2\pi}{3}$};
		\node at (0.7,2) {\scalebox{0.85}{$\mvec{a}$} };
		\node at (0.75,0.95) {\scalebox{0.85}{$\mvec{\Pi}_{C_{\pi/3}}(\mvec{a})$}};
	\end{tikzpicture}
	\caption{A depiction of the two-dimensional polyhedral cone $C_{\pi/3}$ in \Cref{ex:polyhedral_cone}. The projection $\mvec{\Pi}_{C_{\pi/3}}(\mvec{a})$ of $\mvec{a}$ onto $C_{\pi/3}$ falls onto the one-dimensional face of the cone.}
	\label{fig:polyhedral_cone_example}
\end{figure}
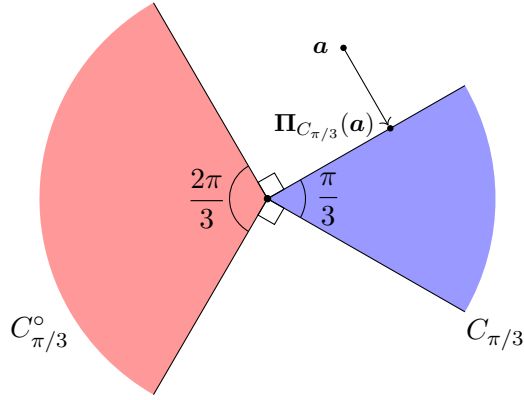
\begin{example}\label{ex:polyhedral_cone}
	Let us consider a simple a two-dimensional polyhedral cone $C_{\pi/3}$ illustrated in \Cref{fig:polyhedral_cone_example} and let us calculate $v_0(C_{\pi/3})$, $v_1(C_{\pi/3})$ and $v_2(C_{\pi/3})$ using \eqref{eq:v_k(C)_of_polyhedral_cone}. 
	
	The cone $C_{\pi/3}$ has a single two-dimensional face (filled with blue), which is the interior of $C_{\pi/3}$. If a random vector $\mvec{a}$ falls inside this face then $\mvec{\Pi}_{C_{\pi/3}}(\mvec{a}) = \mvec{a}$ and, therefore,
	$$ v_2(C_{\pi/3}) = \prob[\mvec{a} \in C_{\pi/3}] = \frac{\pi/3}{2\pi} = \frac{1}{6}. $$
	Let us now calculate $v_0(C_{\pi/3})$. Note that $C_{\pi/3}$ has only one zero-dimensional face, which is the origin. Note also that $\mvec{\Pi}_{C_{\pi/3}}(\mvec{a}) = \mvec{0}$ if and only if $\mvec{a} \in C_{\pi/3}^{\circ}$. Hence,
	$$ v_0(C_{\pi/3}) = \prob[\mvec{a} \in C_{\pi/3}^{\circ}] = \frac{2\pi/3}{2\pi} = \frac{1}{3}. $$
	To calculate $v_1(C_{\pi/3})$, we simply use \eqref{eq: sum_v(C)=1} to obtain 
	$$ v_1(C_{\pi/3}) = 1-v_0(C_{\pi/3})-v_2(C_{\pi/3}) = \frac{1}{2}.$$
\end{example}
\begin{example}[Linear subspace]
	The $k$th intrinsic volume of a $d$-dimensional linear subspace $\mathcal{L}_d$ in $\mathbb{R}^D$ is given by
	\begin{equation}\label{eq: int_vol_lin_subspace}
		v_k(\mathcal{L}_d) = \begin{cases}
			1 & \mbox{if $k = d$,} \\
			0 & \mbox{otherwise.}
		\end{cases}
	\end{equation}
	We already mentioned in \Cref{rem:lin_subspace_is_a_cone} that a $d$-dimensional linear subspace $\mathcal{L}_d$ is a cone. In fact, $\mathcal{L}_d$ is a polyhedral cone which has only one ($d$-dimensional) face. Therefore, the projection of any vector in $\mathbb{R}^D$ onto $\mathcal{L}_d$ will always lie on its (only) $d$-dimensional face. Hence, \eqref{eq: int_vol_lin_subspace} follows from \eqref{eq:v_k(C)_of_polyhedral_cone}.
\end{example}

\begin{example}[Circular cone] A circular cone is another important example; they have a number of applications in convex optimization (see, e.g., \cite[Section 3]{BenTal2001} and \cite[Section 4]{Boyd2004}). The circular cone of angle $\alpha$ in $\mathbb{R}^D$ is denoted by $\ccone_D(\alpha)$ and is defined as 
	\begin{equation} \label{eq: ccone}
	    \ccone_D(\alpha) := \{ \mvec{x} \in \mathbb{R}^D : x_1 \geq \| \mvec{x} \| \cos(\alpha) \} \mbox{ for $0\leq \alpha \leq \pi/2$}.
	\end{equation} 	
	The circular cone can be viewed as a collection of rays connecting the origin and some $D$-dimensional ball which does not contain the origin in its interior. The intrinsic volumes of $\ccone_D(\alpha)$ are given by the formulae (see \cite[Appendix D.1]{Lotz2013}): 
	\begin{equation} \label{eq: d_intr_vol_ccone}
		v_k(\ccone_D(\alpha)) = \frac{1}{2} \binom{(D-2)/2}{(k-1)/2} \sin^{k-1}(\alpha) \cos^{D-k-1}(\alpha)
	\end{equation}
	for $k = 1,2,\dots, D-1$, where $\binom{i}{j}$ is the extension of the binomial coefficient to noninteger $i$ and $j$ through the gamma function,
	\begin{equation}\label{eq:general_binom_coeff}
		\binom{i}{j} = \frac{\Gamma(i+1)}{\Gamma(j+1) \Gamma(i-j+1)}.
	\end{equation}
	The $0$th and $D$th intrinsic volumes of the circular cone are given by (see \cite[Ex. 4.4.8]{Amelunxen2011}): 
	\begin{align}
		v_0(\ccone_D(\alpha)) & = \frac{D-1}{2} \binom{(D-2)/2}{-1/2}\int_{0}^{\pi/2 - \alpha} \sin^{D-2}(x) dx, \label{eq: 0_intr_vol_ccone} \\
		v_D(\ccone_D(\alpha)) & = \frac{D-1}{2} \binom{(D-2)/2}{(D-1)/2}\int_{0}^{\alpha} \sin^{D-2}(x) dx. \label{eq: D_intr_vol_ccone}
	\end{align}
\end{example}
The following property of circular cones will be needed later.  
\begin{lemma}\label{lemma:circ(alpha)<circ(beta)_if_alpha<beta}
	Let $\ccone_D(\alpha)$ and $\ccone_D(\beta)$ be two circular cones with $0 \leq \alpha \leq \beta \leq \pi/2$. Then, $\ccone_D(\alpha) \subseteq \ccone_D(\beta)$. 
\end{lemma}
\begin{proof}
	Let $\mvec{v}$ be any point in $\ccone_D(\alpha)$. By definition of $\ccone_D(\alpha)$, $ v_1 \geq \| \mvec{v} \| \cos(\alpha)$. Since $0 \leq \alpha \leq \beta \leq \pi/2$, it follows that $ v_1 \geq \| \mvec{v} \| \cos(\beta)$, which by definition of $\ccone_D(\beta)$, implies that $\mvec{v}$ must also lie in $\ccone_D(\beta)$.
\end{proof}

\subsection{The Crofton formula}
We now present a useful corollary of the conic kinematic formula. If one of the cones in \Cref{thm:Conic_kinematic_formula} is given by a linear subspace then the conic kinematic formula reduces to the Crofton formula.
\begin{corollary}[Crofton formula] \label{cor: Crofton_formula}
	Let $C$ be a closed convex cone in $\mathbb{R}^D$ and $\mathcal{L}_d$ be a $d$-dimensional linear subspace. Let $\mtx{Q}$ be a $D \times D$ random orthogonal matrix drawn uniformly from the set of all $D \times D$ real orthogonal matrices. We have
	\begin{equation}\label{eq:Crofton_formula}
		\prob[\mtx{Q}\mathcal{L}_d \cap C \neq \{ \mvec{0} \} ] = 2 h_{D-d+1},
	\end{equation}
	with 
	\begin{equation} \label{eq: def_h_k}
	h_{D-d+1} := \begin{cases}
			v_{D-d+1}(C) + v_{D-d+3}(C) + \cdots + v_D(C) & \mbox{if $d$ is odd,}  \\
			v_{D-d+1}(C) + v_{D-d+3}(C) + \cdots + v_{D-1}(C) & \mbox{if $d$ is even.} 
		\end{cases}
	\end{equation}
\end{corollary}
The Crofton formula is easily derived from \eqref{eq: conic_kin_formula} using the fact that the $k$th intrinsic volume of a linear subspace $\mathcal{L}_d$ is $1$ if $d = k$ and $0$ otherwise.  The Crofton formula will be essential in estimating the probability of $\epsilon$-success of \eqref{eq: AREGO}.

\section{Bounding the probability of $\epsilon$-success of the reduced problem \eqref{eq: AREGO}} \label{sec:est_prob_eps_success_RPX_gen_fun}
Building on the tools developed in the last section, we can estimate the right-hand side of \eqref{eq:RPX_eps_succ>p+range(A)_cap_C_noteq_p} in \Cref{thm:RPX_eps_succ>p+range_cap_C(p)}, and thereby obtain bounds on the probability of $\epsilon$-success of \eqref{eq: AREGO}.

Note that if $\mvec{p} \notin B_{\epsilon/L}(\mvec{x}^*)$, then $C_{\mvec{p}}(\mvec{x}^*)$ defined in \eqref{def:C(p)} is a circular cone $\ccone_D(\alpha^*_{\mvec{p}})$ with $\alpha^*_{\mvec{p}} = \arcsin(\epsilon/(L\| \mvec{x}^* - \mvec{p} \|))$ that has been rotated and then translated by $\mvec{p}$, see \eqref{eq: ccone}. Therefore, the intersection $\mvec{p}+\range(\mtx{A}) \cap C_{\mvec{p}}(\mvec{x}^*)$ in \eqref{eq:RPX_eps_succ>p+range(A)_cap_C_noteq_p} is that of a random $d$-dimensional linear subspace and the rotated circular cone both translated by $\mvec{p}$. We can translate these `cones' back to the origin and then, using the Crofton formula, evaluate the right-hand side of \eqref{eq:RPX_eps_succ>p+range(A)_cap_C_noteq_p} exactly since the expressions for the conic intrinsic volumes of the circular cone $C_{\mvec{p}}(\mvec{x}^*)$ are known (see \eqref{eq: d_intr_vol_ccone}, \eqref{eq: 0_intr_vol_ccone} and \eqref{eq: D_intr_vol_ccone}). The Crofton formula and the right-hand side of \eqref{eq:RPX_eps_succ>p+range(A)_cap_C_noteq_p} only differ in the formulation of a random linear subspace: in the former, a random linear subspace is given as $\mtx{Q}\mathcal{L}_d$, whereas in \eqref{eq:RPX_eps_succ>p+range(A)_cap_C_noteq_p} it is represented by $\range(\mtx{A})$. The following theorem states that these two representations are equivalent.
\begin{theorem} \label{thm:A = QL}
	Let $\mtx{A} \in \mathbb{R}^{D \times d}$ be a Gaussian matrix. Let $\mtx{Q}$ be a $D \times D$ random orthogonal matrix drawn uniformly from the set of all $D \times D$ real orthogonal matrices and let $\mathcal{L}_d$ be a $d$-dimensional linear subspace in $\mathbb{R}^D$. Then,
	\begin{equation}
		\range(\mtx{A}) \stackrel{law}{=} \mtx{Q}\mathcal{L}_d.
	\end{equation}  
\end{theorem}
\begin{proof}
	See proof of \cite[Theorem 1.2]{Goldstein2017}.
\end{proof}

The transformation of  \eqref{eq:RPX_eps_succ>p+range(A)_cap_C_noteq_p} into a form suitable for the application of Crofton formula is given in the following corollary.
\begin{corollary}\label{cor:RPX_eps_succ>QL_cap_Circ}
	Let \Cref{ass:f_is_Lipschitz} hold. Let $\mtx{A}$ be a $D \times d$ Gaussian matrix, $\mtx{Q}$ be a $D \times D$ random orthogonal matrix drawn uniformly from the set of all $D \times D$ real orthogonal matrices and $\mathcal{L}_d$ be a $d$-dimensional linear subspace in $\mathbb{R}^D$. Let  $\epsilon > 0$ an accuracy tolerance and let $\mvec{p} \in \mathcal{X} \setminus G_\epsilon$ be a given vector. Let $\ccone_D(\alpha^*_{\mvec{p}})$ be the circular cone with $\alpha^*_{\mvec{p}} = \arcsin(\epsilon/(L\| \mvec{x}^* - \mvec{p}\|))$, where $\mvec{x}^*$ is any global minimizer of \eqref{eq: GO} satisfying \Cref{ass:B_eps_is_a_ball}. Then,
	\begin{equation}\label{eq:RPX_eps_succ>QL_cap_Circ}
		\prob[\eqref{eq: AREGO} \ \text{is $\epsilon$-successful}] \geq \prob[\mtx{Q}\mathcal{L}_d \cap \ccone_D(\alpha^*_{\mvec{p}}) \neq \{ \mvec{0} \}].
	\end{equation}
\end{corollary}
\begin{proof}
	As mentioned earlier, by definition, $C_{\mvec{p}}(\mvec{x}^*)$ is the rotated and translated (by $\mvec{p}$) circular cone $\ccone_D(\alpha^*_{\mvec{p}})$. That is, there exists a $D \times D$ orthogonal matrix $\mtx{S}$ such that $C_{\mvec{p}}(\mvec{x}^*) = \mvec{p}+\mtx{S}\ccone_D(\alpha^*_{\mvec{p}})$. Then,  \Cref{thm:RPX_eps_succ>p+range_cap_C(p)} implies 
	\begin{equation}
		\begin{aligned}
			\prob[\eqref{eq: AREGO} \ \text{is $\epsilon$-successful}] & \geq \prob[\mvec{p}+\range(\mtx{A}) \cap \mvec{p} + \mtx{S}\ccone_D(\alpha^*_{\mvec{p}}) \neq \{ \mvec{p} \}] \\
			& = \prob[\range(\mtx{A}) \cap \mtx{S}\ccone_D(\alpha^*_{\mvec{p}}) \neq \{ \mvec{0} \}] \\
			& = \prob[\mtx{S}^T\range(\mtx{A}) \cap \ccone_D(\alpha^*_{\mvec{p}}) \neq \{ \mvec{0} \}] \\
			& = \prob[\range(\mtx{A}) \cap \ccone_D(\alpha^*_{\mvec{p}}) \neq \{ \mvec{0} \}] \\
			& = \prob[\mtx{Q}\mathcal{L}_d \cap \ccone_D(\alpha^*_{\mvec{p}}) \neq \{ \mvec{0} \}],
		\end{aligned}
	\end{equation}
	where the penultimate equality follows from the orthogonal invariance of Gaussian matrices and where the last equality follows from \Cref{thm:A = QL}.
\end{proof}
\Cref{cor:RPX_eps_succ>QL_cap_Circ} now allows us to use the Crofton formula to quantify the lower bound in \eqref{eq:RPX_eps_succ>QL_cap_Circ}. In the next theorem, we derive our first lower bound, that is dependent on the location of $\mvec{p}$ in $\mathcal{X}$. In particular, note that $\mvec{p}$ is assumed to be at a distance at least $\epsilon/L$ from $\mvec{x}^*$.

\begin{theorem}[A lower bound on the success probability]\label{thm:RPX_eps_succ>t(r_p)}
	Let \Cref{ass:f_is_Lipschitz} hold, let $\mtx{A}$ be a $D \times d$ Gaussian matrix and $\epsilon >0$, an accuracy tolerance. Let $\mvec{p} \in \mathcal{X}\setminus G_{\epsilon}$ be a given vector and let $r_{\mvec{p}} := \epsilon/(L\| \mvec{x}^* - \mvec{p}\|)$, where $\mvec{x}^*$ is any global minimizer of \eqref{eq: GO} that satisfies \Cref{ass:B_eps_is_a_ball}. Then, 
	\begin{equation}\label{eq:RPX_eps_succ>tau_p}
			\prob[\eqref{eq: AREGO} \ \text{is $\epsilon$-successful}] \geq \tau = \tau(r_{\mvec{p}}, d, D),
	\end{equation}
	where the function $\tau(r, d, D)$ for $0<r<1$ and $1 \leq d < D$ is defined as
	\begin{equation}\label{def:tau_p}                                   \tau(r, d, D) :=
		\begin{cases}                                      
			\displaystyle
			(D-1) \cdot                          \binom{\frac{D-2}{2}}{\frac{D-1}{2}}\int_{0}^{\arcsin(r)} \sin^{D-2}(x) dx	                & \text{if $d = 1$,} \\
			\displaystyle \binom{\frac{D-2}{2}}{\frac{D-d}{2}} r^{D-d} (1-r^2)^{\frac{d-2}{2}}  & \text{if $1<d< D$.} 
		\end{cases}
	\end{equation}
Here, $\binom{i}{j}$ denotes the general binomial coefficient defined in \eqref{eq:general_binom_coeff}.
\end{theorem}
\begin{proof}
	Let $\alpha^*_{\mvec{p}} = \arcsin(r_{\mvec{p}})$ and let $C$ denote $\ccone_D(\alpha^*_{\mvec{p}})$ for notational convenience. First, note that by \eqref{eq: d_intr_vol_ccone} and \eqref{eq: D_intr_vol_ccone},  $\tau(r, d, D) = 2v_{D-d+1}(\ccone_D(\arcsin(r)))$. Thus, all we need to show is that $\prob[\eqref{eq: AREGO} \ \text{is $\epsilon$-successful}]$ is lower bounded by $2v_{D-d+1}(C)$.
	
	By \eqref{eq:RPX_eps_succ>QL_cap_Circ} and the Crofton formula \eqref{eq:Crofton_formula}, we have
	\begin{equation}\label{eq:RPX_eps_succ>2v}
		\begin{aligned}
			\prob[\eqref{eq: AREGO} \ \text{is $\epsilon$-successful}] & \geq \begin{cases}
				2(v_{D-d+1}(C) + v_{D-d+3}(C) + \cdots + v_D(C)) & \mbox{if $d$ is odd,}  \\
				2(v_{D-d+1}(C) + v_{D-d+3}(C) + \cdots + v_{D-1}(C)) & \mbox{if $d$ is even} 
			\end{cases} \\
		& \geq 2v_{D-d+1}(C),
		\end{aligned}
	\end{equation}
	where the inequality follows from the fact that $v_k(C)$'s are all nonnegative (see \eqref{eq: sum_v(C)=1}). 
\end{proof}

Let us explain why we choose to bound the $\epsilon$-success of \eqref{eq: AREGO} in \eqref{eq:RPX_eps_succ>2v} by a multiple of $v_{D-d+1}(C)$ in particular, whereas we could have chosen any other intrinsic volume or the entire sum of these volumes. Our reason for such a choice for the lower bound is underpinned by the following observation: using the formulae \eqref{eq: d_intr_vol_ccone} and \eqref{eq: D_intr_vol_ccone} for the intrinsic volumes, one can verify that $v_{D-d+i}(C)/v_{D-d+1}(C) = \mathcal{O}(D^{(1-i)/2})$ for $i = 1,2,\dots,d$ as $D \rightarrow \infty $ with other parameters kept fixed\footnote{The term $v_{D-d+1}(C)$ is dominant also in the scenario when $\| \mvec{x}^* - \mvec{p} \| \rightarrow \infty$ as $D\rightarrow \infty$ with other parameters fixed. In this case, $v_{D-d+i}(C)/v_{D-d+1}(C) = \mathcal{O}((r_{\mvec{p}}/\sqrt{D})^{i-1})$ for $i = 1,2,\dots,d$ as $D \rightarrow \infty $.}. Hence,
$$ v_{D-d+1}(C) + v_{D-d+3}(C) + \cdots = v_{D-d+1}(C)\cdot(1+\mathcal{O}(1/D)). $$
Therefore, approximating the sum by its leading term $v_{D-d+1}(C)$ is reasonable for large values of $D$. 

Given a global minimizer $\mvec{x}^*$ of \eqref{eq: GO} that satisfies \Cref{ass:B_eps_is_a_ball} and a positive constant $R_{\max}$, the following result provides a lower bound on the probability of $\epsilon$-success of \eqref{eq: AREGO} that holds for all $\mvec{p} \in \mathcal{X}$ satisfying $\rVert \mvec{x}^* - \mvec{p}\rVert \leq R_{\max}<\infty$. Note that, in contrast with the last theorem, this result holds for $\mvec{p}$ arbitrarily close to $\mvec{x}^*$; as such, it will be crucial to the convergence of our algorithmic proposals in \Cref{sec:convergence}. Note that  there are natural ways to choose $R_{\max}$ in some cases:
\begin{itemize}
    \item If a sequence of reduced problems \eqref{eq: AREGO} is being considered such that the random subspaces are drawn at the same $\mvec{p} \in \mathcal{X}$, on can simply take $R_{\max}=\|\mvec{x^*}-\mvec{p}\|$.
    \item If the sequence of reduced problems \eqref{eq: AREGO} corresponds to a bounded parameter sequence $\{\mvec{p}^0, \mvec{p}^1, \dots\}$, one can choose $R_{\max}$ to be the (finite) supremum over the sequence $\{\|\mvec{x}^*-\mvec{p}^i\|\}$ for $i\geq 0$.
    \item If $\mathcal{X}$ is bounded, since $\mvec{p}\in\mathcal{X}$ and $\mvec{x}^*\in \mathcal{X}$, one can simply choose $R_{\max}$ to be the diameter of $\mathcal{X}$. 
\end{itemize}

 Note that when $\mathcal{X}$ is not bounded, it is in general difficult to derive a uniform lower bound on the probability of $\epsilon$-success of \eqref{eq: AREGO} that is valid for all $\mvec{p} \in \mathcal{X}$ (taking $\mvec{p} \to \mvec{\infty}$ will make the lower bound go to zero). The above list provides two examples of rules for selecting $\mvec{p}$ that guarantee that the result below holds even in the case $\mathcal{X}$ bounded. Other examples are given in \Cref{sec: XREGO}.

\begin{theorem}[A uniform lower bound on the success probability]\label{thm:RPX_eps_succ_gen_uniform_bound}
	Suppose that \Cref{ass:f_is_Lipschitz} holds. Let $\mtx{A}$ be a $D \times d$ Gaussian matrix, $\epsilon$ a positive accuracy tolerance,  $\mvec{x}^*$ a global minimizer of \eqref{eq: GO} that satisfies \Cref{ass:B_eps_is_a_ball}. For all $\mvec{p} \in \mathcal{X}$ satisfying $\rVert \mvec{p} - \mvec{x}^* \rVert < R_{\max}$ for some suitably chosen constant $R_{\max}$, we have 
	\begin{equation}\label{eq:RPX_eps_succ>tau}
		\prob[\eqref{eq: AREGO} \ \text{is $\epsilon$-successful}] \geq \tau =  \tau(r_{min}, d, D),
	\end{equation}
	where $\tau(\cdot, \cdot, \cdot)$ is defined in \eqref{def:tau_p} and $r_{min} := \epsilon/(L R_{\max})$.
\end{theorem}
\begin{proof}
	Let $\mvec{x}^*$ be a global minimizer that satisfies \Cref{ass:B_eps_is_a_ball}, let $r_{\mvec{p}}$ be defined in \Cref{thm:RPX_eps_succ>t(r_p)} and let $\alpha^*_{\mvec{p}} = \arcsin(r_{\mvec{p}})$. We consider the two cases $\mvec{p} \in \mathcal{X}\setminus G_{\epsilon}$ and $\mvec{p}\in G_{\epsilon}$ separately. 

	First, let $\mvec{p}$ be any point in $ \mathcal{X}\setminus G_{\epsilon}$. Then,
	\begin{equation} \label{eq:r_p>r_min}
		\begin{aligned}
			r_{\mvec{p}} & = \frac{\epsilon}{L \rVert \mvec{p} - \mvec{x}^* \rVert} \geq r_{min}, \\
			\alpha^*_{\mvec{p}} & \geq \arcsin(r_{min}) := \alpha^*_{min}.
		\end{aligned}
	\end{equation}
	Now, define $C_{min}:= \ccone_D(\alpha^*_{min})$. By \eqref{eq:r_p>r_min} and \Cref{lemma:circ(alpha)<circ(beta)_if_alpha<beta}, it follows that $C_{min} \subseteq \ccone_D(\alpha^*_{\mvec{p}})$. Using \Cref{cor:RPX_eps_succ>QL_cap_Circ}, we then obtain 
	\begin{equation}\label{eq:RPX_eps_succ>2v(C_min)}
		\begin{aligned}
					\prob[\eqref{eq: AREGO} \ \text{is $\epsilon$-successful}] & \geq \prob[\mtx{Q}\mathcal{L}_d \cap \ccone_D(\alpha^*_{\mvec{p}}) \neq \{ \mvec{0} \}] \\
					& \geq \prob[\mtx{Q}\mathcal{L}_d \cap C_{min} \neq \{ \mvec{0} \}] \\
					& \geq 2v_{D-d+1}(C_{min}), \\
		\end{aligned}
	\end{equation}
	where the last inequality follows from the same line of argument as in \eqref{eq:RPX_eps_succ>2v}. Using \eqref{eq: d_intr_vol_ccone} and \eqref{eq: D_intr_vol_ccone}, it is easy to verify that $2v_{D-d+1}(C_{min}) = \tau(r_{min}, d, D)$. We have shown \eqref{eq:RPX_eps_succ>tau} for $\mvec{p} \in \mathcal{X}\setminus G_{\epsilon}$.
	
	For $\mvec{p} \in G_{\epsilon}$, \eqref{eq:RPX_eps_succ>tau} holds trivially, since if $\mvec{p} \in G_{\epsilon}$, \eqref{eq: AREGO} is $\epsilon$-successful with probability 1. As a sanity check, $1 \geq 2v(C_{min}) = \tau(r_{min}, d, D)$ where the inequality is implied by the Gauss-Bonnet formula \eqref{eq:Gauss-Bonnet}. 
\end{proof}

Unfortunately, the formula defining $\tau(r, d, D)$ is not easy to interpret. To better understand the dependence of the lower bounds \eqref{eq:RPX_eps_succ>tau_p} and \eqref{eq:RPX_eps_succ>tau} on the parameters of the problem, we now analyse the behaviour of $\tau(r, d, D)$ in the asymptotic regime.

\subsection{Asymptotic expansions} \label{sec:tau_asymp_exp_gen_obj}
We establish the asymptotic behaviour of $\tau(r, d, D)$ for large $D$. The other parameters are kept fixed except for $r$ which we allow to decrease with $D$. Note indeed that $r_{\mvec{p}}$ in \Cref{thm:RPX_eps_succ>t(r_p)} is inversely proportional to $\| \mvec{x}^* - \mvec{p} \|$, which typically increases with $D$. Before we begin, we first need to establish the following lemma.
\begin{lemma}\label{lemma:int_sin^n(x)_asymptotic_exp}
	Let $0 < \alpha < \pi/2$ be either a fixed angle or a function of $D$ that tends to 0 as $D \rightarrow \infty$. Then, as $D \rightarrow \infty$, 
	\begin{equation}
		\int_{0}^{\alpha} \sin^D(x) dx = \frac{1}{D} \frac{\sin^{D+1}(\alpha)}{\cos(\alpha)} + O\left(\frac{\sin^{D+1}(\alpha)}{D^2}\right).
	\end{equation}
\end{lemma}
\begin{proof}
	We write
	\begin{equation}
		\int_{0}^{\alpha} \sin^D(x) dx = \int_{0}^{\alpha} \frac{\sin(x)}{D\cos(x)} \cdot ( D \cos(x) \sin^{D-1}(x)) dx.
	\end{equation}
	Integration by parts with $u = \sin(x)/(D\cos(x))$ and $dv = D \cos(x) \sin^{D-1}(x) dx$ yields
	\begin{equation}
		\int_{0}^{\alpha} \sin^D(x) dx = \frac{\sin^{D+1}(\alpha)}{D\cos(\alpha)} - \frac{1}{D} \int_0^{\alpha} \frac{\sin^D(x)}{\cos^2(x)}dx.
	\end{equation}
	Let $I$ denote $\int_0^{\alpha} \frac{\sin^D(x)}{\cos^2(x)}dx$. It remains to show that $I = O(\sin^{D+1}(\alpha)/D)$. We express $I$ as 
	\begin{equation}
		\int_0^{\alpha} \frac{\sin(x)}{D\cos^3(x)} \cdot (D\cos(x)\sin^{D-1}(x)) dx.
	\end{equation}
	We integrate $I$ by parts with $u = \sin(x)/(D\cos^3(x))$ and $dv = D\cos(x)\sin^{D-1}(x) dx$ to obtain
	\begin{equation}
		I = \frac{1}{D} \frac{\sin^{D+1}(\alpha)}{\cos^3(\alpha)} - \frac{1}{D} \int_0^{\alpha} \frac{1+2\sin^2(x)}{\cos^4(x)}\sin^D(x)dx
	\end{equation}
	Since the latter integral is positive, we have
	\begin{equation}\label{eq:I<sin^n+1(a)/n}
		I \leq \frac{1}{\cos^3(\alpha)} \cdot \frac{\sin^{D+1}(\alpha)}{D}.
	\end{equation}
	Since $I$ is positive for any $0 < \alpha < \pi/2$, \eqref{eq:I<sin^n+1(a)/n} implies that $I = O(\sin^{D+1}(\alpha)/D)$.
\end{proof}
We establish the asymptotic behaviours of $\tau(r_{\mvec{p}},d,D)$ and $\tau(r_{min},d,D)$ by analysing  the asymptotics of $\tau(r,d,D)$ defined in \eqref{def:tau_p} and later substituting $r_{\mvec{p}}$ and  $r_{min}$ for $r$ in $\tau(r,d,D)$.
\begin{theorem}\label{thm:tau(r)_asymptotic}
	Let $\tau(r,d,D)$ be defined in \eqref{def:tau_p}. Let $d$ be fixed and let $r$ be either fixed or tend to zero as $D \rightarrow \infty$. Then,
	\begin{equation}\label{eq:tau(r)_asymptotic}
		\text{$\tau(r,d,D) = \Theta\left( D^{\frac{d-2}{2}} r^{D-d} \right)$ as $D \rightarrow \infty$,}
	\end{equation}
	and the constants in $\Theta(\cdot)$ are independent of $D$.
\end{theorem}
\begin{proof}
	We prove \eqref{eq:tau(r)_asymptotic} for $d=1$ and $1<d<D$ separately. 
	
	First, assume that $d > 1$. By definition of $\tau(r,d,D)$, we have  
	\begin{equation}\label{eq:tau(r_min)_def_expanded}
			\tau(r,d,D)  = \binom{\frac{D-2}{2}}{\frac{D-d}{2}} r^{D-d}(1-r^2)^{\frac{d-2}{2}}.
	\end{equation}
	Let us first  determine the asymptotic behaviour of the binomial coefficient. Using the fact that $\Gamma(z+a)/\Gamma(z+b) = \Theta( z^{a-b})$ for large $z$ (see, e.g., \cite{Tricomi1951}), we obtain 
	\begin{equation} \label{eq:ratio_of_gamma_fun_asym_exp}
		\binom{\frac{D-2}{2}}{\frac{D-d}{2}} = \frac{\Gamma(\frac{D}{2})}{\Gamma(\frac{D-d+2}{2}) \Gamma(\frac{d}{2})} = \frac{\Gamma(\frac{D-d+2}{2} + \frac{d-2}{2})}{\Gamma(\frac{D-d+2}{2}) \Gamma(\frac{d}{2})} = \Theta \left( \left( \frac{D-d+2}{2} \right)^{\frac{d-2}{2}}\right) =  \Theta\left(D^{\frac{d-2}{2}}\right).
	\end{equation}
	To obtain\footnote{Here, we have also used the fact that if functions $f(x)$, $f'(x)$, $g(x)$ and $g'(x)$ satisfy $f(x) = \Theta(g(x))$ and $f'(x) = \Theta(g'(x))$ (as $x \rightarrow \infty$), then $f(x)f'(x) = \Theta(g(x)g'(x))$.} \eqref{eq:tau(r)_asymptotic}, we substitute \eqref{eq:ratio_of_gamma_fun_asym_exp} into \eqref{eq:tau(r_min)_def_expanded}. Note that $(1-r^2)^{\frac{d-2}{2}}$ is bounded above and bounded away from zero by constants independent of $D$; thus, it can be absorbed into the constants of $\Theta$. 
	
	Let us now prove \eqref{eq:tau(r)_asymptotic} for $d=1$. We have
	\begin{equation}\label{eq:t(r_min)_d=1_in_thm}
		\tau(r,d,D) = 
		(D-1) \cdot                          \binom{\frac{D-2}{2}}{\frac{D-1}{2}}\int_{0}^{\arcsin(r)} \sin^{D-2}(x) dx,
	\end{equation}
	where, by \eqref{eq:ratio_of_gamma_fun_asym_exp}, 
	\begin{equation}\label{eq:binom_coeff_sim_d=1}
		\binom{\frac{D-2}{2}}{\frac{D-1}{2}} = \Theta \left( D^{-\frac{1}{2}}  \right)               
	\end{equation}
	and, by \Cref{lemma:int_sin^n(x)_asymptotic_exp}, 
	\begin{equation}\label{eq:int_sin_dx_through_Theta}
			\int_{0}^{\arcsin(r)} \sin^{D-2}(x) dx =\Theta \left(\frac{1}{D-1} \frac{r^{D-1}}{\sqrt{1-r^2}} \right).
	\end{equation}
	By substituting \eqref{eq:binom_coeff_sim_d=1} and \eqref{eq:int_sin_dx_through_Theta} into \eqref{eq:t(r_min)_d=1_in_thm}, we obtain \eqref{eq:tau(r)_asymptotic} for $d = 1$. For similar reasons as stated above, we can relegate the term $1/\sqrt{1-r^2}$ in \eqref{eq:int_sin_dx_through_Theta} into the constants of $\Theta$. 
\end{proof}
Now, to obtain the asymptotics for $\tau(r_{\mvec{p}},d,D)$ and $\tau(r_{\min},d,D)$, we simply apply \Cref{thm:tau(r)_asymptotic} for $r = r_{\mvec{p}} = \epsilon/(L\| \mvec{x}^* - \mvec{p} \|)$ and $r = r_{\min} = \epsilon/(L R_{\max})$, respectively.

\begin{corollary}\label{cor:t(r)_asymp}
  Asymptotically for $D \to \infty$, keeping $d$, $\epsilon$ and $L$ fixed and letting $\| \mvec{x}^* - \mvec{p}\|$ be either fixed or tend to infinity as $D \rightarrow \infty$, the lower bounds \eqref{eq:RPX_eps_succ>tau_p} and \eqref{eq:RPX_eps_succ>tau} satisfy
	\begin{equation} \label{eq:tau(r_p)_asymptotic}
		\text{$\tau(r_{\mvec{p}}, d, D) = \Theta\left( D^{\frac{d-2}{2}} \left( \frac{\epsilon}{L \| \mvec{x}^* - \mvec{p} \|} \right)^{D-d} \right)$ as $D \rightarrow \infty$,}
	\end{equation}
	with $r_{\mvec{p}} = \epsilon/(L \rVert \mvec{x}^* - \mvec{p} \rVert)$  and where the constants in $\Theta(\cdot)$ are independent of $D$. Similarly, 
    \begin{equation} \label{eq:tau(r_p)_asymptotic}
		\text{$\tau(r_{\min}, d, D) = \Theta\left( D^{\frac{d-2}{2}} \left( \frac{\epsilon}{L R_{\max}} \right)^{D-d} \right)$ as $D \rightarrow \infty$,}
	\end{equation}
	with $r_{\min} = \epsilon/(L R_{\max})$.
\end{corollary}
\begin{proof}
	Note that $r_{\mvec{p}} = \epsilon/(L\| \mvec{x}^* - \mvec{p} \|)$ is either fixed or tends to zero as $D \rightarrow \infty$. Then, the result follows from  \Cref{thm:tau(r)_asymptotic}. 
\end{proof}

 \Cref{cor:t(r)_asymp} shows that for any $\mvec{p}$ not in $G_{\epsilon}$, the lower bounds in \Cref{thm:RPX_eps_succ>t(r_p)} and \Cref{thm:RPX_eps_succ_gen_uniform_bound} decrease exponentially with $D$, which is as expected since problem \eqref{eq: GO} is generally NP-hard. Note that this decrease is slower for larger values of $d$ or $\mvec{p}$ closer to $\mvec{x}^*$, which is reassuring.

\subsection{Comparing \eqref{eq: AREGO} to simple random search}

Using the above lower bounds on the probability of $\epsilon$-success of the reduced problem \eqref{eq: AREGO}, we now compare \eqref{eq: AREGO} to a simple random search method to understand the relative performance of \eqref{eq: AREGO} and when it is beneficial to use it for general functions. As a baseline for comparison, we use Uniform Sampling (US) and we restrict ourselves, in this section, to the specific case $\mathcal{X} = [-1,1]^D$ (as this will allow us to estimate the probability of success of US). We start off with the derivation of a lower bound for the probability of $\epsilon$-success of US and the computation of its asymptotics. 

Note that if a uniformly sampled point falls inside $B_{\epsilon/L}(\mvec{x}^*)$ then US is $\epsilon$-successful. This implies that
\begin{equation}\label{eq:US_eps_succ>tau_us}
	\prob[\text{US is $\epsilon$-successful}] \geq \frac{\vol(B_{\epsilon/L}(\mvec{x^*}))}{\vol(\mathcal{X})} = \frac{\pi^{D/2}}{2^D \Gamma(\frac{D}{2}+1)}\left( \frac{\epsilon}{L} \right)^{D}:= \tau_{us},
\end{equation}
where we have used the fact that $\vol(B_{\epsilon/L}(\mvec{x^*})) = \frac{\pi^{D/2}}{\Gamma(\frac{D}{2}+1)}\left( \frac{\epsilon}{L} \right)^{D}$ (see \cite[Equation 5.19.4]{NIST}) and that $\vol(\mathcal{X}) = 2^D$.

Using Stirling's approximation, it is straightforward to establish the asymptotic behaviour of the lower bound $\tau_{us}$.
\begin{lemma}
	Let $\tau_{us}$ be defined in \eqref{eq:US_eps_succ>tau_us} and let $\epsilon$ and $L$ be fixed. Then, 
	\begin{equation}\label{eq:tau_us_asymptotic}
		\text{$\tau_{us} = \Theta\left(D^{-\frac{D}{2}-\frac{1}{2}} \left( \frac{\pi e}{2} \right)^{\frac{D}{2}}  \left( \frac{\epsilon}{L} \right)^D\right)$ as $D \rightarrow \infty$.}
	\end{equation}
\end{lemma} 
\begin{proof}
	By Stirling's approximation (see \cite[Equation 5.11.7]{NIST}),
	\begin{equation}\label{eq:Gamma_stirling_app}
		\text{$\Gamma\left(\frac{D}{2}+1\right) = \Theta \left(e^{-\frac{D}{2}} \left(\frac{D}{2}\right)^{\frac{D+1}{2}} \right)$ as $D \rightarrow \infty$.} 
	\end{equation}
	By substituting \eqref{eq:Gamma_stirling_app} into \eqref{eq:US_eps_succ>tau_us}, we obtain the desired result.
\end{proof}

Let us now compare the lower bound $\tau_{us}$ of US to the lower bound $\tau(r_{\mvec{p}}, d, D)$ for \eqref{eq: AREGO}. It is clear from the analysis of $\tau(r_{\mvec{p}}, d, D)$ in \Cref{sec:tau_asymp_exp_gen_obj} that the probability of $\epsilon$-success of \eqref{eq: AREGO} is higher if $\mvec{p}$ is closer to the set of global minimizers. In the next theorem, we determine a threshold distance $\Delta_0$ between $\mvec{p}$ and a global minimizer $\mvec{x}^*$ such that $\tau(r_{\mvec{p}}, d, D)$ and $\tau_{us}$ are approximately equal to each other. This would tell us how close $\mvec{p}$ should be to $\mvec{x}^*$ for \eqref{eq: AREGO} to have a larger lower bound for the probability of success than that of US. The analysis is done in the asymptotic regime.
\begin{theorem}\label{thm:t(r_p)vs_t_us}
	Suppose that \Cref{ass:f_is_Lipschitz} holds, and that $\mathcal{X} = [-1,1]^D$. Let $\mvec{x}^*$ be a global minimizer of \eqref{eq: GO} satisfying \Cref{ass:B_eps_is_a_ball}.
	Let $\tau(r_{\mvec{p}}, d, D)$ and $\tau_{us}$ be defined in \Cref{thm:RPX_eps_succ>t(r_p)} and \eqref{eq:US_eps_succ>tau_us}, respectively. Let $\epsilon$, $L$, $d$ be fixed and let $\Delta_0 = \sqrt{\frac{2D}{\pi e}}$. Then,
	\begin{enumerate}
		\item[a)] If $\displaystyle \lim_{D\rightarrow \infty}\frac{\Delta_0 }{\| \mvec{x}^* -  \mvec{p}\|} = \psi > 1$, then $\tau(r_{\mvec{p}}, d, D)/\tau_{us} \rightarrow \infty$ as $D \rightarrow \infty$.
		\item[b)] If $\displaystyle \lim_{D\rightarrow \infty}\frac{\Delta_0 }{\| \mvec{x}^* -  \mvec{p}\|} = \psi < 1$, then $\tau(r_{\mvec{p}}, d, D)/\tau_{us} \rightarrow 0$ as $D \rightarrow \infty$.
	\end{enumerate}
\end{theorem}
\begin{proof}
	From \eqref{eq:tau(r_p)_asymptotic} and \eqref{eq:tau_us_asymptotic}, we have
	\begin{equation}\label{eq:tau(r_p)/tau(us)=Theta}
		\begin{aligned}
			\frac{\tau(r_{\mvec{p}}, d, D)}{\tau_{us}} = \frac{\Theta\left( D^{\frac{d-2}{2}} \left( \frac{\epsilon}{L \| \mvec{x}^* - \mvec{p} \|} \right)^{D-d} \right)}{\Theta\left(\frac{1}{\sqrt{D}}\left( \frac{\pi e}{2D} \right)^{\frac{D}{2}} \left( \frac{\epsilon}{L} \right)^D\right)} & \stackrel{\footnotemark}{=}\Theta\left( \left(\frac{\epsilon}{L}\right)^{-d} \left( \frac{2}{\pi e} \right)^{D/2} D^{\frac{D+d-1}{2}} \| \mvec{x}^* - \mvec{p} \|^{d-D} \right) \\
			& = \Theta \Bigg( \Bigg(\underbrace{\left[ \frac{\sqrt{2 D/\pi e}}{\| \mvec{x}^*-\mvec{p} \|} \right]}_\text{$= \Delta_0/\| \mvec{x}^*-\mvec{p} \|$} \cdot  D^{\frac{2d-1}{2(D-d)}} \Bigg)^{D-d} \Bigg),
		\end{aligned}
		\footnotetext{Here, we use the fact that if functions $f(x)$, $f'(x)$, $g(x)$ and $g'(x)$ satisfy $f(x) = \Theta(g(x))$ and $f'(x) = \Theta(g'(x))$ (as $x \rightarrow \infty$), then $f(x)/f'(x) = \Theta(g(x)/g'(x))$.}
	\end{equation}
	Note that in the second line there is a term $\left(\frac{\epsilon}{L}\right)^{-d}\left(\frac{2}{\pi e}\right)^{d/2}$ missing inside $\Theta$, which we removed as it is independent of $D$. Now, by definition of $\Theta$, \eqref{eq:tau(r_p)/tau(us)=Theta} implies that there exist positive constants $M_1$ and $M_2$ such that 
	\begin{equation}\label{eq:M_1<tau(r_p)/t_us<M_2}
		 M_1 \left(\frac{\Delta_0}{\| \mvec{x}^* - \mvec{p} \|} D^{\frac{2d-1}{2(D-d)}} \right)^{D-d} \leq \frac{\tau(r_{\mvec{p}}, d,D)}{\tau_{us}} \leq M_2 \left(\frac{\Delta_0}{\| \mvec{x}^* - \mvec{p} \|} D^{\frac{2d-1}{2(D-d)}} \right)^{D-d}
	\end{equation}
	as $D \to \infty$.
	Note that $D^{\frac{2d-1}{2(D-d)}} \rightarrow 1$ as $D \rightarrow \infty$. Hence, if $\Delta_0/\| \mvec{x}^* - \mvec{p} \| \rightarrow \psi > 1$ then both lower and upper bounds in \eqref{eq:M_1<tau(r_p)/t_us<M_2} tend to infinity implying that $\tau(r_{\mvec{p}}, d, D)/\tau_{us} \rightarrow \infty$. On the other hand,  if $\Delta_0/\| \mvec{x}^* - \mvec{p} \| \rightarrow \psi < 1$ then both lower and upper bounds in \eqref{eq:M_1<tau(r_p)/t_us<M_2} tend to zero implying that $\tau(r_{\mvec{p}}, d, D)/\tau_{us} \rightarrow 0$.
	\end{proof}
\Cref{thm:t(r_p)vs_t_us} tells us that the distance between $\mvec{p}$ and $\mvec{x}^*$ (in the asymptotic setting) must be no greater than $\Delta_0 \approx 0.48\sqrt{D}$ for $\tau(r_{\mvec{p}}, d, D)$ to be larger than $\tau_{us}$ in the case $\mathcal{X} = [-1,1]^D$. Note that, since the distance between the origin and a corner of $\mathcal{X}$ is equal to $\sqrt{D}$ ($> 0.48\sqrt{D}$),  there is no point $\mvec{p}$ such that the ball of radius $\Delta_0$ centred at $\mvec{p}$ covers all points in $\mathcal{X}$. In other words, in the specific case $\mathcal{X} = [-1,1]^D$, for any $\mvec{p}$ in $\mathcal{X}$, there always exists $\mvec{x}^*$ for which $\tau(r_{\mvec{p}}, d, D)$ is smaller than $\tau_{us}$; on the other hand, if $\mvec{p} = \mvec{0}$ and $\mvec{x}^*$ is close to the origin then $\tau(r_{\mvec{p}}, d, D) > \tau_{us}$.  Note also that $\Delta_0$ has no dependence on the embedding subspace dimension $d$. This is due to the asymptotic nature of the analysis: in \eqref{eq:M_1<tau(r_p)/t_us<M_2}, we see that both inequalities depend on $d$, but the dependence diminishes as $D \rightarrow \infty$ since $d$ is kept fixed. Although the asymptotic analysis shows no significant dependence on the subspace dimension, numerical experiments show that the value of $d$ has a notable effect on success of \eqref{eq: AREGO}. In \Cref{fig:RPXvsUS}, we plot $\tau(r_{\mvec{p}}, d, D)$ as a function of $\| \mvec{x}^* - \mvec{p} \|$ for different values of $d$ with $D$ fixed at 200. The lower bound $\tau_{us}$ of US is represented by a black horizontal line. We see that, for larger $d$, $\tau(r_{\mvec{p}}, d, D)$ decreases at a slower rate and has greater threshold distance before becoming smaller than $\tau_{us}$. 
\begin{figure}[!t]
	\centering
	\includegraphics[scale = 0.5]{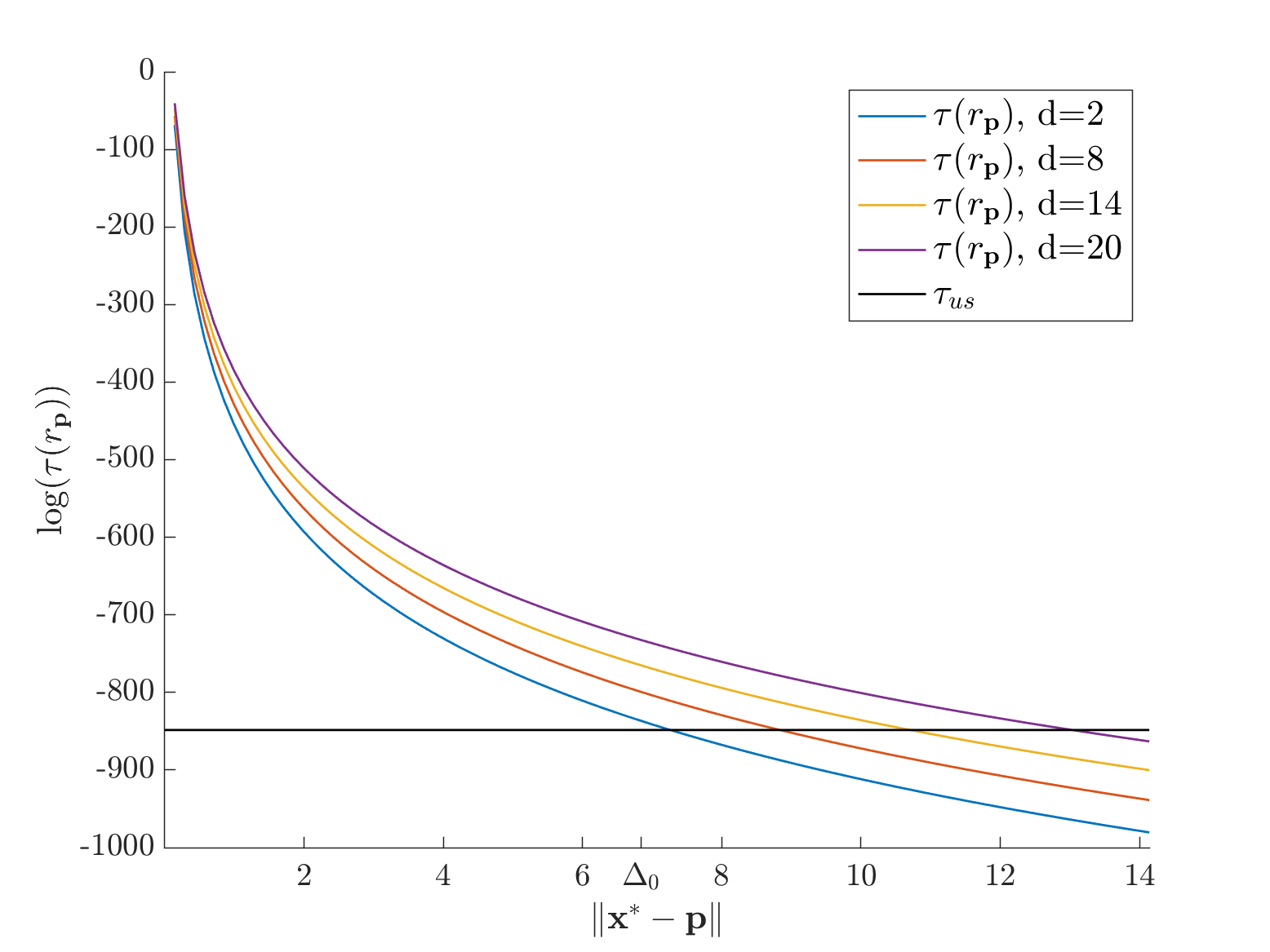}
	\caption{A plot of $\tau(r_{\mvec{p}})$ versus $\| \mvec{x}^* - \mvec{p} \|$ for different values of the subspace embedding dimension $d$. The lower bound $\tau_{us}$ of US does not depend on $\| \mvec{x}^* - \mvec{p} \|$ and, thus, it is displayed as a straight horizontal line.}
	\label{fig:RPXvsUS}
\end{figure}
\begin{rem}
	An important distinction must be made between the implications of the $\epsilon$-success of \eqref{eq: AREGO} and the $\epsilon$-success of US in solving the original problem \eqref{eq: GO}. Note that the $\epsilon$-success of US means that US has sampled a point that lies in $G_{\epsilon}$, which in turn implies that US has successfully (approximately) solved \eqref{eq: GO}. This is not the case for \eqref{eq: AREGO}. Recall that $\epsilon$-success of \eqref{eq: AREGO} by definition means that there is an approximate solution $\mvec{x}^*$ to \eqref{eq: GO} that lies in the embedded $d$-dimensional subspace. One needs to perform an additional global search over the subspace to locate $\mvec{x}^*$. Therefore, for an entirely fair comparison between the two approaches, this additional computational complexity should be taken into account. 
\end{rem}

\section{X-REGO: an algorithmic framework for global optimization using random embeddings}
\label{sec: XREGO}

This section presents the proposed algorithmic framework for global optimization using random embeddings, named X-REGO by analogy with \cite{Cartis2020b} (see the Introduction for distinctions between these variants). X-REGO is a generic algorithmic framework that replaces the high-dimensional original problem \eqref{eq: GO} by a sequence of low-dimensional random problems of the form \eqref{eq: AREGO}; these reduced random problems can then be solved using any global --- and in practice, even  a local --- optimization solver.

Note that the $k$th embedding in X-REGO is determined by a realization $\tilde{\mtx{A}}^k = \mtx{A}^k(\mvec{\omega}^k)$ of the random Gaussian matrix $\mtx{A}^k \in \mathbb{R}^{D \times d^{k}}$, for some (deterministic) $d^k \in \{1, \dots,D-1\}$. For generality of our analysis, we also assume that the parameter $\mvec{p}$ in \eqref{eq: AREGO} is a random variable. The $k$th embedding is drawn at the point $\tilde{\mvec{p}}^{k-1} = \mvec{p}^{k-1}(\mvec{\omega}^{k-1})$, a realization of the random variable $\mvec{p}^{k-1}$, assumed to have support included in $\mathcal{X}$. Note that this definition includes deterministic choices for $\mvec{p}^{k-1}$, by writing it as a random variable with support equal to a singleton (deterministic and stochastic selection rules for the $\mvec{p}$ are given below). 

\begin{algorithm}[h!]
	\caption{$\mathcal{X}$-Random Embeddings for Global Optimization (X-REGO) applied to~\eqref{eq: GO}}
	\label{alg:XREGO}
	\begin{algorithmic}[1]
		\State Initialize $d^1 \in \{1, 2, \dots, D-1\}$ and $\tilde{\mvec{p}}^0 \in \mathcal{X}$.
		\For{\text{$k \geq 1$ until termination}} \label{termination}
		\State Draw $\tilde{\mtx{A}}^k$, a realization of the $D\times d^{k}$ Gaussian matrix $\mtx{A}^k$. \label{line:drawA}
		\State Calculate $\tilde{\mvec{y}}^k$ by solving approximately and possibly, probabilistically,
		\begin{equation}\label{prob: AREGO_subproblem_re}
			\tag{$\widetilde{\text{RP}\mathcal{X}^k}$}
			\begin{aligned} 
				\tilde{f}^k_{min} = \min_{\mvec{y}\in\mathbb{R}^d} & \; f(\tilde{\mtx{A}}^k \mvec{y} + \tilde{\mvec{p}}^{k-1}) \\
				\text{subject to} & \; \tilde{\mtx{A}}^k \mvec{y} + \tilde{\mvec{p}}^{k-1} \in \mathcal{X}. 
			\end{aligned}
		\end{equation}
		\State Let
		\begin{equation} \label{eq: xck}
			\tilde{\mvec{x}}^k := \tilde{\mtx{A}}^k \tilde{\mvec{y}}^k + \tilde{\mvec{p}}^{k-1}.
		\end{equation} 
		\State Choose (deterministically or randomly) $\tilde{\mvec{p}}^k \in \mathcal{X}$. \label{line:choosep}
		\State Choose $d^{k+1} \in \{1,2,\dots,D-1\}$.
		\EndFor
	\end{algorithmic}
\end{algorithm}

Each iteration of X-REGO solves -- approximately and possibly, with a certain probability -- a realization \eqref{prob: AREGO_subproblem_re} of the random problem 
\begin{equation} \label{prob: AREGO_subproblem}
	\tag{$\text{RP}\mathcal{X}^k$}
	\begin{aligned} 
		f^k_{min} = \min_{\mvec{y}} & \; f(\mtx{A}^k \mvec{y} + \mvec{p}^{k-1}) \\
		\text{subject to} & \; \mtx{A}^k \mvec{y} + \mvec{p}^{k-1} \in \mathcal{X}. 
	\end{aligned}
\end{equation}

As such, X-REGO can be seen as a stochastic process: additionally to $\tilde{\mvec{p}}^k$, and $\tilde{\mtx{A}}^k$, each algorithm realization provides sequences  $\tilde{\mvec{x}}^k = \mvec{x}^k(\mvec{\omega}^k)$, $\tilde{\mvec{y}}^k = \mvec{y}^k(\mvec{\omega}^k)$ and $\tilde{f}_{min}^k = f_{min}^k(\mvec{\omega}^k)$, for $k \geq 1$, that are realizations of the random variables $\mvec{x}^k$, $\mvec{y}^k$ and $f_{min}^k$, respectively. 
To calculate $\tilde{\mvec{y}}^k$, \eqref{prob: AREGO_subproblem_re} may be solved to some required accuracy using a deterministic global optimization algorithm that is allowed to fail with a certain probability; or employing a stochastic algorithm, so that $\tilde{\mvec{y}}^k$ is only guaranteed to be an approximate global minimizer of  \eqref{prob: AREGO_subproblem_re}  (at least) with a certain probability. This allows us to account for solvers having some stochastic component (multistart methods, genetic algorithms, ...), or deterministic solvers that may fail in some cases due, e.g., to a computational budget shortage.

Note also that the choice of the random variable $\mvec{p}^k$ and of the subspace dimension $d^k$ provide some flexibility in the algorithm. For $\mvec{p}^k$, possibilities include: 
\begin{itemize}
    \item $\mvec{p}^k = \mvec{p}$: all the random embeddings explored are drawn at the same point (in case $\mvec{p}$ is a fixed vector in $\mathcal{X}$), or according to the same distribution (if $\mvec{p}$ is a random variable),
    \item The sequence $\mvec{p}^0, \mvec{p}^1, \dots$ can be constructed dynamically during the optimization, e.g., based on the information gathered so far on the objective. For example, one may choose $\mvec{p}^k = \mvec{x}_{opt}^k$, where $\mvec{x}_{opt}^k$ is the best point found up to the $k$th embedding:
    \begin{equation} \label{eq: xoptk}
	\mvec{x}_{opt}^k := \arg \min \{ f(\mvec{x}^1), f(\mvec{x}^2), \dots, f(\mvec{x}^k)\}.
\end{equation}
\end{itemize}

Note that \eqref{prob: AREGO_subproblem_re} is always feasible for all choices of $\mvec{p}^k$ ($\mvec{y} = 0$ is feasible since $\tilde{\mvec{p}}^k \in \mathcal{X}$). However, it may happen that this is the only feasible point of \eqref{prob: AREGO_subproblem_re}; to avoid this situation we may assume that $\mvec{p}^k$ is in the interior of $\mathcal{X}$. This latter assumption is not needed for our convergence results to hold, but it is a desirable assumption from a numerical point of view.

 Regarding the subspace dimension $d^k$, one can be for example choose a fixed value based on the computational budget available for the reduced problem, or $d^k$ can be progressively increased, using a warm start in each embedding. We refer the reader to \Cref{sec: Numerics} for a numerical comparison of some of those strategies.

The termination in \Cref{termination} could be set to  a given maximum number of embeddings, or could check that no significant progress in decreasing the objective function has been achieved over the last few embeddings, compared to the value $f(\tilde{\mvec{x}}^k_{opt})$. For generality, we leave it unspecified here.

\section{Global convergence of X-REGO to a set of global \texorpdfstring{$\epsilon$}{eps}-minimizers}
\label{sec:convergence}

The convergence results presented in this paper extend the ones given in \cite{Cartis2020b}, in which X-REGO (with fixed subspace dimension $d^k = d \geq d_e$ for all $k$) is proven to converge for functions with low-effective dimension $d_e$. \Cref{sec:assProof} is devoted to a generic convergence analysis of X-REGO, under generic assumptions on the probability of $\epsilon$-success of \eqref{prob: AREGO_subproblem} and on the probability of success of the solver to find an approximate minimizer of its realisation \eqref{prob: AREGO_subproblem_re}, while \Cref{sec:convergencegenf} presents the application of these results to arbitrary Lipschitz-continuous objectives, building on the results presented in the previous sections to show the validity of the $\epsilon$-success assumption. 

\subsection{A general convergence framework for X-REGO} \label{sec:assProof}
 This section recalls results in \cite{Cartis2020b} that are needed for our main convergence results in the next section.
We show that $\mvec{x}^k_{opt}$ defined in \eqref{eq: xoptk} converges to the set of $\epsilon$-minimizers $G_{\epsilon}$ almost surely as $k \rightarrow \infty$ (see  \Cref{thm: glconv}). Intuitively, our proof relies on the fact that any vector $\tilde{\mvec{x}}^k$ defined in \eqref{eq: xck} belongs to $G_\epsilon$ if the following two conditions hold simultaneously: 
\begin{enumerate}
\item[(a)] the reduced problem \eqref{prob: AREGO_subproblem} is $(\epsilon - \lambda)$-successful in the sense of \Cref{def: success_red_prob}, namely, 
    \begin{equation}\label{eq:succ-red}
	f_{min}^k \leq f^* + \epsilon-\lambda;
    \end{equation}
\item[(b)] the reduced problem \eqref{prob: AREGO_subproblem_re} is solved    (by a deterministic/stochastic algorithm) to an accuracy $\lambda\in (0,\epsilon)$ in the objective function value, namely,
\begin{equation}\label{eq:approxf}
	f(\mtx{A}^k \mvec{y}^k + \mvec{p}^{k-1}) \leq  f_{min}^k + \lambda
\end{equation}
holds (at least) with a certain probability.
\end{enumerate}
Note that in order to prove convergence of X-REGO to (global) $\epsilon$-minimizers, the value of $\epsilon$ in the success probability of the reduced problem \eqref{eq: AREGO} needs to be replaced by $(\epsilon-\lambda)$. This change is motivated by the fact that we allow inexact solutions (up to accuracy $\lambda$) of the reduced problem \eqref{prob: AREGO_subproblem_re}. We also emphasize that, according to the discussion in \Cref{sec: XREGO}, and for the sake of generality, the parameter $\mvec{p}^k$ in \eqref{prob: AREGO_subproblem} is now a random variable (in contrast with \Cref{sec:est_prob_eps_success_RPX_gen_fun} where it was assumed to be deterministic). 

Let us introduce two additional random variables that capture the conditions in (a) and (b) above,
\begin{align}
	R^k  &= \mathds{1}\{\text{\eqref{prob: AREGO_subproblem} is 
		$(\epsilon-\lambda)$-successful in the sense of \eqref{eq:succ-red}}\}, \label{eq: Rk} \\ 
	S^k  &= \mathds{1}\{\text{\eqref{prob: AREGO_subproblem} is solved to accuracy $\lambda$ in the sense of \eqref{eq:approxf}}\}, \label{eq: Sk}
\end{align}
where $\mathds{1}$ is the usual indicator function for an event. 

Let $\mathcal{F}^k = \sigma(\mtx{A}^1, \dots, \mtx{A}^k, \mvec{y}^1, \dots, \mvec{y}^k, \mvec{p}^0, \dots, \mvec{p}^k)$ be the $\sigma$-algebra generated by the random variables $\mtx{A}^1, \dots, \mtx{A}^k, \mvec{y}^1, \dots, \mvec{y}^k, \mvec{p}^0, \dots, \mvec{p}^k$ (a mathematical concept that represents the history of the  X-REGO algorithm as well as its randomness
until the $k$th embedding)\footnote{A similar setup regarding random iterates of probabilistic models can be found in \cite{Bandeira2014, Cartis2018} in the context of local optimization.}, with  $\mathcal{F}^0 = \sigma(\mvec{p}^0)$. 
We also construct an `intermediate' $\sigma$-algebra, namely,
$$\mathcal{F}^{k-1/2} = \sigma(\mtx{A}^1, \dots, \mtx{A}^{k-1}, \mtx{A}^{k}, \mvec{y}^1, \dots, \mvec{y}^{k-1}, \mvec{p}^0, \dots, \mvec{p}^{k-1}), $$
with  $\mathcal{F}^{1/2} = \sigma(\mvec{p}^0, \mtx{A}^{1})$.
Note that $\mvec{x}^k$, $R^k$ and $S^k$ are $\mathcal{F}^{k}$-measurable\footnote{It would be possible to restrict the definition of the $\sigma$-algebra $\mathcal{F}^k$
	so that it contains strictly the randomness of the embeddings $\mtx{A}^i$ and $\mvec{p}^i$ for $i\leq k$; then we would need to assume that $\mvec{y}^k$  
	is $\mathcal{F}^k$-measurable, 
	which would imply that $R^k$, $S^k$ and $\mvec{x}^k$ are also $\mathcal{F}^k$-measurable. Similar comments apply to the definition of 
	$\mathcal{F}^{k-1/2}$.}, and 
$R^k$ is also $\mathcal{F}^{k-1/2}$-measurable;
thus they are well-defined random variables.

\begin{rem}
	The random variables $\mtx{A}^1, \dots, \mtx{A}^k$, $\mvec{y}^1, \dots, \mvec{y}^k$, $\mvec{x}^1, \dots, \mvec{x}^k$, $\mvec{p}^0, \mvec{p}^1, \dots, \mvec{p}^k$, $R^1$, $\dots$, $R^k$, $S^1, \dots, S^{k}$ are  $\mathcal{F}^{k}$-measurable since $\mathcal{F}^0 \subseteq \mathcal{F}^1 \subseteq \cdots \subseteq \mathcal{F}^{k}$. Also, $\mtx{A}^1, \dots, \mtx{A}^k$, $\mvec{y}^1, \dots, \mvec{y}^{k-1}$, $\mvec{x}^1, \dots, \mvec{x}^{k-1}$, $\mvec{p}^0, \mvec{p}^1, \dots, \mvec{p}^{k-1}$, $R^1$, $\dots$, $R^k$, $S^1, \dots, S^{k-1}$ are  $\mathcal{F}^{k-1/2}$-measurable since $\mathcal{F}^0 \subseteq \mathcal{F}^{1/2} \subseteq \mathcal{F}^1 \subseteq \cdots \subseteq \mathcal{F}^{k-1} \subseteq \mathcal{F}^{k-1/2}$.
	
\end{rem}

The following assumption says that the reduced problem \eqref{prob: AREGO_subproblem} needs to be solved to required accuracy with some positive probability. Note that this is a rather weak assumption, that is satisfied by any reasonable solver. 

\begin{assumpA}{Success-Solv}\label{assump: prob_of_I_R>rho}
	For all $k \geq 1$, there exists $\rho^k \in [\rho_{lb},1]$, with $\rho_{lb} > 0$ such that\footnote{The equality in the displayed equation follows from $\mathbb{E}[S^k | \mathcal{F}^{k-1}] = 1 \cdot \prob[ S^k = 1 | \mathcal{F}^{k-1}] + 0 \cdot \prob[ S^k = 0 | \mathcal{F}^{k-1} ]$.}
	$$  \prob[ S^k = 1 | \mathcal{F}^{k-1/2}]=\mathbb{E}[S^k | \mathcal{F}^{k-1/2}]  \geq \rho^k, $$ 
	i.e., with (conditional) probability at least $\rho^k \geq \rho_{lb}$, the solution $\mvec{y}^k$ of \eqref{prob: AREGO_subproblem} 
	satisfies \eqref{eq:approxf}.\footnote{In general, $\rho^k$ will depend on the dimension $d^k$ of the $k$th random embedding.}
\end{assumpA}

\begin{rem}
	If a deterministic (global optimization) algorithm is used to solve \eqref{prob: AREGO_subproblem_re}, then $S^k$ is  always $\mathcal{F}_k^{k-1/2}$-measurable and \Cref{assump: prob_of_I_R>rho} is equivalent to $S^k\geq \rho^k>0$. Since $S^k$ is an indicator function, this further implies that $S^k\equiv 1$.
\end{rem}

The next assumption says that the drawn subspaces are $(\epsilon - \lambda)$-successful with a positive probability.

\begin{assumpA}{Succes-Emb}\label{ass:lowerbdRPK}
    For all $k \geq 1$, there exists $\tau^k \in [\tau_{lb},1]$, with $\tau_{lb} > 0$ such that	\begin{equation}\label{ineq: cond_exp>tau}
		\prob[R^k = 1 | \mathcal{F}^{k-1}] = 	\mathbb{E}[R^k | \mathcal{F}^{k-1}]  \geq \tau^k,
	\end{equation}
	i.e., with (conditional) probability at least $\tau^k \geq \tau_{lb} >0$, \eqref{prob: AREGO_subproblem} is ($\epsilon - \lambda$)-successful. 
\end{assumpA}

Note that \Cref{assump: prob_of_I_R>rho} and \Cref{ass:lowerbdRPK} have been slightly modified compared to \cite{Cartis2020b}: here, the dimension of the reduced problem is varying, so in general the probabilities of success of the solver and embedding depend on $k$ as well. Under \Cref{assump: prob_of_I_R>rho}  and \Cref{ass:lowerbdRPK}, the following result shows the convergence of X-REGO to the set of $\epsilon$-minimizers.

\begin{theorem}[Global convergence]\label{thm: glconv}
	Suppose \Cref{assump: prob_of_I_R>rho} and \ref{ass:lowerbdRPK} hold. Then,
	\begin{equation}\label{eq:lim_prob_x_opt_in_G_eps = 1}
		\lim_{k\rightarrow \infty} \prob[\mvec{x}^k_{opt} \in G_{\epsilon}]=\lim_{k\rightarrow \infty} \prob[f(\mvec{x}^k_{opt}) \leq f^* + \epsilon] = 1
	\end{equation}
	where $\mvec{x}^k_{opt}$ and $G_{\epsilon}$ are defined in \eqref{eq: xoptk} and   \eqref{eq: G_epsilon}, respectively.	Furthermore, for any $\xi \in (0,1)$, 
	\begin{equation}\label{eq:prob[x_opt^k_in_G_eps]>alpha}
		\text{$\prob[ \mvec{x}^k_{opt} \in G_{\epsilon} ]= \prob[f(\mvec{x}^k_{opt}) \leq f^* + \epsilon]\geq \xi$ for all $k \geq K_{\xi}$,}
	\end{equation}
	where $K_\xi:= \displaystyle\ceil*{\frac{|\log(1-\xi)|}{\tau_{lb} \rho_{lb}}}$.
\end{theorem}
\begin{proof}
The proof is a straightforward extension of the one given in \cite{Cartis2020b}, and for completeness, we include it  in \Cref{sec:app_proof_thm65}.
\end{proof}

\begin{rem}
	If the original problem \eqref{eq: GO} is convex (and known a priori to be so), then clearly, a local (deterministic or stochastic) optimization algorithm may be used to solve \eqref{prob: AREGO_subproblem_re}
	and achieve \eqref{eq:approxf}. Apart from this important speed-up and simplification, it seems difficult at present to see how else problem convexity could be exploited in order to improve the success bounds and convergence of X-REGO.
\end{rem}

\subsection{Global convergence of X-REGO for general objectives}
\label{sec:convergencegenf}  

The previous section provides a convergence result, with associate convergence rate, that depends on some parameters $\rho_{lb}$ and $\tau_{lb}$  defined in \Cref{assump: prob_of_I_R>rho} and \ref{ass:lowerbdRPK}. The former intrinsically depends on the solver used to solve the reduced subproblems, and will not be discussed further here. However, the latter parameter $\tau_{lb}$ can be estimated for general Lipschitz-continuous objectives  using the results derived in  \Cref{sec:est_prob_eps_success_RPX_gen_fun}.

\begin{corollary}\label{cor:glconv_gen_fun}
	Suppose that \Cref{ass:f_is_Lipschitz} holds, that there exists a global minimizer $\mvec{x}^*$ of \eqref{eq: GO} that satisfies \Cref{ass:B_eps_is_a_ball} (replacing $\epsilon$ by $\epsilon-\lambda$ in \Cref{ass:B_eps_is_a_ball}, i.e., slightly relaxing the assumption), and that $\tilde{\mvec{p}}^k$ satisfies $\rVert \tilde{\mvec{p}}^k - \mvec{x}^*\rVert \leq R_{\max}$ for all $k$ and for some suitably chosen $R_{\max}$. Suppose also that $d^k \geq d_{lb}$ for some $d_{lb}>0$.  Then, \Cref{ass:lowerbdRPK} holds with 
	\[\tau_{lb} = \tau(r_{\min},d_{lb},D),\]
	with $r_{\min} = (\epsilon-\lambda)/(L R_{\max})$.
\end{corollary}
 
\begin{proof}
Let us first recall that for all $k$, there holds by \Cref{cor:RPX_eps_succ>QL_cap_Circ}:
\[ \prob[\text{\eqref{prob: AREGO_subproblem} is $(\epsilon-\lambda)$-successful}] \geq \prob[\mtx{Q} \mathcal{L}_{d^k} \cap \ccone_D(\alpha^*_{\tilde{\mvec{p}}^{k-1}}) \neq \{\mvec{0}\}], \]
where $\mtx{Q}$ is a $D \times D$ random orthogonal matrix drawn uniformly
from the set of all $D \times D$ real orthogonal matrices, $\mathcal{L}_{d^k}$ a $d^k$-dimensional linear subspace, and $\alpha_{\tilde{\mvec{p}}^{k-1}}^* :=\arcsin((\epsilon-\lambda)/\|\mvec{x}^*-\tilde{\mvec{p}}^{k-1}\|)$. Let $\alpha^*_{\min} := \arcsin((\epsilon-\lambda)/(L R_{\max}))$, and note that $\alpha^*_{\min} \leq \alpha^*_{\tilde{\mvec{p}}^{k-1}}$ for all $k$. By \Cref{lemma:circ(alpha)<circ(beta)_if_alpha<beta}, for any $\alpha^*_{\min} \leq \alpha \leq \pi/2$, there holds $\ccone_D(\alpha^*_{\min}) \subseteq \ccone_D(\alpha)$ so that
\[ \prob[\mtx{Q} \mathcal{L}_{d^k} \cap \ccone_D(\alpha^*_{\tilde{\mvec{p}}^{k-1}}) \neq \{\mvec{0}\}] \geq \prob[\mtx{Q} \mathcal{L}_{d^k} \cap \ccone_D(\alpha^*_{\min}) \neq \{\mvec{0}\}] \]
for all $k$.
By the Crofton formula, there holds 
\[ \prob[\mtx{Q} \mathcal{L}_{d^k} \cap \ccone_D(\alpha^*_{\min}) \neq \{\mvec{0}\}] = 2h_{D-d^k+1}.\]
By \cite[Prop. 5.9]{Lotz2013}, $h_k \geq h_{k+1}$ for all $k = 0, \dots, D-1$. We deduce that
\[  \prob[\mtx{Q} \mathcal{L}_{d^k} \cap \ccone_D(\alpha^*_{\min}) \neq \{\mvec{0}\}] = 2h_{D-d^k+1} \geq 2h_{D-d_{lb}+1}.\]
Using the fact that the intrinsic volumes are all non-negative, and the definition of $h_k$, we get: 
\[  \prob[\mtx{Q} \mathcal{L}_{d^k} \cap \ccone_D(\alpha^*_{\min}) \neq \{\mvec{0}\}] \geq 2 v_{D-d_{lb}+1} = \tau(r_{\min},d_{lb},D).\]
Note finally that, in terms of conditional expectation, we can write $\mathbb{E}[R^k | \mathcal{F}^{k-1}] = 1 \cdot \prob[ R^k = 1 | \mathcal{F}^{k-1}] + 0 \cdot \prob[ R^k = 0 | \mathcal{F}^{k-1} ] \geq \tau_{lb}$. This shows that \eqref{ineq: cond_exp>tau} in \Cref{ass:lowerbdRPK} holds. 
\end{proof}

We now estimate the rate of convergence of X-REGO for Lipschitz continuous functions using the estimates for $\tau$ provided in \Cref{cor:t(r)_asymp}.

\begin{theorem}\label{thm:conv_rate_gen_f}
	Suppose that Assumptions \ref{ass:f_is_Lipschitz} and \ref{assump: prob_of_I_R>rho} hold, that there exists a global minimizer $\mvec{x}^*$ of \eqref{eq: GO} that satisfies \Cref{ass:B_eps_is_a_ball} (replacing $\epsilon$ by $\epsilon-\lambda$ in \Cref{ass:B_eps_is_a_ball}), and that $\tilde{\mvec{p}}^k$ satisfies $\rVert \tilde{\mvec{p}}^k - \mvec{x}^*\rVert \leq R_{\max}$ for all $k$ and for some suitably chosen $R_{\max}$. Suppose also that $d^k \geq d_{lb}$ for some $d_{lb}>0$. Then, $\mvec{x}_{opt}^k$ defined in \eqref{eq: xoptk} converges to the set of $\epsilon$-minimizers $G_{\epsilon}$ almost surely as $k \to \infty$, and 
	\[\text{$\prob[ \mvec{x}^k_{opt} \in G_{\epsilon} ]= \prob[f(\mvec{x}^k_{opt}) \leq f^* + \epsilon]\geq \xi$ for all $k \geq K_{\xi}$,}\]
	with
	\begin{equation}
		\text{$K_\xi = \frac{\left|\log(1-\xi)\right|}{\rho_{lb}} O\left(D^{\frac{2-d_{lb}}{2}} \left( \frac{L R_{\max}}{\epsilon-\lambda} \right)^{D-d_{lb}}\right)$ as $D \rightarrow \infty$.}
	\end{equation}
\end{theorem}
\begin{proof}
	The result follows from \Cref{thm: glconv}, \Cref{cor:glconv_gen_fun} and \Cref{cor:t(r)_asymp}.
\end{proof}

\subsection{Ensuring boundedness of $\tilde{\mvec{p}}^k$}
So far, our convergence results rely on the assumption that, for each $k$, $\rVert \tilde{\mvec{p}}^k-\mvec{x}^*\rVert \leq R_{\max}$ for some suitably chosen $R_{\max}$ and for some global minimizer $\mvec{x}^*$ surrounded by a ball of radius $(\epsilon-\lambda)$ of feasible solutions, see \Cref{ass:B_eps_is_a_ball}. We show in this section that the following strategies for choosing the random variable $\mvec{p}^k$ guarantee that $\mvec{x}_{opt}^k$ defined in \eqref{eq: xoptk} converges to the set of $\epsilon$-minimizers $G_{\epsilon}$ almost surely as $k \to \infty$.

\begin{enumerate}
    \item $\mvec{p}^k$ is deterministic and does not vary with $k$ (e.g., $\mvec{p}^k = \mvec{0}$ for all $k$). \label{case:pfixed}
    \item $(\mvec{p}^k)_{k = 1, 2, \dots}$ is a bounded sequence of deterministic values. \label{case:pseq}
    \item $\mvec{p}^k$ is any random variable with support contained in $\mathcal{X}$, and $\mathcal{X}$ is bounded. \label{case:pbounded}
     \item $\mvec{p}^{k}$ is a random variable defined as $\mvec{p}^{k} = \mvec{x}^k_{opt}$, where $\mvec{x}_{opt}^k$ is the best point found over the $k$ first embeddings, see \eqref{eq: xoptk}, and the objective is coercive. \label{case:pbest}
\end{enumerate}
       
Note that for the strategies \ref{case:pfixed}, \ref{case:pseq} and \ref{case:pbounded}, the validity of \Cref{thm:conv_rate_gen_f} follows simply from the triangular inequality: 
\[ \| \tilde{\mvec{p}}^k - \mvec{x}^* \| \leq \| \tilde{\mvec{p}}^k \| + \| \mvec{x}^* \| < \infty, \]
and the fact that $ \| \tilde{\mvec{p}}^k \|$ is bounded. We prove next that $\mvec{x}_{opt}^k$ defined in \eqref{eq: xoptk} converges to the set of $\epsilon$-minimizers $G_{\epsilon}$ almost surely as $k \to \infty$ for strategy \ref{case:pbest} if the objective is coercive. 

\begin{assump}[Coerciveness, see \cite{beck2014}] \label{ass:coerciveness}
    When $\mathcal{X}$ is unbounded, the (continuous) function $f : \mathcal{X} \to \mathbb{R}$ in \eqref{eq: GO} satisfies
    \begin{equation}
        \lim_{\| \mvec{x} \| \to \infty} f(\mvec{x}) = \infty.
    \end{equation}
\end{assump}

\begin{corollary} \label{cor:coerciveness_implies_bounded_p}
    Let \Cref{ass:coerciveness} hold, and let $\mvec{x}^*$ be a global minimizer of \eqref{eq: GO}. Let $\mvec{p}^k = \mvec{x}_{opt}^k$ for $k \geq 1$, with $\mvec{x}_{opt}^k$ defined in \eqref{eq: xoptk}, and let $\mvec{p}^0 \in \mathcal{X}$ be such that $f(\tilde{\mvec{p}}^0)<\infty$. Then, there exists $R_{\max}$ such that, for all $k$,
    \begin{equation}\label{eq:bounded_dist_corol}
        \| \tilde{\mvec{p}}^k - \mvec{x}^* \| \leq R_{\max}.
    \end{equation}
\end{corollary}

\begin{proof}
Note that the sequence $(f(\tilde{\mvec{p}}^k))_{k = 0, 1, 2, \dots}$ is decreasing by definition of the random variable $\mvec{x}_{opt}^k$. Therefore, for all $k$ there holds 
\[ f(\tilde{\mvec{p}}^k) \leq f(\tilde{\mvec{p}}^0) < \infty. \]
By coerciveness of $f$, there exists $R < \infty$ such that for any deterministic vector $\mvec{y} \in \mathcal{X}$, $\| \mvec{y} \| > R$ implies $f(\mvec{y}) > f(\tilde{\mvec{p}}^0)$. We deduce that $\| \tilde{\mvec{p}}^k \| < R$ for all $k$, so that $\| \tilde{\mvec{p}}^k - \mvec{x}^*\| \leq \| \tilde{\mvec{p}}^k \| + \| \mvec{x}^* \| \leq R + \| \mvec{x}^*\|$. The result follows by writing $R_{\max} = R + \| \mvec{x}^* \|$.
\end{proof}

\begin{corollary}
	Suppose that Assumptions \ref{ass:f_is_Lipschitz}, \ref{assump: prob_of_I_R>rho} and \ref{ass:coerciveness} hold, that there exists a global minimizer $\mvec{x}^*$ of \eqref{eq: GO} that satisfies \Cref{ass:B_eps_is_a_ball} (replacing $\epsilon$ by $\epsilon-\lambda$ in \Cref{ass:B_eps_is_a_ball}), and that $d^k \geq d_{lb}$ for some $d_{lb}>0$.  Let $\mvec{p}^k = \mvec{x}_{opt}^k$ for $k \geq 1$, with $\mvec{x}_{opt}^k$ defined in \eqref{eq: xoptk}, and let $\mvec{p}^0 \in \mathcal{X}$ be such that $f(\tilde{\mvec{p}}^0)<\infty$. Then, $\mvec{x}_{opt}^k$ converges to the set of $\epsilon$-minimizers $G_{\epsilon}$ almost surely as $k \to \infty$, and  there exists $R_{\max}$ such that
	\[\text{$\prob[ \mvec{x}^k_{opt} \in G_{\epsilon} ]= \prob[f(\mvec{x}^k_{opt}) \leq f^* + \epsilon]\geq \xi$ for all $k \geq K_{\xi}$,}\]
	with
	\begin{equation}
		\text{$K_\xi = \frac{\left|\log(1-\xi)\right|}{\rho_{lb}} O\left(D^{\frac{2-d_{lb}}{2}} \left( \frac{L R_{\max}}{\epsilon-\lambda} \right)^{D-d_{lb}}\right)$ as $D \rightarrow \infty$.}
	\end{equation}
\end{corollary}

\begin{proof}
	The result follows from \Cref{thm:conv_rate_gen_f} and \Cref{cor:coerciveness_implies_bounded_p}.
\end{proof}

\section{Applying X-REGO to functions with low effective dimensionality}
\label{sec: loweffdim}

The recent works \cite{Cartis2020, Cartis2020b} explore random embedding algorithms for functions with low effective dimension, that only vary over a subspace of dimension $d_e<D$, and address respectively the case $\mathcal{X} = \mathbb{R}^D$ and $\mathcal{X} = [-1,1]^D$. Both papers assume that the dimension of the random subspace $d$ in \eqref{eq: AREGO} is the same or exceeds  the effective dimension $d_e$, and derive bounds on the probability of \eqref{eq: AREGO} to be $\epsilon$-successful in that setting; these bounds are then used to prove convergence of respective random subspace methods. For the remainder of this paper, we explore the use of X-REGO for unconstrained global optimization of functions with low effective dimension, for any random subspace dimension $d$, thus removing the assumption  $d \geq d_e$. To prove convergence of X-REGO in that setting, we rely on the results derived in \Cref{sec:est_prob_eps_success_RPX_gen_fun}. 

\subsection{Definitions and existing results}
\begin{definition}[Functions with low effective dimensionality, see \cite{Wang2016}] \label{def: simple_function}
	A function $f : \mathbb{R}^D \rightarrow \mathbb{R}$ has effective dimension $d_e < D$  if there exists a linear subspace $\mathcal{T}$ of dimension $d_e$ such that for all vectors $\mvec{x}_{\top}$ in $\mathcal{T}$ and $\mvec{x}_{\perp}$ in $\mathcal{T}^{\perp}$ (the orthogonal complement of $\mathcal{T}$), we have
	\begin{equation}\label{eq: def_fun_eff_dim}
		f(\mvec{x}_{\top} + \mvec{x}_{\perp}) = f(\mvec{x}_{\top}),
	\end{equation}
	and $d_e$ is the smallest integer satisfying \eqref{eq: def_fun_eff_dim}.
\end{definition}

The linear subspaces $\mathcal{T}$ and $\mathcal{T}^{\perp}$ are respectively named the \textit{effective} and \textit{constant} subspaces of $f$. In this section, we make the following assumption on the function $f$. 

\begin{assumpA}{LowED}\label{ass:AREGO_fun_eff_dim}
	The function $f:\mathbb{R}^D \rightarrow \mathbb{R}$ has effective dimensionality $d_e $ with effective subspace\footnote{Note that $\mathcal{T}$ in \Cref{ass:AREGO_fun_eff_dim} may not be aligned with the standard axes. } $\mathcal{T}$ and constant subspace $\mathcal{T}^{\perp}$ spanned by the columns of the orthonormal matrices $\mtx{U} \in \mathbb{R}^{D \times d_e}$ and $\mtx{V} \in \mathbb{R}^{D\times(D-d_e)}$, respectively. We write $\mvec{x}_\top = \mtx{U} \mtx{U}^T \mvec{x}$ and $\mvec{x}_\perp = \mtx{V} \mtx{V}^T \mvec{x}$, the unique Euclidean projections of any vector $\mvec{x} \in \mathbb{R}^D$ onto $\mathcal{T}$ and $\mathcal{T}^\perp$, respectively.
\end{assumpA}

As discussed in \cite{Cartis2020b}, functions with low effective dimension have the nice property that their global minimizers are not isolated: to any global minimizer $\mvec{x}^*$ of \eqref{eq: GO}, with Euclidean projection $\mtx{x}_\top^*$ on the effective subspace $\mathcal{T}$, one can associate a subspace $\mathcal{G}^*$ on which the objective reaches its minimal value. Indeed, writing
\begin{equation}  \label{eq:calGstar}
    \mathcal{G}^* = \{ \mvec{x}^*_\top + \mtx{V} \mvec{h}: \mvec{h} \in \mathbb{R}^{D - d_e} \},
\end{equation}  
\Cref{ass:AREGO_fun_eff_dim} implies that $f(\mvec{x}) = f^*$ for all $\mvec{x} \in \mathcal{G}^*$.

In the case $d \geq d_e$, the following result, derived in \cite{Wang2016}, says that the reduced problem \eqref{eq: AREGO} is successful with probability one.

\begin{theorem}(see \cite[Theorem 2]{Wang2016}, and \cite[Rem. 2.22]{Otemissov2021}) \label{thm:Wang}
Let $\mathcal{X} = \mathbb{R}^D$ and let \Cref{ass:AREGO_fun_eff_dim} hold. Let $\mtx{A}$ be a $D \times d$ Gaussian
matrix with $d \geq d_e$, and let $\mvec{p} \in \mathbb{R}^D$. Then, with probability one, for any fixed $\mvec{x} \in \mathbb{R}^D$, there exists a $\mvec{y} \in \mathbb{R}^d$ such that $f(\mvec{x}) = f(\mtx{A} \mvec{y} + \mvec{p})$. In particular, for any global minimizer $\mvec{x}^*$ of \eqref{eq: GO}, with probability one, there exists $\mvec{y}^* \in \mathbb{R}^d$ such that $f(\mtx{A} \mvec{y}^* + \mvec{p}) = f(\mvec{x}^*) = f^*$.
\end{theorem}
Thus, in the unconstrained case $\mathcal{X} = \mathbb{R}^D$, the solution of a single reduced problem \eqref{eq: AREGO} with subspace dimension $d \geq d_e$ provides an exact global minimizer of the original problem \eqref{eq: GO} with probability one. In the next section, we address the case $d < d_e$.

\subsection{Probability of success of the reduced problem for lower dimensional embeddings}

Unfortunately, \Cref{thm:Wang} crucially depends on the assumption $d \geq d_e$. When $d < d_e$, we quantify the probability of the random embedding to contain a (global) $\epsilon$-minimizer. Similarly to the definition of $\mathcal{G}^*$ above, one may associate to any global minimizer $\mvec{x}^*$ of \eqref{eq: GO} a connected set $\mathcal{G}_\epsilon^*$ of $\epsilon$-minimizers. Denoting   the Euclidean projection of $\mvec{x}^*$ on the effective subspace by $\mvec{x}_\top^*$, under \Cref{ass:f_is_Lipschitz} (Lipschitz continuity of $f$), $\mathcal{G}_\epsilon^*$ is the Cartesian product of a $d_e$-dimensional ball (contained in the effective subspace) by the constant subspace $\mathcal{T}^\perp$ (see \Cref{ass:AREGO_fun_eff_dim}):
\begin{equation} \label{eq: calG_epsstar}
    \mathcal{G}_\epsilon^* := \{ \mvec{x}^*_\top + \mtx{U} \mvec{g} + \mtx{V} \mvec{h} : \mvec{g} \in \R^{d_e}, \rVert \mvec{g} \rVert \leq \epsilon/L, \mvec{h} \in \R^{D-d_e}\},
\end{equation}
where  $L$ is the Lipschitz constant of $f$. Indeed, let $\mvec{x}  := \mvec{x}^*_\top + \mtx{U} \mvec{g} + \mtx{V} \mvec{h} \in \mathcal{G}_{\epsilon}^*$, for some $\mvec{g}\in \R^{d_e}$ satisfying $\rVert\mvec{g}\rVert \leq \epsilon/L$ and for some  $\mvec{h}\in \R^{D-d_e}$. Then, 
$f(\mvec{x}) = f(\mvec{x}^*_\top + \mtx{U} \mvec{g})$ by \Cref{ass:AREGO_fun_eff_dim}, since $\mtx{V} \mvec{h} \in \mathcal{T}^\perp$. By Lipschitz continuity of $f$, we get: 
\begin{equation} \label{eq:calGepsstar_approx_min}
    f(\mvec{x}) = f(\mvec{x}^*_\top + \mtx{U} \mvec{g}) \leq f(\mvec{x}^*_\top) + L \| \mtx{U} \mvec{g} \| \leq f^* + \epsilon. 
\end{equation}

As already discussed in \Cref{sec: prelim}, the reduced problem \eqref{eq: AREGO} is $\epsilon$-successful if the random subspace $\mvec{p}+\range(\mtx{A})$ intersects the set of approximate global minimizers, which by \Cref{lem:Gstarteps_contained_in_Geps} contains any connected components $\mathcal{G}_\epsilon^*$ defined in \eqref{eq: calG_epsstar} for some global minimizer $\mvec{x}^*$ of \eqref{eq: GO}. \Cref{fig:illustr_low_effective_dim} shows an abstract representation of the situation where the random subspace $\mvec{p}+\range(\mtx{A}_1)$  intersects the connected component $\mathcal{G}_\epsilon^*$, the corresponding embedding is therefore $\epsilon$-successful; conversely, the random subspace $\mvec{p}+\range(\mtx{A}_2)$ does not intersect $\mathcal{G}_\epsilon^*$. If $\mathcal{G}_\epsilon^* = G_\epsilon$ defined in \eqref{eq: G_epsilon}, this implies that the corresponding embedding is not $\epsilon$-successful. 

\begin{figure}[h!]
    \centering
    \includegraphics[scale = 0.6,trim = 0  0 0 0, clip]{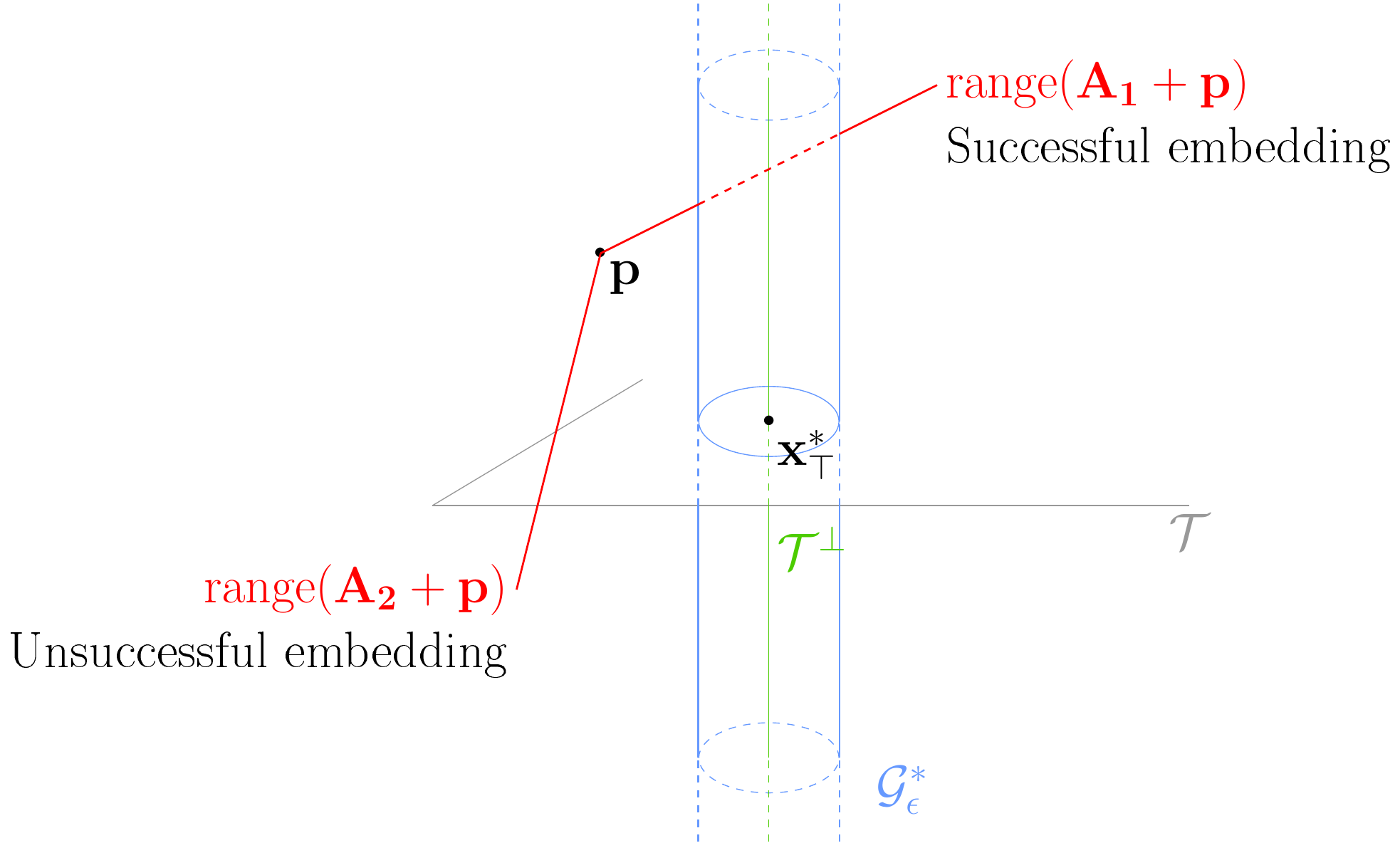}
    \caption{Abstract illustration of  embeddings for functions with low effective dimension. The reduced problem is $\epsilon$-successful if the random subspace intersects the connected component $\mathcal{G}_{\epsilon}^*$.}
    \label{fig:illustr_low_effective_dim}
\end{figure}

The following result further characterizes the probability of success of \eqref{eq: AREGO}.

\begin{theorem} \label{lem:Gstarteps_contained_in_Geps}
Let $\mathcal{X} = \mathbb{R}^D$, and let Assumptions \ref{ass:f_is_Lipschitz} and \ref{ass:AREGO_fun_eff_dim} hold. Let $\mtx{A}$ be a $D \times d$ Gaussian matrix, $\mvec{p} \in \R^D$ be a fixed vector, $\epsilon>0$ an accuracy tolerance and $\mvec{x}^*$ any global minimizer of \eqref{eq: GO} with associate connected component $\mathcal{G}_\epsilon^*$ as in \eqref{eq: calG_epsstar}. Then, 
\begin{align*}
    \prob[\eqref{eq: AREGO} \text{ is $\epsilon$-successful}] &\geq \prob[\mvec{p}+\range(\mtx{A}) \cap \mathcal{G}^*_{\epsilon} \neq \varnothing], \\
    &= \prob[ \mtx{U}^T \mvec{p} + \mathrm{range}(\mtx{B}) \cap B_{\epsilon/L}(\mtx{U}^T\mvec{x}^*) \neq \varnothing], 
\end{align*} 
where $\mtx{U}$ is an orthonormal matrix whose columns span the effective subspace $\mathcal{T}$ (see \Cref{ass:AREGO_fun_eff_dim}), $\mtx{B} := \mtx{U}^T \mtx{A}$, a $d_e  \times d$ Gaussian matrix and $B_{\epsilon/L}(\mtx{U}^T\mvec{x}^*)$, the $d_e$-dimensional ball of radius $\epsilon/L$ centered at $\mtx{U}^T\mvec{x}^*$.
\end{theorem} 

\begin{proof}
The first inequality simply follows from \eqref{eq:RPX_eps_succ=subspace_cap_G_eps} and from the fact that $\mathcal{G}_\epsilon^* \subseteq G_\epsilon$, see \eqref{eq:calGepsstar_approx_min}. 
For the second relationship, since the matrix $\mtx{Q} := \begin{bmatrix} \mtx{U} & \mtx{V} \\ \end{bmatrix}$ (with $\mtx{V}$ defined in \Cref{ass:AREGO_fun_eff_dim}) is orthogonal, for all $\mvec{y} \in \mathbb{R}^{d}$
\begin{equation*}
    \mtx{A} \mvec{y} + \mvec{p}  = \mtx{Q} \mtx{Q}^T (\mtx{A} \mvec{y} + \mvec{p}) =  \begin{bmatrix} \mtx{U} & \mtx{V} \\ \end{bmatrix} \begin{bmatrix} \mtx{U}^T \\ \mtx{V}^T \\ \end{bmatrix} (\mtx{A} \mvec{y} + \mvec{p})  = (\mtx{U}\mtx{U}^T + \mtx{V}\mtx{V}^T) (\mtx{A} \mvec{y} + \mvec{p}).
\end{equation*}
Writing $\mtx{B} := \mtx{U}^T \mtx{A} \in \R^{d_e \times d}$ and $\mtx{C} := \mtx{V}^T \mtx{A} \in \R^{(D-d_e)  \times d}$, we get for any global minimizer $\mvec{x}^*$  of \eqref{eq: GO} (with associate Euclidean projection $\mvec{x}_\top^*$ on the effective subspace)
\begin{align}
    \mtx{A} \mvec{y} + \mvec{p}  &= \mtx{U} (\mtx{B} \mvec{y} + \mtx{U}^T \mvec{p}) + \mtx{V} (\mtx{C} \mvec{y} + \mtx{V}^T \mvec{p}) \\ 
     &= \mvec{x}_\top^* + \mtx{U} (\mtx{B} \mvec{y} + \mtx{U}^T \mvec{p} - \mtx{U}^T \mvec{x}^*) + \mtx{V} (\mtx{C} \mvec{y} + \mtx{V}^T \mvec{p}). \label{eq: subspace_dec}
\end{align}  
By definition of $\mathcal{G}_\epsilon^*$, there follows that $\mtx{A} \mvec{y} + \mvec{p} \in \mathcal{G}_\epsilon^*$ if and only if $ \mtx{B} \mvec{y} + \mtx{U}^T \mvec{p} \in B_{\epsilon/L}(\mtx{U}^T \mvec{x}^*)$. By \Cref{thm: orthog_inv_of_Gaussian_matrices}, $\mtx{B}$ is a $d_e \times d$ Gaussian matrix, which completes the proof.
\end{proof}

The probability of $\epsilon$-success of \eqref{eq: AREGO} can thus be lower bounded by the probability of the $d$-dimensional random subspace $\range(\mtx{B})+ \mtx{U}^T \mvec{p}$ intersecting the  ball $B_{\epsilon/L}(\mtx{U}^T \mvec{x}^*)$ in $\R^{d_e}$. We now estimate the latter probability using the conic integral geometry results presented in \Cref{sec: Conic_integral_geomtry} and \Cref{sec:est_prob_eps_success_RPX_gen_fun}: the two next results can be seen as the immediate counterparts of \Cref{thm:RPX_eps_succ>t(r_p)} and \Cref{thm:RPX_eps_succ_gen_uniform_bound} for functions with low effective dimensionality. 

\begin{corollary} \label{thm:LEDRPX_eps_succ>t(r_p)}
Let $\mathcal{X} = \mathbb{R}^D$, and suppose that Assumptions \ref{ass:f_is_Lipschitz} and \ref{ass:AREGO_fun_eff_dim} hold, with effective dimension $d_e > d$. Let $\epsilon>0$ be an accuracy tolerance, $\mtx{A}$, a $D \times d$ Gaussian matrix and $\mvec{p} \in \mathbb{R}^D \setminus G_{\epsilon}$, a given vector.  Let $r_{\mvec{p}}^{\eff} := \epsilon / \left( L \rVert \mtx{U}^T \mvec{x}^* - \mtx{U}^T \mvec{p} \rVert \right)$, where $\mvec{x}^*$ is any global minimizer of \eqref{eq: GO}. Then
\begin{equation} \label{eq:LEDRPX_eps_succ>tau_p}
    \prob[\eqref{eq: AREGO} \text{ is $\epsilon$-successful}] \geq \tau(r_{\mvec{p}}^{\eff},d,d_e), 
\end{equation} 
where the function $\tau(r,d,d_e)$ for $0<r<1$ is defined in \eqref{def:tau_p}. 
\end{corollary}

\begin{proof}
The result is a direct extension of the analysis made in \Cref{sec:est_prob_eps_success_RPX_gen_fun}, and more precisely, \Cref{thm:RPX_eps_succ>t(r_p)}, replacing $\mtx{A}$ by $\mtx{B}$, $\mvec{x}^*$ by $\mtx{U}^T \mvec{x}^*$, $\mvec{p}$ by $\mtx{U}^T \mvec{p}$ and $D$ by $d_e$. 
\end{proof}

Similarly as \Cref{thm:RPX_eps_succ_gen_uniform_bound}, the next result provides a uniform lower bound on the probability of $\epsilon$-success of \eqref{eq: AREGO}. 

\begin{corollary} \label{thm:LEDRPX_eps_succ_gen_uniform_bound}
Let $\mathcal{X} = \mathbb{R}^D$, and suppose that Assumptions \ref{ass:f_is_Lipschitz} and \ref{ass:AREGO_fun_eff_dim} hold, with effective dimension $d_e > d$. Let $\epsilon>0$ be an accuracy tolerance, $\mtx{A}$, a $D \times d$ Gaussian matrix, $\mvec{x}^*$, any global minimizer of \eqref{eq: GO}. Let $\mvec{p} \in \mathbb{R}^D$ be a given vector that satisfies $\|\mtx{U}^T \mvec{p} - \mtx{U}^T \mvec{x}^*\| \leq R_{\max}$, for some suitably chosen $R_{\max}$, and let $r_{\min}^{\eff} := \epsilon / \left( L R_{\max} \right)$. Then
\begin{equation} \label{eq:LEDRPX_eps_succ>tau}
    \prob[\eqref{eq: AREGO} \text{ is $\epsilon$-successful}] \geq \tau(r_{\min}^{\eff},d,d_e), 
\end{equation} 
where the function $\tau(r,d,d_e)$ for $0<r<1$ is defined in \eqref{def:tau_p}. 
\end{corollary}

\begin{proof}
The result is a mere adaptation of \Cref{thm:RPX_eps_succ_gen_uniform_bound}, replacing $\mtx{A}$ by $\mtx{B}$, $\mvec{x}^*$ by $\mtx{U}^T \mvec{x}^*$, $\mvec{p}$ by $\mtx{U}^T \mvec{p}$ and $D$ by $d_e$. 
\end{proof}

Note that adding some constraints (setting $\mathcal{X} \subset \mathbb{R}^D$) makes the analysis much more complicated as even if a random subspace $\{\mvec{p}+\range(\mtx{A})\}$ intersects $\mathcal{G}_{\epsilon}^*$, this intersection may be outside the feasible domain; we therefore restrict ourselves to the unconstrained case in this paper.

\subsection{X-REGO for functions with low effective dimension}
We present an X-REGO variant dedicated to the optimization of functions with low effective dimension. This algorithm starts by exploring an embedding of low dimension $d_{lb}$, assuming $d_{lb} \leq d_e$, and the dimension is progressively increased until capturing the effective dimension of the problem, see \Cref{alg:LEDXREGO}. Note that \Cref{line:drawA} to \Cref{line:choosep} are exactly the same as in \Cref{alg:XREGO}. Recall that \Cref{thm:Wang} guarantees that the algorithm finds the global minimum of \eqref{eq: GO} with probability one if the reduced problem is solved exactly and if $d^k \geq d_e$, so that in this ideal case we can terminate the algorithm after $d_e - d_{lb} + 1$ random embeddings; thus, \Cref{alg:LEDXREGO} terminates in finitely many random embeddings. Since the effective dimension is unknown, we typically terminate the algorithm when no progress is observed in the objective value, see \Cref{sec: Numerics} for numerical illustrations.

\begin{algorithm}[h!]
	\caption{X-REGO for \eqref{eq: GO} when $f$ has low effective dimension}
	\label{alg:LEDXREGO}
	\begin{algorithmic}[1]
		\State Initialize $d^1 = d_{lb}$ for some $d_{lb}\geq 1$ and $\tilde{\mvec{p}}^0 \in \mathcal{X}$
		\For{\text{$k \geq 1$ until termination}} \label{termination}
		\State Run lines \ref{line:drawA} to \ref{line:choosep} in \Cref{alg:XREGO}.
		\State Let $d^{k+1} = d^k +1$.
		\EndFor
	\end{algorithmic}
\end{algorithm}

\subsection{Convergence of X-REGO for functions with low effective dimension}

Similarly to \Cref{sec: XREGO} and \Cref{sec:convergence}, for each $k$, $\mvec{p}^k$ is a random variable. The particular case of a deterministic $\mvec{p}^k$ is represented using a random variable whose support is a singleton. To prove convergence of \Cref{alg:LEDXREGO} to an $\epsilon$-minimizer while allowing the reduced problems to be solved approximately, we again require the reduced problems to be $(\epsilon-\lambda)$-successful, see \Cref{assump: prob_of_I_R>rho} and  \Cref{ass:lowerbdRPK}. Note that unlike \Cref{sec:convergencegenf}, the results below are finite termination results, as we known that with an ideal solver, \Cref{alg:LEDXREGO} finds an $\epsilon$-minimizer after at most $d_e-d_{lb}+1$ embeddings. Let us first show that \Cref{ass:lowerbdRPK} holds, and derive the value of $\tau_{lb}$. 

\begin{corollary}\label{cor:glconv_LED}
    Suppose that $\mathcal{X} = \mathbb{R}^D$, that Assumptions \ref{ass:f_is_Lipschitz} and \ref{ass:AREGO_fun_eff_dim} hold, that $\tilde{\mvec{p}}^k$ satisfies $\rVert \mtx{U}^T \tilde{\mvec{p}}^k - \mtx{U}^T \mvec{x}^*\rVert \leq R_{\max}$ for all $k$ and for some suitably chosen $R_{\max}$ and that $d_{lb} < d_e$.  Then, \Cref{ass:lowerbdRPK} holds with 
	\[\tau_{lb} = \tau(r_{\min}^{\eff},d_{lb},d_e),\]
	with $r_{\min}^{\eff} = (\epsilon-\lambda)/(L R_{\max})$ and $\tau(\cdot,\cdot,\cdot)$ defined in \eqref{def:tau_p}.
\end{corollary}
\begin{proof}
For all embeddings such that $d^k < d_e$, the proof is the same as for \Cref{cor:glconv_gen_fun}, replacing $D$ by $d_e$ and $r_{\min}$ by $r_{\min}^{\eff}$. Note that if $d^k \geq d_e$, \eqref{eq: AREGO} is successful with probability one according to \Cref{thm:Wang}. The result follows then simply from the fact that $1 \geq \tau(r_{\min},d_{lb},d_e) = 2 v_{d_e-d_{lb}+1}$ (see the Gauss-Bonnet formula \eqref{eq:Gauss-Bonnet}, and the fact that the intrinsic volumes are nonnegative).
\end{proof}

The following result proves convergence of \Cref{alg:LEDXREGO} to the set of $\epsilon$-minimizers almost surely after at most $d_e - d_{lb} +1$ random embedding. Note in particular that this convergence result has no dependency on $D$.

\begin{corollary}[Global convergence of X-REGO for functions with low effective dimension]\label{cor:glrate_LED}
	Suppose that $\mathcal{X} = \mathbb{R}^D$, that Assumptions \ref{ass:f_is_Lipschitz}, \ref{assump: prob_of_I_R>rho} and \ref{ass:AREGO_fun_eff_dim} hold, and that $\tilde{\mvec{p}}^k$ satisfies $\rVert \mtx{U}^T \tilde{\mvec{p}}^k - \mtx{U}^T \mvec{x}^*\rVert \leq R_{\max}$ for all $k$ and for some suitably chosen $R_{\max}$. Suppose also that $d_{lb} \leq d_e$, let $\epsilon > 0$ be an accuracy tolerance, and let $k_{\max} = d_e-d_{lb}+1$ be the index of the first embedding with dimension $d_e$. Then
	\[ \prob[f(\mvec{x}_{opt}^{k_{\max}}) \leq f^* + \epsilon] > \rho^{k_{\max}} \]
	where $\mvec{x}_{opt}^k$ is defined in \eqref{eq: xoptk} and $\rho^k$ is the probability of success of the solver for \eqref{prob: AREGO_subproblem_re} (see \Cref{assump: prob_of_I_R>rho}). In particular, if the reduced problem is solved exactly, then $f(\mvec{x}_{opt}^{k_{\max}}) \leq f^* + \epsilon$ with probability one. For $1 \leq k < k_{\max}$, we have
	\[ \prob[f(\mvec{x}_{opt}^{k}) \leq f^* + \epsilon] \geq 1 - (1-\rho_{lb} \tau_{lb})^k\]
	where $\tau_{lb} = \tau(r_{\min}^{\eff},d_{lb},d_e)$, with $\tau(\cdot,\cdot,\cdot)$ defined in \eqref{def:tau_p} and $r_{\min}^{\eff} = (\epsilon-\lambda)/(L R_{\max})$.
\end{corollary}
\begin{proof} 
Note that, by \Cref{cor:glconv_LED}, \Cref{thm: glconv} applies. However, since we are interested in finite termination results, we do not directly use \Cref{thm: glconv}; instead, we extract the following claim from its convergence proof, see \eqref{eq:prob[x_opt_in_G_eps>1-(1-tr)^K]}. For all $K\geq 1$,
	\begin{equation}
		\prob[\{ \mvec{x}^K_{opt} \in G_{\epsilon} \}] \geq 1 - \Pi_{k = 1}^K (1- \tau^k \rho^k). 
	\end{equation}
	It follows that 
	\begin{equation}
		\prob[\{ \mvec{x}^K_{opt} \in G_{\epsilon} \}] \geq 1 - (1- \tau_{lb} \rho_{lb})^K,
	\end{equation}
	where $\tau_{lb}$ and $\rho_{lb}$ are computed/defined in \Cref{cor:glconv_LED} and \Cref{assump: prob_of_I_R>rho}, respectively. Finally, if $K \geq k_{\max}$, it follows that $d^K \geq d_e$, so that the probability of \eqref{prob: AREGO_subproblem} to be $(\epsilon-\lambda)$-successful is equal to one according to \Cref{thm:Wang}. So, if $K \geq k_{\max}$, 
		\begin{equation}
		\prob[\{ \mvec{x}^K_{opt} \in G_{\epsilon} \}] \geq 1 - (1-\rho^{K}) \Pi_{k = 1}^{K-1} (1- \tau^k \rho^k) > 1 - (1-\rho^{K}),
	\end{equation}
which concludes the proof.

\end{proof}

\section{Numerical experiments}
\label{sec: Numerics}

Let us illustrate the behavior of X-REGO on a set of benchmark global optimization problems whose objectives have low effective dimension. We show empirically that \Cref{alg:LEDXREGO} simultaneously manages to accurately estimate the effective dimension of the problem, and outperforms significantly (and especially in the high-dimensional regime) the no-embedding framework, in which the original problem \eqref{eq: GO} is solved directly, with no exploitation of the special structure.

\subsection{Setup}

\paragraph*{Test set.} Our synthetic test set is very similar to the one we used in \cite{Cartis2020, Cartis2020b}, and contains a set of benchmark global optimization problems adapted to have low effective dimensionality in the objective as explained in \Cref{app:prob}. Our test set is made of 18 $D$-dimensional functions with low effective dimension, with $D =10, 100$ and $1000$. These $D$-dimensional functions are constructed from 18 low-dimensional global optimization problems with known global minima (some of which are in the Dixon-Szego test set \cite{Dixon-Szego1975}), by artificially adding coordinates and then applying a rotation so that the effective subspace is not aligned with the coordinate axes.

\paragraph*{Solver.} The reduced problems are solved using the KNITRO solver (\cite{Byrd2006}). Note that, by default, KNITRO is a local solver, but switches to a global solver by activating its multistart feature. We therefore consider three ``KNITRO''-type solvers: local KNITRO (no multistart used, referred to as KNITRO), and multistart KNITRO with a low/high number of starting points (cheap or expensive versions of multistart KNITRO, referred to here as ch-mKNITRO and exp-mKNITRO, respectively). The higher the number of starting points, the more likely the solver is to find a global minimizer of the reduced problem. See \Cref{table: solvers_descriptions} for a detailed description of the settings of the different solvers. 

\begin{table}[!h]
	\centering
	\footnotesize
	\caption{The table outlines the experimental setup for the solvers, used both in the `no embedding' algorithm and for solving the low-dimensional problem \eqref{prob: AREGO_subproblem_re}.}
	\label{table: solvers_descriptions}
	\begin{tabular}{p{2.8cm}|p{2.8cm}|p{4.7cm}|p{2.8cm}}
		&  exp-mKNITRO  & ch-mKNITRO &  KNITRO \\ \hline
		 \mbox{Measure of} computational cost
		 & function evaluations & function evaluations  & function evaluations \\ \hline
		 \mbox{Max. budget} to solve \eqref{prob: AREGO_subproblem_re} & $\min(200,10d^k)$ starting points & $\min(100,2d^k)$ starting points & $1$ starting point\\ \hline
		 \mbox{Max. budget} to solve  \eqref{eq: GO} & $\min(200,10D)$ starting points (only used for the \emph{no-embedding} framework) & $\min(100,2D)$ starting points (only used for the \emph{no-embedding} framework) & 1 starting point (only used for the \emph{no-embedding} framework)  \\ \hline
		Termination for \eqref{prob: AREGO_subproblem_re} and \eqref{eq: GO} & Default options (unless overwritten by additional options) & Default options (unless overwritten by additional options) & Default options (unless overwritten by additional options)  \\ \hline
		 Additional options for \eqref{prob: AREGO_subproblem_re} and \eqref{eq: GO} & $\verb|ms_enable|=1$, $\verb|ms_bndrange|=2$  & $\verb|ms_enable|=1$, $\verb|ms_bndrange|=2$,  $\verb|ms_maxsolves|=\min(100,2d^k)$ (for \eqref{prob: AREGO_subproblem_re}), $\verb|ms_maxsolves|=\min(100,2D)$ (for \eqref{eq: GO}) &  / \\ 
	\end{tabular}
\end{table}

\paragraph*{Algorithms using a global solver (ch-mKNITRO and exp-mKNITRO).}  We test two different instances of the algorithmic framework presented in \Cref{alg:LEDXREGO}  against the \textit{no-embedding} framework, in which \eqref{eq: GO} is solved directly without using any random embedding and with no explicit exploitation of its special structure. For each instance, we let $d_{lb} = 1$. Since the effective dimension of the problem is assumed to be unknown, termination in \Cref{alg:LEDXREGO} is defined as the first embedding on which either stagnation is observed in the computed optimal cost of the reduced problem \eqref{prob: AREGO_subproblem_re}, or if not, until $d^k = D$. Objective stagnation is measured as follows: stop after $k_f$ embeddings, where $k_f$ is the smallest $k \geq 2$ that satisfies
\begin{equation} \label{eq: stopCrit}
	\left\rvert f(\wt{\mtx{A}}^{k} \tilde{\mvec{y}}^{k} + \tilde{\mvec{p}}^{k-1})-f(\wt{\mtx{A}}^{k-1} \tilde{\mvec{y}}^{k-1} + \tilde{\mvec{p}}^{k-2}) \right\rvert \leq \gamma = 10^{-5}.
\end{equation}
If $k_f \leq D$, we let $d_e^{\mathrm{est}} := k_f-1$ be our estimate of the effective dimension of the problem. Indeed, by \Cref{thm:Wang},  two random problems of dimension $d$ and $d+1$ with $d \geq d_e$ have the same optimal cost with probability one, so that the left-hand side of \eqref{eq: stopCrit} would be zero if the reduced problems were solved exactly (i.e., under the assumption of an ideal solver). We argue that, on the other hand, it is very unlikely that two random reduced problems of dimension $d$ and $d+1$ with $d < d_e$ have the same optimal cost.\footnote{Admittedly, when optimizing difficult functions, for example that are flat almost everywhere and very steep around the minimizer, it could happen that two successive reduced problems of dimension $d$ and $d+1$, with $d < d_e$, have the same optimal cost though none of them intersects the set of $\epsilon$-minimizers; we exclude here such pathological cases.} We therefore terminate the algorithm after either $k = k_f$ (if there exists $k_f\leq D$ satisfying \eqref{eq: stopCrit}), or else $k = D$ random embeddings. Regarding the choice of $\mvec{p}^k$, we consider two possibilities: either $\mvec{p}^k$ is a vector that does not depend on $k$, or $\mvec{p}^k$ is the best point found over the $k$ first embeddings (i.e., $\mvec{p}^k = \mvec{x}_{opt}^k$).

\paragraph{Algorithms relying on a local solver (KNITRO) and a resampling strategy.} We also compare several instances of \Cref{alg:LEDXREGO} with the no-embedding framework when the reduced problem is solved using a local solver. Note that due to the possible nonconvexity of the problem, running a local solver on the original problem is not expected to find the global minimizer; results combining the no-embedding framework with a local solver are thus only reported for comparison. Recall also that our convergence analysis requires the solver to be able to find an approximate global minimizer of the subproblem with a sufficiently high probability. We show numerically that local solvers can be used when the points $\mvec{p}^{k}$ are suitably chosen to globalize the search; we typically let $\mvec{p}^{k}$, for some indices $k$, be a random variable with a sufficiently large support to contain a global minimizer of \eqref{eq: GO}. Similarly as with global solvers, let $k_f$ be the smallest $k \geq 2$ that satisfies \eqref{eq: stopCrit}, and, if $k_f \leq D$, let $d_e^{\mathrm{est}} := k_f-1$ be our estimate of the effective dimension of the problem. However, since the solver is local, we cannot assume that \eqref{eq: stopCrit} implies that we found an approximate global minimizer of the original problem \eqref{eq: GO}. We therefore continue the optimization, fixing the subspace dimension: $d^k = d_e^{\est}$ for all $k > k_f$, and assuming that $\mvec{p}^k$ will be such that the next random subspace will leave the basin of attraction of the actual local minimizer. To prevent against local solutions, we use a stronger stopping criterion: the algorithm is stopped either after $D$ embeddings, or earlier, when $k > k_f$ and when the computed optimal cost of the reduced problem did not change significantly over the last $n_{\mathrm{stop}}$ random embeddings, i.e., if  \begin{equation}\label{eq: stopCritLoc}
    f(\mvec{x}_{opt}^{k-n_{\mathrm{stop}}+1}) -f(\mvec{x}_{opt}^k) \leq \gamma = 10^{-5}.
\end{equation}
In our experiments, we considered two possibilities: $n_{\mathrm{stop}} = 3$ or $n_{\mathrm{stop}} = 5$. Here again, we consider two main strategies for choosing $\mvec{p}^k$: either $\mvec{p}^k$ does not depend on $k$ (e.g., $\mvec{p}^k$ is an identically distributed random variable, for all $k$), or $\mvec{p}^k$ is the best point found over the past embeddings ($\mvec{p}^k = \mvec{x}_{opt}^k)$, resampling $\mvec{p}^k$ at random in a sufficiently large domain for some values of $k$, see below.

\paragraph{Summary of the algorithms:}
In total, we compare four instances of \Cref{alg:LEDXREGO}, that correspond to specific choices of  $\mvec{p}^k$, $k \geq 0$, and on the choice of a local/global solver. 
\begin{itemize}
	\item[-] Adaptive X-REGO (A-REGO). In \Cref{alg:LEDXREGO}, the reduced problem is solved using a global solver and the point $ \mvec{p}^k$ is chosen as the best point found\footnote{ If the reduced problem \eqref{prob: AREGO_subproblem_re} is solved using a global solver, then $f( \mvec{\tilde{p}}^k) \leq f(\mvec{\tilde{p}}^{k-1})$ since $\mvec{\tilde{p}}^{k-1}$ belongs to the search space of \eqref{prob: AREGO_subproblem_re}, so that we are indeed keeping the best point found so far. If we are using a local solver, we always initialize $\mvec{y} = \mvec{0}$ when solving \eqref{prob: AREGO_subproblem_re}, so that the same conclusion holds.}  up to the $k$th embedding:  $\mvec{p}^{k} := \mtx{A}^k \mvec{y}^k+\mvec{p}^{k-1}$.
	\item[-] Local adaptive X-REGO (LA-REGO). In \Cref{alg:LEDXREGO}, the reduced problem \eqref{prob: AREGO_subproblem_re} is solved using a local solver (instead of global as in A-REGO). Until we find the effective dimension (i.e., for $k < k_{f}$), we use the same update rule for $\mvec{p}^k$ as in A-REGO: $\mvec{p}^{k} := \mtx{A}^k \mvec{y}^k+\mvec{p}^{k-1}$. For the remaining embeddings, the point $\mvec{p}^k$ is chosen as follows: $\mvec{p}^k = \mtx{A}^k \mvec{y}^k + \mvec{p}^{k-1}$ if $|f(\mtx{A}^k \mvec{y}^k + \mvec{p}^{k-1}) - f(\mvec{p}^{k-1})| > \gamma = 10^{-5}$, and $\mvec{p}^k$ is draw uniformly in $[-1,1]^D$ otherwise, to compensate for the local behavior of the solver\footnote{We know, from the way we have constructed the test set, that for each problem there exists a global minimizer that belongs to $[-1,1]^D$.}.
	\item[-] Nonadaptive X-REGO (N-REGO). In \Cref{alg:LEDXREGO}, the reduced problem is solved globally, and all the random subspaces are drawn at some fixed point: $\mvec{p}^k = \mvec{a}$. The fixed value $\mvec{a}$ is simply defined as a realization of a random variable uniformly distributed in $[-1,1]^D$.\footnote{One could take simply $\mvec{p} = \mvec{0}$, but due to the way we have constructed the problem set, setting $\mvec{p} = \mvec{0}$ may give some advantage to the algorithm, so we let $\mvec{p}^k = \mvec{a}$ instead, where $\mvec{a}$ is a random variable drawn once at the beginning of the algorithm.} 
	\item[-] Local nonadaptive X-REGO (LN-REGO). In \Cref{alg:LEDXREGO}, the reduced problem \eqref{prob: AREGO_subproblem_re} is solved using a local solver. Until we find the effective dimension (i.e., for $k < k_f$), we set $\mvec{p}^{k} = \mvec{a}$, with $\mvec{a}$  as in N-REGO. For $k \geq k_f$, $\mvec{p}^k$ is a random variable distributed uniformly in $[-1,1]^D$ (and resampled at each embedding), to compensate for the local behavior of the solver.
\end{itemize}

 Note that, regarding the choice of $\mvec{p}^k$ when using a local solver, we typically have two phases. In the first phase, we apply the same selection rules for $\mvec{p}^k$, $k < k_f$, as when using a global solver. For $k \geq k_f$, we allow resampling to avoid the algorithm to be trapped at a local minimizer. We do not introduce some resampling in the first phase, because then stochasticity would impact the criterion \eqref{eq: stopCrit} and our estimate of the effective dimension of the problem.

\paragraph*{Experimental setup.} 
For each algorithm described above, we solve the entire test set 3 times to estimate the average performance of the algorithms, and record the computational cost, which we measure in terms of function evaluations (the termination criterion is described above). Note that from the four algorithms described above, we get six different algorithms, since algorithms A-REGO and N-REGO are endowed with two different global solvers: exp-mKNITRO and ch-mKNITRO, corresponding respectively to a low and large number of starting points. To compare with `no-embedding', we solve the full-dimensional problem \eqref{eq: GO} directly with the corresponding solver with no use of random embeddings. The budget and termination criteria used to solve \eqref{prob: AREGO_subproblem_re} within X-REGO or to solve \eqref{eq: GO} in the `no-embedding' framework are the default ones, summarized in \Cref{table: solvers_descriptions}.

\begin{rem}
 All the experiments were run in MATLAB on the 16 cores (2$\times$8 Intel with hyper-threading) Linux machines with 256GB RAM and 3300 MHz speed. 
\end{rem}

We present the main numerical results using \citeauthor{Dolan2002}'s performance profile \cite{Dolan2002} --- a popular framework to compare the performance of optimization algorithms applied to a given test set. For a given algorithm $\mathcal{A}$, and for each function $f$ in the test set $\mathcal{S}$, we define
$$N_f(\mathcal{A}) := \text{min. \# of fun.~evals required by the algorithm  to converge}.$$
If $\mathcal{A}$ fails to successfully converge to a $\epsilon$-minimizer of $f$, with $\epsilon  = 10^{-3}$, within the maximum computational budget, we set $N_f(\mathcal{A}) = \infty$. We further define 
$$ N_f^* := \min_{\mathcal{A}} N_f(\mathcal{A}),$$
as the minimal computational cost required by any algorithm to optimize $f$. We normalize all the computational costs by $N_f^*$ and, for each $\mathcal{A}$, we plot a function $\pi_{\mathcal{A}}(\alpha)$ that computes the proportion of $f$'s in the test set $\mathcal{S}$, for which the normalized computational effort spent by $\mathcal{A}$ was less than $\alpha$. Mathematically speaking,
$$ \text{$\pi_{\mathcal{A}}(\alpha) := \frac{|\{f: N_f(\mathcal{A}) \leq \alpha N_f^*\}|}{|\mathcal{S}|} $ for $\alpha \geq 1$,}$$
where $|\cdot|$ denotes the cardinality of a set. The algorithm $\mathcal{A}$ is considered to have achieved better performance if it produces higher values for $\pi_{\mathcal{A}}(\alpha)$ for lower values of $\alpha$, i.e., on figures, the curve $\pi_{\mathcal{A}}(\alpha)$ is higher and lefter.

\subsection{Numerical results}

\paragraph{Comparison of X-REGO with the no-embedding framework.} The comparison between the above-mentioned instances of X-REGO and the no-embedding framework is given in \Cref{fig:KNITRO}. A-REGO and N-REGO clearly outperform the no-embedding framework in terms of accuracy vs computational cost, especially for large $D$. Reducing the number of starting points in the multistart strategy (i.e., replacing exp-mKNITRO by ch-mKNITRO) allows to further significantly improve the performance, though the total proportion of problems ultimately solved is slightly decreased compared to exp-mKNITRO. Note also that the use of a local solver (LA-REGO and LN-REGO) outperforms both global X-REGO instances and the no-embedding framework, especially for large $D$. They find the global minimizer in a significantly higher number of subproblems than when directly addressing the original high-dimensional problem with the local solver: the resampling strategy for $\mvec{p}^{k}$ described above helps to globalize the search. \Cref{tab:nemb} contains the average, over the test problems, of the number of embeddings used per algorithm; note that for (approximately) global solvers, and especially using $\mvec{p}^{k} = \mvec{x}_{opt}^k$, the average number of embeddings is very close to the ideal $k_f$. Indeed, the average effective dimension on our problem sets is equal to 3.7, so the ideal average number of embeddings should be 4.7, as we need an additional embedding for the stopping criterion \eqref{eq: stopCrit} to be satisfied. For local solvers, the average number of embeddings is slightly higher due to the need to resample candidate solutions to globalize the search and due to the stronger stopping criterion.

\begin{figure}[h!]
	\centering
	\includegraphics[trim = 100 20 0 20, clip, scale=0.8]{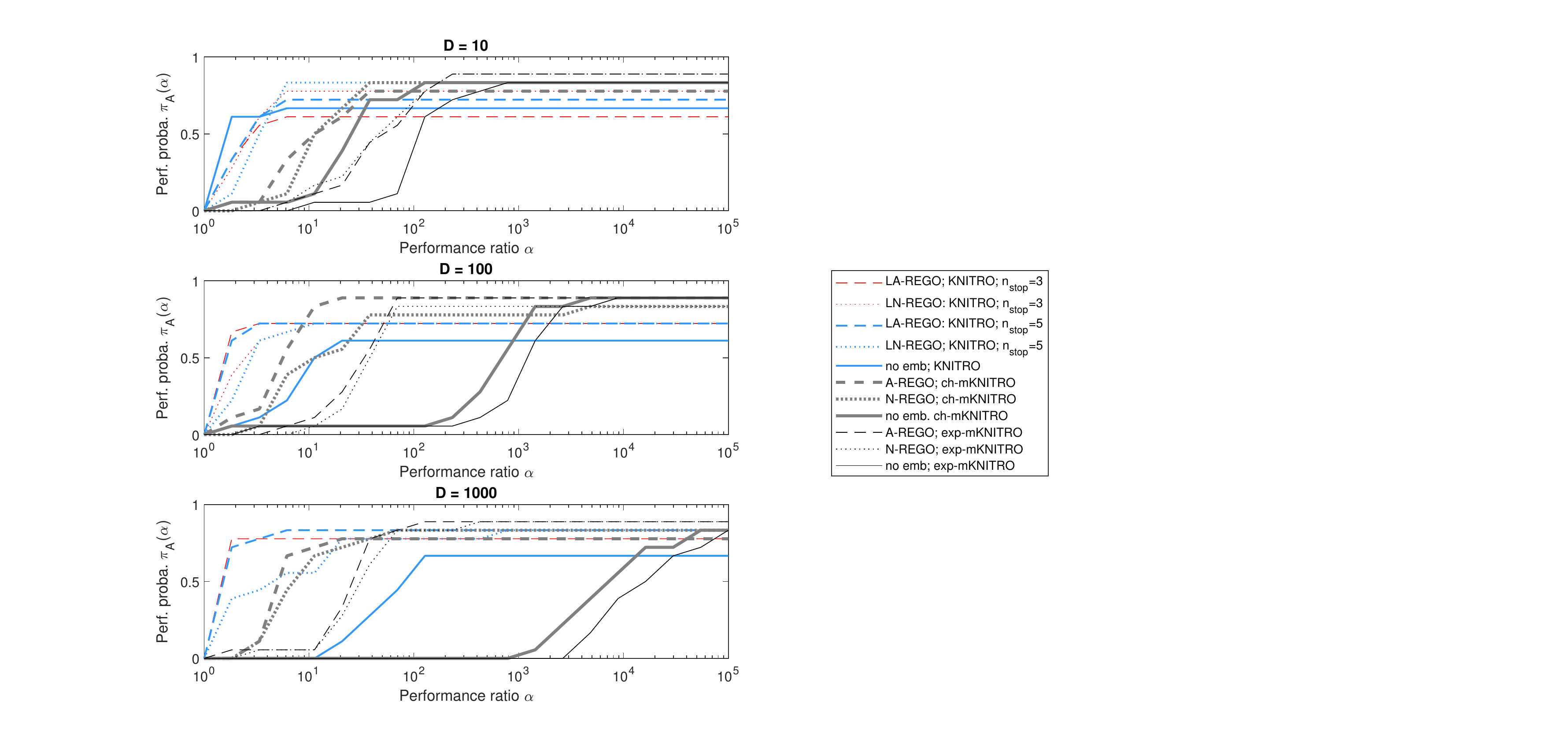}
	\caption{Comparison between the X-REGO algorithms and `no-embedding' with KNITRO. Each algorithm was run three times on the whole dataset; since all three runs returned similar curves, we display only one of them.}
	\label{fig:KNITRO}
\end{figure}

\begin{table}[h!]
	\centering
	\footnotesize
    \caption{Average number of embeddings per problem, estimated from 3 independent runs of the algorithms on the test set.}
	\label{tab:nemb}
    \begin{tabular}{c|c|c|c|c|c|c|c|c}
    & \multicolumn{2}{c|}{KNITRO $(n_{\mathrm{stop}} = 3)$} & \multicolumn{2}{c|}{KNITRO $(n_{\mathrm{stop}} = 5)$} & \multicolumn{2}{c|}{ch-mKNITRO} & \multicolumn{2}{c}{exp-mKNITRO} \\
         &  LA-REGO & LN-REGO & LA-REGO & LN-REGO & A-REGO & N-REGO & A-REGO & N-REGO\\
         \hline
        $D = 10$ &  5.96 & 6.59 & 7.87 & 8.07 & 4.78 & 5.19 & 4.70 & 4.89 \\
        $D = 100$ & 6.41 & 7.81 & 9.09 & 9.33 & 4.76 & 8.94 & 4.70 & 5.63 \\
        $D = 1000$ &6.26 & 9.15 & 8.54 & 10.83& 4.67 & 7.17 & 4.67 & 6.67 \\
    \end{tabular}
\end{table}

\paragraph{Estimation of the effective dimension.}
As described earlier, instances of X-REGO naturally provide an estimate $d_e^{\est}$ of the effective dimension of the problem: $d_e^{\est} = k_f-1$, where $k_f$ is the smallest integer that satisfies \eqref{eq: stopCrit}. In case there exists no $k_f \leq D$ satisfying \eqref{eq: stopCrit}, we set $d_e^{\est} = D$. For several instances of \Cref{alg:LEDXREGO}, Table \ref{tab:estimatingde} reports the number of problems of the data set on which $d_e^{\est} \in [d_e, d_e+2]$, where $d_e$ is the exact effective dimension of the problem, for $D = 10$, $D = 100$ and $D = 1000$. Typically, adaptive choices of $\mvec{p}^k$ results in a slightly larger estimate of the effective dimension; we also note that the use of a local solver is comparable to a global one regarding the ability of the algorithm to estimate the effective dimension on this problem set when $\mvec{p}^k$ is chosen adaptively, and significantly lower otherwise. The values given in \Cref{tab:estimatingde} have been averaged over three independent runs of our experiment, on the whole dataset, to account for randomness in the algorithms. 

\begin{table}[h!]
	\centering
	\footnotesize
    \caption{Percentage of problems for which the estimated effective dimension lies in the interval $[d_e, d_e+2]$ , where $d_e$ is the true effective dimension of the problem.}
	\label{tab:estimatingde}
    \begin{tabular}{c|c|c|c|c|c|c}
         &  LA-REGO & LN-REGO & A-REGO & N-REGO & A-REGO & N-REGO\\
         & KNITRO & KNITRO & ch-mKNITRO & ch-mKNITRO & exp-mKNITRO & exp-mKNITRO \\
         \hline
        $D = 10$ & 94.44 & 79.63 & 94.44 & 83.33 & 94.44 & 88.89 \\
        $D = 100$ & 94.44 & 68.52 & 94.44 & 79.63 & 94.44 & 85.19\\
        $D = 1000$ & 94.44 & 66.67 & 88.89 & 81.48 & 90.74 & 85.19\\
    \end{tabular}
\end{table}

\paragraph{What if we know the effective dimension of the problem?} In the favorable situation when the effective dimension $d_e$ of each problem is known, we can set $d_{lb} = d_e$ in \Cref{alg:LEDXREGO}, and theoretically, for an ideal global solver, \Cref{alg:LEDXREGO} is guaranteed to solve exactly the original problem using one embedding. \Cref{fig:1emb} explores numerically the validity of this claim. We compare several instances of X-REGO with corresponding counterparts, where the effective dimension is known. When using an (approximately) global solver (ch-mKNITRO or exp-mKNITRO), we stop \Cref{alg:LEDXREGO} after one embedding of dimension $d_e$. When the solver is local (KNITRO), we let \Cref{alg:LEDXREGO} explore several embeddings of dimension $d_e$, and stop the algorithm when \eqref{eq: stopCritLoc} is satisfied, with $n_{\mathrm{stop}} = 3$, or otherwise after 50 embeddings. \Cref{fig:1emb} shows the corresponding performance profiles, when comparing these strategies with the ones presented on \Cref{fig:illustr_low_effective_dim}, and the corresponding no-embedding algorithms. In general, and except when using local solvers, knowing $d_e$ allows to solve a significant proportion of the problems in a considerably smaller time. Admittedly, these conclusions strongly depend on the probability of the solver to be successful, i.e., of the number of starting points of the multistart procedure. Note also than in our test set, the effective dimension is typically low (average value is 3.7), which might also decrease the benefit of knowing the effective dimension and thus avoiding to explore lower-dimensional subspaces; we expect the gap between \Cref{alg:LEDXREGO} and algorithms where $d_e$ is known to increase with the effective dimension of the problem.

\begin{figure}[h!]
    \centering
    \includegraphics[trim = 100 20 0 20, clip,scale = 0.8]{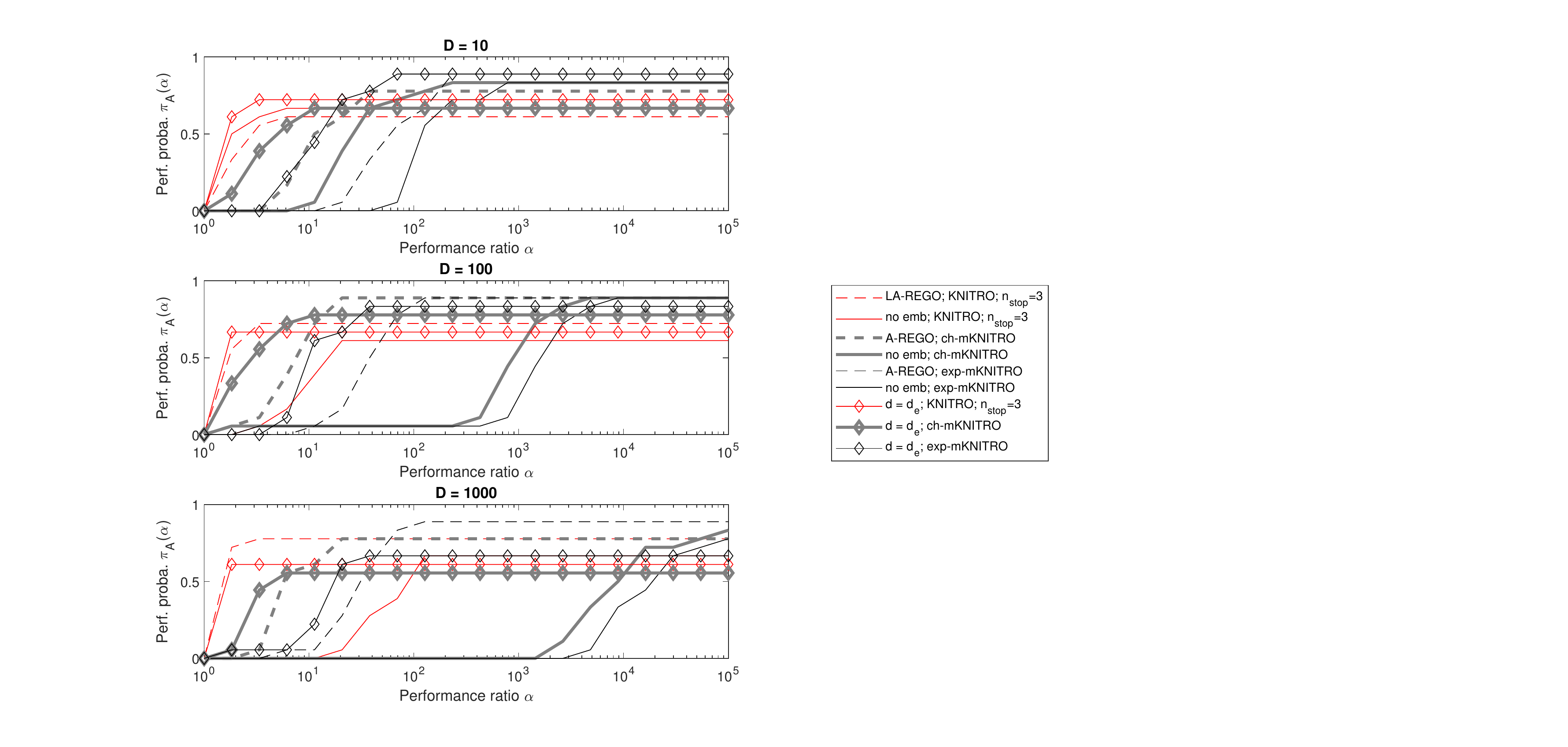}
    \caption{Comparison of X-REGO with no-embedding, and variants of X-REGO in which the random subspace dimension is equal to the effective dimension (assumed to be known).}
    \label{fig:1emb}
\end{figure}

\subsection{Conclusions to numerical experiments.}
We have compared several instances of \Cref{alg:LEDXREGO} with the no-embedding framework, where the original problem is addressed directly, with no use of random embeddings nor exploitation of the special structure. Overall, \Cref{alg:LEDXREGO} outperforms the no-embedding framework, and this observation becomes more and more apparent when the dimension of the original problem increases. We have also combined \Cref{alg:LEDXREGO} with a local solver; though our convergence theory does not cover this situation, we have shown that the resulting algorithm can outperform both the no-embedding framework and instances of \Cref{alg:LEDXREGO} relying on global solver when the parameters $\mvec{p}^k$ are sampled at random in a sufficiently large domain to ``globalize'' the search. Regarding the estimation of the effective dimension, we noticed that instances of \Cref{alg:LEDXREGO} relying on adaptive rules for selecting $\mvec{p}^k$ (A-REGO and LA-REGO) significantly outperform their fixed $\mvec{p}^{k}$ counterparts. Finally, we have shown that, in the favourable case when the effective dimension is known, letting $d_{lb} \geq d_e$ in \Cref{alg:LEDXREGO} leads to a substantial improvement in performance.

\section{Conclusions and future work}
We explored a generic algorithmic framework, X-REGO, for global optimization of Lipschitz-continuous functions.  X-REGO is based on successively generating  reduced problems \eqref{eq: AREGO}, where the parameter $\mvec{p}$ is flexibly-chosen. Flexibility in choosing $\mvec{p}$ allows the user to calibrate the level of exploration in $\mathcal{X}$. 

Our central result is the proof of global convergence of X-REGO, that heavily relies on an estimate of the probability of the reduced problem \eqref{eq: AREGO} to be $\epsilon$-successful. By looking at the reduced problem through the prism of conic geometry, we have developed a new type of analysis to bound the probability of $\epsilon$-success of \eqref{eq: AREGO}. The bounds are expressed in terms of the so-called conic intrinsic volumes of circular cones which have exact formulae and thus are quantifiable. Using these formulae, we analysed the asymptotic behaviour of the bounds for large $D$. The analysis suggests that the success rate of \eqref{eq: AREGO} (as expected) decreases exponentially with growing $D$. Confirming our intuition, the analysis also shows that \eqref{eq: AREGO} has a high success rate for larger $d$ and smaller distances between the location where subspaces are embedded (i.e., the point $\mvec{p}$) and the location of a global minimizer $\mvec{x}^*$. This latter property of \eqref{eq: AREGO} for general Lipschitz continuous functions is remindful of the dependence of the success rates of \eqref{eq: AREGO} for functions with low effective dimensionality on the distance between $\mvec{p}_{\top}$ and $\mvec{x}_{\top}^*$, see \cite{Cartis2020b}. Furthermore, to understand the relative performance of \eqref{eq: AREGO}, we compared it with a uniform sampling technique. We looked at lower bounds for the probability of $\epsilon$-success of the two techniques and found that the lower bound $\tau(r_{\mvec{p}},d,D)$ for \eqref{eq: AREGO} is greater than the lower bound $\tau_{us}$ for uniform sampling if the distance $\| \mvec{x}^* - \mvec{p} \|$ is smaller than $0.48\sqrt{D}$ in the asymptotic regime ($D \rightarrow \infty$). In the asymptotic analysis, the embedding subspace $d$ was kept fixed. The analysis showed that in this regime $d$ has no significant effect on the relative performance of \eqref{eq: AREGO}. Future research may involve comparison of the performances of  \eqref{eq: AREGO} and uniform sampling in different asymptotic settings, for example, when $d = \beta D$ for some fixed constant~$\beta$.    

Our derivations are conceptual in nature, exploring new connections of global optimization to other areas such as conic integral geometry. As an illustration, in the second part of the paper, we used our analysis to obtain lower bounds --- that are independent of $D$ --- for the probability of $\epsilon$-success of  \eqref{eq: AREGO} for functions with low effective dimensionality in the case $d < d_e$. This analysis is exploited algorithmically and allows lifting the restriction of needing to know $d_e$ for random embeddings algorithms for functions with low effective dimensionality. We tested the effectiveness of X-REGO numerically using global and local KNITRO for solving the reduced problem on a set of benchmark global optimization problems modified to have low effective dimensionality. We proposed different variants of X-REGO each corresponding to a specific rule for choosing $\mvec{p}$'s and contrasted them against each other and against the `no-embedding' framework in which the solvers were applied to \eqref{eq: GO} directly with no use of subspace embeddings. The results of the experiments showed that the difference in performance between X-REGO and `no-embedding' becomes more prominent for larger $D$, in favour of X-REGO. The results further suggest that the effectiveness of X-REGO, just like of REGO in \cite{Cartis2020}, is solver-dependent. In our experiments, the best results were achieved by the local solver. In the future, we plan to investigate the performance of X-REGO when applied to general objectives and compare it with popular global optimization solvers.

\bibliographystyle{plainnat}
{\footnotesize
	\bibliography{bibliography}}

\appendix
\section{Technical definitions and results}
\subsection{Gaussian random matrices}
\begin{definition}[Gaussian matrix] \label{def: Gaussian_matrix}
	A Gaussian (random) matrix is a matrix whose each entry is an independent standard normal random variable. 
\end{definition}
Gaussian matrices have been well-studied with many results available at hand; here, we mention the following result that we use in the analysis; for a collection of results pertaining to Gaussian matrices and other related distributions refer to \cite{Gupta1999, Vershynin2018}.

\begin{theorem} \label{thm: orthog_inv_of_Gaussian_matrices} (see \cite[Theorem 2.3.10]{Gupta1999})
	Let $\mtx{A}$ be a $D \times d$ Gaussian random matrix. If $\mtx{U} \in \mathbb{R}^{D \times p}$, $D \geq p $, and $\mtx{V} \in \mathbb{R}^{d \times q}$, $d \geq q$, are orthonormal, then $\mtx{U}^T\mtx{A}\mtx{V}$ is a Gaussian random matrix. 
\end{theorem}

\section{Global convergence proof} \label{app:converg}
This section contains material already presented in \cite{Cartis2020b}, with minor changes to capture the fact that the lower bounds $\rho^k$ and $\tau^k$ are now variable with $k$ (or, in other words, the probability that the reduced problem \eqref{prob: AREGO_subproblem_re} is $\epsilon$-successful, as well as the probability that the solver finds a sufficiently accurate solution of the reduced problem, is changing with the dimension of the reduced problem $d^k$ in \Cref{alg:XREGO}). The following three lemmas are needed in our convergence proof. 
\begin{lemma}
	If \Cref{assump: prob_of_I_R>rho} holds, then 
	\begin{equation}\label{ineq: exp_IS_IR>tau*rho}
		\mathbb{E}[R^k S^k | \mathcal{F}^{k-1/2}]  \geq  \rho^k R^k, \quad {\rm for}\quad k\geq 1.
	\end{equation}
\end{lemma}
\begin{proof}
	\Cref{assump: prob_of_I_R>rho} implies
	$$\mathbb{E}[R^k S^k | \mathcal{F}^{k-1/2}] = R^k  \mathbb{E}[ S^k | \mathcal{F}^{k-1/2}] \geq  \rho^k R^k,$$
	where the equality follows from the fact that $R^k$ is $\mathcal{F}^{k-1/2}$-measurable and, thus, can be pulled out of the expectation (see \cite[Theorem 4.1.14]{Durrett2019}).
\end{proof}
A useful property is given next. 
\begin{lemma}\label{lemma: lim_of_prob_is_1}
	Let \Cref{assump: prob_of_I_R>rho} and \ref{ass:lowerbdRPK} hold. Then, for $K\geq 1$, we have
	$$ \prob\Big[ \bigcup_{k=1}^K \left\{ \{R^k = 1\} \cap \{ S^k = 1 \} \right\} \Big] \geq 1 - \Pi_{k = 1}^K(1-\tau^k \rho^k). $$
\end{lemma}
\begin{proof}
	We define an auxiliary random variable,
	$	 J^K :=  \mathds{1} \left(\bigcup_{k=1}^K \left\{  \{R^k = 1\} \cap \{ S^k = 1 \} \right\} \right).  $
	Note that $J^K = 1- \prod_{k=1}^{K} (1-R^k S^k)$. We have
	\begin{align*}
		\prob\Big[ \bigcup_{k=1}^K \left\{ \{R^k = 1\} \cap \{ S^k = 1 \} \right\} \Big] &= \mathbb{E}[J^K]  = 1 - \mathbb{E}\Big[\prod_{k=1}^{K} (1-R^k S^k)\Big] \\
		& \stackrel{(*)}{=} 1 - \mathbb{E}\Big[\mathbb{E}\Big[ \prod_{k=1}^{K} (1-R^k S^k) \Big| \mathcal{F}^{K-1/2} \Big]\Big] \\
		& \stackrel{(\circ)}{=} 1 - \mathbb{E}\Big[ \prod_{k=1}^{K-1} (1-R^k S^k) \cdot \mathbb{E}\big[ 1 - R^K S^K | \mathcal{F}^{K-1/2} \big]\Big] \\
		& \geq 1 - \mathbb{E}\Big[(1-\rho^K R^K)\cdot \prod_{k=1}^{K-1} (1-R^k S^k)  \Big] \\
		& \stackrel{(*)}{=} 1 - \mathbb{E}\Big[\mathbb{E}\Big[(1-\rho^K R^K)\cdot \prod_{k=1}^{K-1} (1-R^k S^k)   \Big| \mathcal{F}^{K-1}\Big]\Big] \\
		& \stackrel{(\circ)}{=} 1 - \mathbb{E}\Big[ \prod_{k=1}^{K-1} (1-R^k S^k) \cdot \mathbb{E}\big[ 1 - \rho^K R^K | \mathcal{F}^{K-1} \big]\Big] \\
		& \geq 1 - (1-\tau^K \rho^K) \cdot \mathbb{E}\Big[ \prod_{k=1}^{K-1} (1-R^kS^k)\Big],
	\end{align*}
	where 
	\begin{itemize}
		\item[-]
		$(*)$ follow from the tower property of conditional expectation (see (4.1.5) in \cite{Durrett2019}), 
		\item[-]
		$(\circ)$ is due to the fact that $R^1, \dots, R^{K-1}$ and $S^1,\dots,S^{K-1}$ are $\mathcal{F}^{K-1/2}$-  \,and $\mathcal{F}^{K-1}$-measurable (see Theorem 4.1.14 in \cite{Durrett2019}), 
		\item[-]
		the inequalities follow from \eqref{ineq: exp_IS_IR>tau*rho} and \eqref{ineq: cond_exp>tau}, respectively. 
	\end{itemize}
	We repeatedly expand the expectation of the product for $K-1$, $\ldots$, $1$, in exactly the same manner as above, to obtain the desired result.
\end{proof}

In the next lemma, we show that if \eqref{prob: AREGO_subproblem} is $(\epsilon-\lambda)$-successful and  is solved to accuracy $\lambda$ in objective value, then the solution $\mvec{x}^k$ must be inside $G_{\epsilon}$.

\begin{lemma}\label{lemma: if WcapG then_x in G_epsilon}
	Suppose \Cref{assump: prob_of_I_R>rho} and \ref{ass:lowerbdRPK} hold. Then,
	$$
	\{R^k = 1\}\cap \{S^k = 1\} \subseteq \{\mvec{x}^{k}  \in G_{\epsilon}\}.
	$$

\end{lemma}

\begin{proof}
	By \Cref{def: success_red_prob}, if \eqref{prob: AREGO_subproblem} is $(\epsilon-\lambda)$-successful, then there exists $\mvec{y}^k_{int} \in \mathbb{R}^{d^k}$ such that $\mtx{A}^k \mvec{y}^k_{int} + \mvec{p}^{k-1} \in \mathcal{X}$ and 
	\begin{equation} \label{ineq: asym_conv_ineq1}
		f(\mtx{A}^k\mvec{y}^k_{int} + \mvec{p}^{k-1}) \leq f^* + \epsilon - \lambda.
	\end{equation}
	Since $\mvec{y}^k_{int}$ is in the feasible set of \eqref{prob: AREGO_subproblem} and $ f^k_{min}$ is the global minimum of \eqref{prob: AREGO_subproblem}, we have
	\begin{equation} \label{ineq: asym_conv_ineq3}
		f^k_{min} \leq f(\mtx{A}^k\mvec{y}^k_{int} + \mvec{p}^{k-1}).
	\end{equation}
	Then, for $\mvec{x}^k$, \eqref{eq:approxf} gives the first inequality below,
	$$ f(\mvec{x}^k) \leq f^k_{min} + \lambda \leq f(\mtx{A}^k\mvec{y}^k_{int} + \mvec{p}^{k-1}) + \lambda \leq f^* + \epsilon, $$
	where the second and third inequalities follow from \eqref{ineq: asym_conv_ineq3} and \eqref{ineq: asym_conv_ineq1}, respectively. This shows that $\mvec{x}^k \in G_\epsilon$. 
\end{proof}

\subsection{Proof of \Cref{thm: glconv}.} \label{sec:app_proof_thm65}

	Lemma \ref{lemma: if WcapG then_x in G_epsilon} and the definition of $\mvec{x}^k_{opt}$ in \eqref{eq: xoptk} provide 
	$$ \{ R^k = 1 \} \cap \{ S^k = 1 \} \subseteq \{ \mvec{x}^k \in G_{\epsilon} \} \subseteq \{ \mvec{x}_{opt}^k \in G_{\epsilon} \} $$
	for $k = 1, 2,\dots, K$ and for any integer $K\geq 1$. Hence, 
	\begin{equation}\label{rel: cup_I is in cup_X}
		\bigcup_{k=1}^K \{ R^k = 1 \} \cap \{ S^k = 1 \} \subseteq \bigcup_{k=1}^K \{ \mvec{x}^k_{opt} \in G_{\epsilon} \}.
	\end{equation}
	Note that  the sequence $\{ f(\mvec{x}^1_{opt}), f(\mvec{x}^2_{opt}), \dots, f(\mvec{x}^K_{opt})\}$ is monotonically decreasing. Therefore, if $\mvec{x}^k_{opt} \in G_{\epsilon}$ for some $k \leq K$ then $\mvec{x}^i_{opt} \in G_{\epsilon}$ for all $i = k, \dots, K$; and so the sequence $(\{ \mvec{x}^k_{opt} \in G_{\epsilon} \})_{k = 1}^K$ is an increasing sequence of events. Hence,
	\begin{equation}\label{eq:cup_x_opt_in_G=x_opt_in_G}
		\bigcup_{k=1}^K \{ \mvec{x}^k_{opt} \in G_{\epsilon} \} = \{ \mvec{x}^K_{opt} \in G_{\epsilon} \}.
	\end{equation}
	From \eqref{rel: cup_I is in cup_X} and \eqref{eq:cup_x_opt_in_G=x_opt_in_G}, we have for all $K\geq 1$,
	\begin{equation}\label{eq:prob[x_opt_in_G_eps>1-(1-tr)^K]}
		\prob[\{ \mvec{x}^K_{opt} \in G_{\epsilon} \}]  \geq \prob\Big[\bigcup_{k=1}^K \{ R^k = 1 \} \cap \{ S^k = 1 \} \Big]  \geq 1 - \Pi_{k = 1}^K (1- \tau^k \rho^k),
	\end{equation}
	where the second inequality follows from \Cref{lemma: lim_of_prob_is_1}.
	Finally, passing to the limit with $K$ in \eqref{eq:prob[x_opt_in_G_eps>1-(1-tr)^K]}, we deduce
	\[	1 \geq \lim_{K \rightarrow \infty} \prob[\{ \mvec{x}^K_{opt} \in G_{\epsilon} \}] \geq \lim_{K \rightarrow \infty} \left[1 - \Pi_{k = 1}^K (1- \tau^k \rho^k)\right] \geq \lim_{K \rightarrow \infty} \left[1 -  (1- \tau_{lb} \rho_{lb})^K \right]= 1, \]
	with $\tau_{lb}$ and $\rho_{lb}$ defined in \Cref{ass:lowerbdRPK} and \Cref{assump: prob_of_I_R>rho}, respectively. Since $\tau_{lb} \rho_{lb} > 0$ by \Cref{assump: prob_of_I_R>rho} and \Cref{ass:lowerbdRPK}, we get the required result.
	Note that if 
	\begin{equation} \label{eq:1-(1-tau rho)^K>alpha}
		1-(1-\tau_{lb}  \rho_{lb})^k \geq \xi
	\end{equation}
	then \eqref{eq:prob[x_opt_in_G_eps>1-(1-tr)^K]} implies $\prob[ \mvec{x}^k_{opt} \in G_{\epsilon} ] \geq \xi$. Since \eqref{eq:1-(1-tau rho)^K>alpha} is equivalent to
	$ k \geq \displaystyle\frac{\log(1-\xi)}{\log(1- \tau_{lb} \rho_{lb})}$, 
	\eqref{eq:1-(1-tau rho)^K>alpha} holds for all $k\geq K_\xi$ since 
	$K_\xi \geq \displaystyle\frac{\log(1-\xi)}{\log(1-\tau_{lb} \rho_{lb})}$.

\section{Problem set} \label{app:prob}

\Cref{table: Test set} contains the name, domain and global minimum of the functions used to generate the high-dimensional test set. Similarly as in \cite{Cartis2020,Cartis2020b}, the problem set contains 18 problems taken from \cite{AMPGO, Ernesto2005, Bingham2013}. To generate this problem set, we transformed each of the 18 functions in \Cref{table: Test set} into a high-dimensional function with low-effective dimension, by adapting the method proposed by Wang et al. \cite{Wang2016}. Let $\bar{g}(\bar{\mvec{x}})$ be any function from \Cref{table: Test set}, with dimension $d_e$ and let the given domain be scaled to $[-1, 1]^{d_e}$. We create a $D$-dimensional function $g(\mvec{x})$ by adding $D-d_e$ fake dimensions to $\bar{g}(\bar{\mvec{x}})$, $ g(\mvec{x}) = \bar{g}(\bar{\mvec{x}}) + 0\cdot x_{d_e+1} + 0 \cdot x_{d_e+2} + \cdots + 0\cdot x_{D}$. We further rotate the function by applying a random orthogonal matrix $\mtx{Q}$ to $\mvec{x}$ to obtain a nontrivial constant subspace. The final form of the function we test is
\begin{equation}\label{eq: f=g(Qx)}
	f(\mvec{x}) = g(\mtx{Q}\mvec{x}).
\end{equation}
Note that the first $d_e$ rows of $\mtx{Q}$ now span the effective subspace $\mathcal{T}$ of $f(\mvec{x})$.

For each problem in the test set, we generate three functions $f$ according to \eqref{eq: f=g(Qx)}, one for each $D = 10$, $100$, $1000$. Note that the range of effective dimension covered by our test set is slightly larger than in \cite{Cartis2020, Cartis2020b}, to better assess the ability of the algorithm to learn $d_e$.

\begin{table}[!ht]
	\centering
	\caption{The problem set listed in alphabetical order.}
	\label{table: Test set}
	\begin{tabular} {|L{4.1cm} | C{3cm} | C{3.3cm} | }
		\hline
		Function & Domain  & Global minima  \\ \hline
		1) Beale \cite{Ernesto2005}    & $\mvec{x} \in [-4.5,4.5]^2$ &  $g(\mvec{x}^*) = 0$ \\ \hline
		2) Branin \cite{Ernesto2005}    & \pbox{20cm}{$x_1 \in [-5,10]$ \\ $x_2 \in [0, 15]$} 
		& $g(\mvec{x}^*) = 0.397887$ \\ \hline
		
		3) Brent \cite{AMPGO}  & $\mvec{x} \in [-10,10]^2$ & $g(\mvec{x}^*) = 0$  \\ \hline
		
		
		5) Easom  \cite{Ernesto2005} & $\mvec{x}\in [-100,100]^2$   & $g(\mvec{x}^*) = -1$ \\ \hline
		
		6) Goldstein-Price \cite{Ernesto2005} & $\mvec{x} \in [-2,2]^2 $  & $g(\mvec{x}^*) = 3$ \\ \hline
		
		7) Hartmann 3 \cite{Ernesto2005} & $\mvec{x} \in [0,1]^3$ & $g(\mvec{x}^*) = -3.86278$
		\\ \hline
		
		8) Hartmann 6 \cite{Ernesto2005}  & $\mvec{x} \in [0,1]^6$  & $g(\mvec{x}^*) = -3.32237$
		\\ \hline
		
		9) Levy \cite{Bingham2013}  & $\mvec{x} \in [-10,10]^6$ & $g(\mvec{x}^*) = 0$  \\ \hline
		
		10) Perm 4, 0.5 \cite{Bingham2013} & $\mvec{x} \in [-4,4]^4$ & $g(\mvec{x}^*) = 0$  \\ \hline
		
		11) Rosenbrock \cite{Bingham2013}   & $\mvec{x} \in [-5,10]^7$ & $g(\mvec{x}^*) = 0$  \\ \hline
		
		12) Shekel $5$ \cite{Bingham2013}  & $\mvec{x} \in [0,10]^4$ & $ g(\mvec{x}^*) = -10.1532$
		\\ \hline
		
		13) Shekel $7$ \cite{Bingham2013}  & $\mvec{x} \in [0,10]^4$ & $g(\mvec{x}^*) = -10.4029$
		\\ \hline
		
		14) Shekel $10$ \cite{Bingham2013} & $\mvec{x} \in [0,10]^4$ & $g(\mvec{x}^*) = -10.5364$
		\\ \hline
		
		15) Shubert \cite{Bingham2013}  & $\mvec{x} \in [-10,10]^2$ & $g(\mvec{x}^*) = -186.7309$ \\ \hline
		
		16) Six-hump camel \cite{Bingham2013} & \pbox{20cm}{$x_1 \in [-3,3]$ \\ $x_2 \in [-2,2]$} & $g(\mvec{x}^*) = -1.0316$ \\ \hline
		
		17) Styblinski-Tang \cite{Bingham2013}   & $\mvec{x} \in [-5,5]^8$ & $g(\mvec{x}^*) = -313.329$  \\ \hline
		
		18) Trid  \cite{Bingham2013} & $\mvec{x} \in [-25,25]^5$ & $g(\mvec{x}^*) = -30$ \\ \hline
		
		19) Zettl \cite{Ernesto2005}   & $\mvec{x} \in [-5,5]^2$ & $g(\mvec{x}^*) = -0.00379$  \\ \hline
	\end{tabular}
\end{table}

\end{document}